\documentclass[10pt,a4paper]{article}
\usepackage[utf8]{inputenc}
\usepackage{amsmath}
\usepackage{hyperref}
\usepackage{cleveref}
\usepackage{amsfonts}
\usepackage{amsthm}
\usepackage{amssymb}
\usepackage{graphicx}
\usepackage[left=2cm,right=2cm,top=2cm,bottom=2cm]{geometry}
\usepackage{lipsum}
\usepackage{float}
\numberwithin{equation}{section}
\newtheorem{theorem}{Theorem}[section]
\crefname{theorem}{Theorem}{Theorems}
\newtheorem{lemma}{Lemma}[section]
\crefname{lemma}{Lemma}{Lemmas}
\newtheorem{prop}{Proposition}[section]
\crefname{prop}{Proposition}{Propositions}
\newtheorem{remark}{Remark}[section]
\crefname{remark}{Remark}{Remarks}
\crefname{section}{Section}{Sections}
\crefname{subsection}{Subsection}{Subsections}
\crefname{appendix}{Appendix}{Appendices}

\DeclareFontFamily{OT1}{rsfs}{}
\DeclareFontShape{OT1}{rsfs}{m}{n}{ <-7> rsfs5 <7-10> rsfs7 <10-> rsfs10}{}
\DeclareMathAlphabet{\mycal}{OT1}{rsfs}{m}{n}

\newcommand{\mres}
{\mathbin{\vrule height 1.6ex depth 0pt width 0.13ex\vrule height 0.13ex depth 0pt width 1.3ex}}
\usepackage{array}
\usepackage{wrapfig}
\usepackage{subcaption}
\usepackage{multirow}
\usepackage{esint}

\title{Cubic microlattices embedded in nematic liquid crystals:\\ a Landau-de Gennes study}
\author{Razvan-Dumitru Ceuca\thanks{BCAM,  Basque  Center  for  Applied  Mathematics,  Mazarredo  14,  48009  Bilbao,  Bizkaia,  Spain.
(rceuca@bcamath.org)} \thanks{UPV/EHU, Universidad del País Vasco/Euskal Herriko Unibertsitatea, Barrio Sarriena s/n, 48940 Leioa, Bizkaia, Spain} \thanks{UAIC, Alexandru Ioan Cuza University of Iasi, Bulevardul Carol I, Nr. 11, 700506 Iasi, Romania}}
\date{\today}
\begin{document}

\maketitle

\begin{abstract}
We consider a Landau-de Gennes model for a connected cubic lattice scaffold in a nematic host, in a dilute regime. We analyse the homogenised limit for both cases in which the lattice of embedded particles presents or not cubic symmetry and then we compute the free effective energy of the composite material. 

In the cubic symmetry case, we impose different types of surface anchoring energy densities, such as quartic, Rapini-Papoular or more general versions, and, in this case, we show that we can tune any coefficient from the corresponding bulk potential, especially the phase transition temperature. 

In the case with loss of cubic symmetry, we prove similar results in which the effective free energy functional has now an additional term, which describes a change in the preferred alignment of the liquid crystal particles inside the domain. 

Moreover, we compute the rate of convergence for how fast the surface energies converge to the \linebreak homogenised one and also for how fast the minimisers of the free energies tend to the minimiser of the homogenised free energy. 
\end{abstract}

\section{Introduction}

We consider a cubic microlattice scaffold constructed of connected particles of micrometer scale, within a nematic liquid crystal. In this article, we treat the particles of the cubic microlattice as being inclusions from the mathematical point of view, while they might be interpreted as colloids from the physical point of view, even though they do not possess all of their properties. The cubic microlattice scaffold is also called a bicontinuous porous solid matrix (BPSM) in the physics literature (for example, see \cite{BPSM1}, \cite{nematiccage} or \cite{BPSM2}). By cubic microlattice scaffold we understand a connected family of parallelipipeds or cubes of different sizes, placed in a periodic fashion, as in Figure \ref{fig:nematiccage1}, where only the embedded particles have been shown. For simplicity, we might refer to this object as being a scaffold or a cubic microlattice. This type of scaffold is usually obtained using the \emph{two-photon polymerization} (TPP or 2PP) process, which represents a technique of 3D-manufacturing structures and which can generate stand-alone objects. An overview of the field of TPP processes can be found in \cite{2PP}. There are numerous experiments, theory and computer simulations regarding embedding microparticles into nematic liquid crystals (for example, see \cite{Ravnik1}, \cite{Ravnik2} and \cite{Ravnik3}).

\begin{figure}[h]
     \centering
     \includegraphics[width=0.5\textwidth]{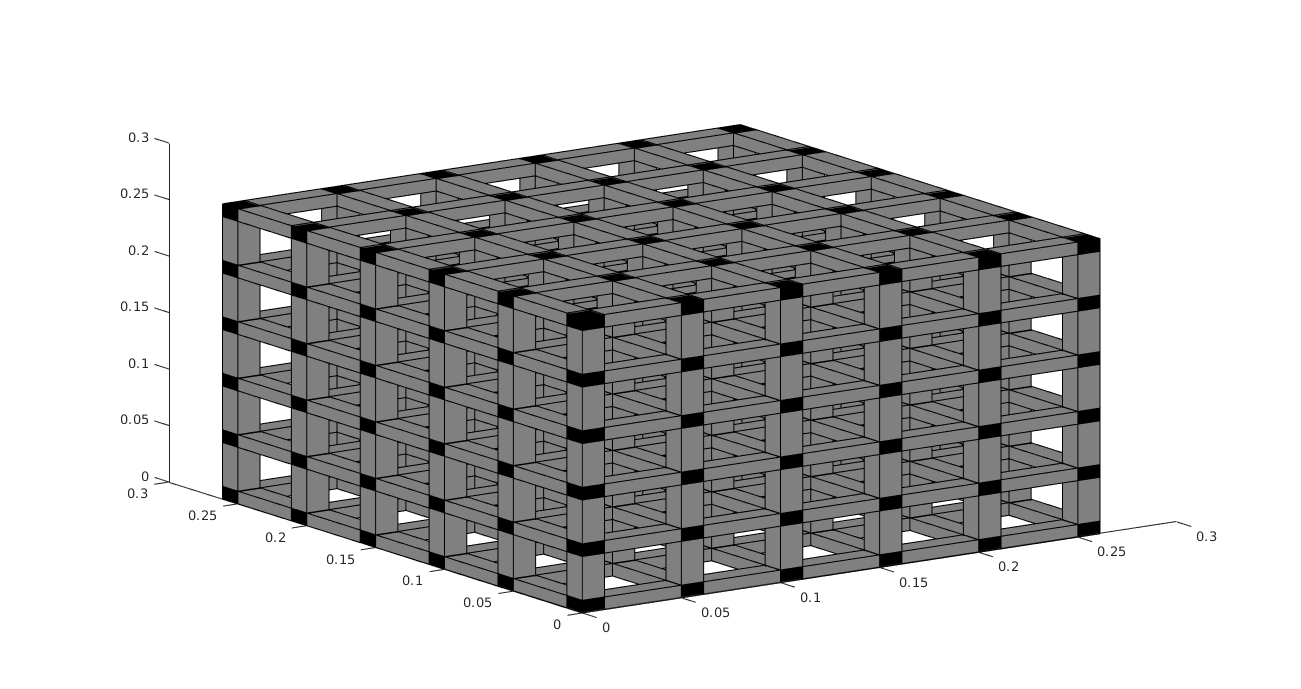}
     \caption{Example of a cubic microlattice.}
     \label{fig:nematiccage1}
\end{figure}

The system bears mathematical similarities to that of colloids embedded into nematic liquid crystals. The mathematical studies of nematic colloids (the mixture of colloidal particles embedded into nematic liquid crystals) are split into two broad categories:
\begin{itemize}
\item[•] one is dealing with the effect produced by a small number of particles in this mixture, with a focus on the defect patterns that arise in the alignment of the nematic particles induced by the interaction at the boundary of the colloid between the two combined materials (see, for example, \cite{Alama1, Alama2, Canevari1, Canevari2, Canevari3, Canevari4, Wang});
\item[•] the other one treats the study of the collective effects, that is the homogenisation process (see, for example, \cite{Bennett, Berlyland, Calderer, CanevariZarnescu1} and \cite{CanevariZarnescu2}).
\end{itemize}

This work continues within the second direction, that is studying the homogenised material, and it is built on the work from \cite{CanevariZarnescu1} and \cite{CanevariZarnescu2}, which was also based on \cite{Bennett, Berlyland, Calderer}. In \cite{CanevariZarnescu1} and \cite{CanevariZarnescu2}, the \textit{inclusion} is considered to be the union of some disconnected particles, obtained from different or identical model particles, in such a way that the distance between the particles is considerable larger than the size of them, which is called the dilute regime. Also, in this regime, the volume fraction of colloids tends to zero. 

In this article, we are going to consider the case of a cubic microlattice scaffold, as shown in Figure \ref{fig:nematiccage1}. The idea of using such a particular geometry for the scaffold comes from the work done in \cite{nematiccage}. At the same time, this geometric configuration is more relevant from the physical point of view, since in \cite{CanevariZarnescu1} and in \cite{CanevariZarnescu2} one cannot position \textit{a priori} the colloidal particles in a periodic fashion. Here the periodicity is automatically generated by the structure of the cubic microlattice. We construct two types of scaffolds: one with identical cubes centered in a periodic 3D lattice of points, cubes which are inter-connected by parallelipipeds, and one where we replace the cube with a parallelipiped with three different length sides. If by cubic symmetry we understand the family of rotations that leave a cube invariant, then the first case is when the scaffold particles have cubic symmetry and the second one is with the loss of this type of symmetry. The main new aspects of this work are:
\begin{itemize}
\item[•] the set of all the inclusions is now connected;
\item[•] the model particle that we use (that is, a parallelipiped or a cube) grants us the possibility to compute the surface contribution for arbitrarily high order terms in the surface energy density - hence, a generalisation has been done for higher order polynomials in the bulk energy potential that admit at least one local minimiser (see \cref{th:gen});
\item[•] in the case where the cubic symmetry is lost, we obtain a new term into the homogenised limit that can be seen as a change in the preferred alignment of the liquid crystal particles inside the domain (see \cref{th:asym});
\item[•] we obtain a rate of convergence for how fast the surface energies converge to the homogenised one (more details in \cref{prop:rate_of_conv}); in \cref{remark:rate_of_convergence}, we also obtain a rate of convergence for how fast the sequence of minimisers of the free energies tend to a minimiser of the homogenised free energy; 
\end{itemize}

Liquid crystal materials, which typically consist of either rod-like or disc-like molecules, can achieve a state of matter which has properties between those of conventional liquids and those of solid crystals. The liquid crystal state of matter is one where there exists a long range orientational order for the molecules. In order to quantify the local preferred alignment of the rod-like molecules, we use the theory of Q-tensors (for more details, see \cite{Mottram}). A background of the field of liquid crystal materials can be found in \cite{deGennes}.
 
Let $\Omega\subset\mathbb{R}^3$ be an open and bounded domain from $\mathbb{R}^3$. For every $\varepsilon>0$, we construct a cubic microlattice $\mathcal{N}_\varepsilon$ inside of $\Omega$, such that, as $\varepsilon\rightarrow 0$, the volume of the scaffold tends to 0. More details regarding the construction of the cubic microlattice can be found in \cref{section:assumptions} and in \cref{section:constructing_lattice}. 

Let $\Omega_{\varepsilon}=\Omega\setminus\mathcal{N}_{\varepsilon}$, which represents the space where only liquid crystal particles can be found. We use functions $Q:\Omega_{\varepsilon}\rightarrow\mathcal{S}_0$ to describe the orientation of the liquid crystal particles, where: $$\mathcal{S}_{0}=\{Q\in\mathbb{R}^{3\times 3}:\;Q=Q^T,\;\text{tr}(Q)=0\},$$ is denoted as the set of $Q$-tensors. In the space $\mathcal{S}_0$, if we define $|Q|=(\text{tr}(Q^2)\big)^{1/2}$, for any $Q\in\mathcal{S}_0$, we can see that $\mathcal{S}_0$ is a normed linear space and the so-called Frobenius norm is induced by the scalar product $Q\cdot P=\text{tr}(Q\cdot P)$.

We consider the following Landau-De Gennes free energy functional: 
\begin{equation}\label{eq:variat1}
\mathcal{F}_{\varepsilon}[Q]:=\int_{\Omega_{\varepsilon}} \big(f_e(\nabla Q)+f_b(Q)\big)\text{d}x+\dfrac{\varepsilon^{3}}{\varepsilon^{\alpha}(\varepsilon-\varepsilon^{\alpha})}\int_{\partial\mathcal{N}_{\varepsilon}}f_s(Q,\nu)\text{d}\sigma,
\end{equation} 
where $f_e$ represents the \textit{elastic energy}, $f_b$ the \textit{bulk energy}, $f_s$ the \textit{surface density energy}, $\alpha$ is a real parameter and $\partial\mathcal{N}_{\varepsilon}$ the \textit{surface of the scaffold}. The parameter $\alpha$ is chosen such that the term $\varepsilon^{\alpha}(\varepsilon-\varepsilon^{\alpha})$ is proportional with the surface terms from $\partial\mathcal{N}_{\varepsilon}$ in order that, in the limit $\varepsilon\rightarrow 0$, it balances the effect given by the areas of the surface terms. All of these quantities are described and detailed more in \cref{section:assumptions}.

The \textit{elastic energy}, also called the distortion energy, penalises the distortion of $Q$ in the space and, in the Landau-de Gennes theory, it is usually considered to be a positive definite quadratic form in $\nabla Q$.

The \textit{bulk energy} in our case consists only of the thermotropic energy, which is a potential function that describes the preferred state of the liquid crystal, that is either uniaxial, biaxial or isotropic\footnote{The isotropic case corresponds to the case in which $Q=0$. The uniaxial case corresponds to the one in which two of the eigenvalues of $Q$ are equal and the third one has a different value. The biaxial case corresponds to the case in which all the eigenvalues have different values.}. For large values of the temperature, the minimum of this energy is obtained in the isotropic case, that is $Q=0$, and for small values, the minimum set is a connected set of the form $s(\nu\otimes \nu-\mathbb{I}_3/3)$, with $\nu\in \mathbb{S}^2$ and $\mathbb{I}_3$ the identity $3\times 3$ matrix, and this is a connected set diffeomorphic with the real projective plane. The simplest form that we can take for the \textit{bulk energy} in our case is the quartic expansion:
\begin{align*}
f_b(Q)=a\;\text{tr}(Q^2)-b\;\text{tr}(Q^3)+c\;\text{tr}(Q^2)^2,
\end{align*}
where the coefficient $a$ depends on the temperature of the liquid crystal and $b$ and $c$ depend on the properties of the liquid crystal material, with $b,c>0$.

The \textit{surface energy} describes the interaction between the liquid crystal material and the boundary of the scaffold. We assume, for simplicity, that it depends only on $Q$ and on $\nu$, where $\nu$ is the outward normal at the boundary of the cubic microlattice. One of the most common forms for the \textit{surface energy} is the Rapini-Papoular energy:
\begin{align}\label{eq:intro_Rapini-Papoular}
f_s(Q,\nu)=W\;\text{tr}\big(Q-s_+\big(\nu\otimes\nu-\mathbb{I}_3/3\big)\big)^2,
\end{align}
where $W$ is a coefficient measuring the strength of the anchoring, $s_+$ is measuring the deviation from the homeotropic (perpendicular) anchoring to the boundary and $\mathbb{I}_3$ is the $3\times 3$ identity matrix. 

We are interested in studying the behaviour of the whole material when $\varepsilon\rightarrow 0$. We will show that in our dilute regime we obtain for the homogenised material an energy functional of the following form
\begin{equation*}
\mathcal{F}_0[Q]:=\int_{\Omega}\big(f_e(\nabla Q)+f_b(Q)+f_{hom}(Q)\big)\text{d}x,
\end{equation*}
where $f_{hom}$ is defined in \eqref{defn:f_hom} and in \eqref{defn:f_hom_sym}, depending on the choice of $f_b$. Unlike in the standard homogenisation studies, our focus will be on \textit{a priori} designing the $f_{hom}$, in terms of the available parameters of the system.

The article is organised in the following manner:
\begin{itemize}
\item[•] in \cref{section:assumptions_and_main_results} we present the technical assumptions that we have chosen and the main results of this article;
\item[•] in \cref{section:prop_F_eps} we present the study of the properties of the functional $\mathcal{F}_{\varepsilon}$ for a fixed value for $\varepsilon>0$;
\item[•] in \cref{section:conv_local_min} we glue together the properties studied in the previous section and analyse the $\Gamma$-limit of $\mathcal{F}_{\varepsilon}$ as $\varepsilon\rightarrow 0$ and we prove the main theorems stated in \cref{section:assumptions_and_main_results};
\item[•] in \cref{section:rate_of_conv} we analyse the rate of convergence of the sequence of surface energies to the homogenised surface functional, where the main result is \cref{prop:rate_of_conv}, but we also analyse the rate of convergence of the sequence of minimisers of the free energies to a minimiser of the homogenised free energy (see \cref{remark:rate_of_convergence})
\end{itemize}
\noindent and
\begin{itemize}
\item[•] in \cref{section:appendix} we prove various results, the most important of which is the proposition regarding the explicit extension function that we use in \cref{subsection:extension}.
\end{itemize} 

\section{Tehnical assumptions and main results}\label{section:assumptions_and_main_results}

\subsection{Assumptions, notations and main result}\label{section:assumptions}

Let $\Omega\subset\mathbb{R}^3$ be a bounded, Lipschitz domain, that models the ambient liquid crystal, and let $\mathcal{C}\subset\mathbb{R}^3$ be the model particle for the cubic microlattice. Since $\Omega$ is bounded in $\mathbb{R}^3$, then:
\begin{align}\label{defn:L_0_l_0_h_0}
\exists\; L_0,\;l_0,\;h_0\in[0,+\infty)\;\text{such that}\;\overline{\Omega}\subseteq[-L_0,L_0]\times[-l_0,l_0]\times[-h_0,h_0].
\end{align}

In Figure \ref{fig:nematiccage}, we illustrate some examples of cubic microlattices, where the ``connecting" boxes (which can be seen better in Figure \ref{fig:nematiccage1} as being the black cubes) are cubes of size $\varepsilon^{\alpha}$, with $\alpha=1.4999$ \footnote{We choose $\alpha$ close to the value 3/2 in order to make the difference between the lengths of the sides of the black cubes and the gray parallelipipeds from Figure \ref{fig:nematiccage1} more visible, for relatively ``large" values of $\varepsilon$ ($0.01$, $0.05$ or $0.001$).} and $\varepsilon$ has a positive value close to 0, since we desire to work in the dilute regime. The distance between two closest black cubes is equal to $\varepsilon$, therefore the length of the black cubes is significantly smaller than the distance between them, by using the exponent $\alpha$. \footnote{The reason why we represent the lattice only in the box $[0,l]^3$, with $l=5\varepsilon$, is that if we keep the same $l$ and shrink $\varepsilon$, then the number of boxes appearing in the image would be significantly larger, hence, as we make $\varepsilon$ smaller, we also zoom in to have a better picture of what is happening for small values of $\varepsilon$.} For Figure \ref{fig:nematiccage1}, $\varepsilon=0.05$, $\alpha=1.4999$ and $l=0.25$, so we have the same ratio between $\varepsilon$ and $l$.

\begin{figure}[h]
     \centering
     \begin{subfigure}[b]{0.48\textwidth}
         \centering
         \includegraphics[width=\textwidth]{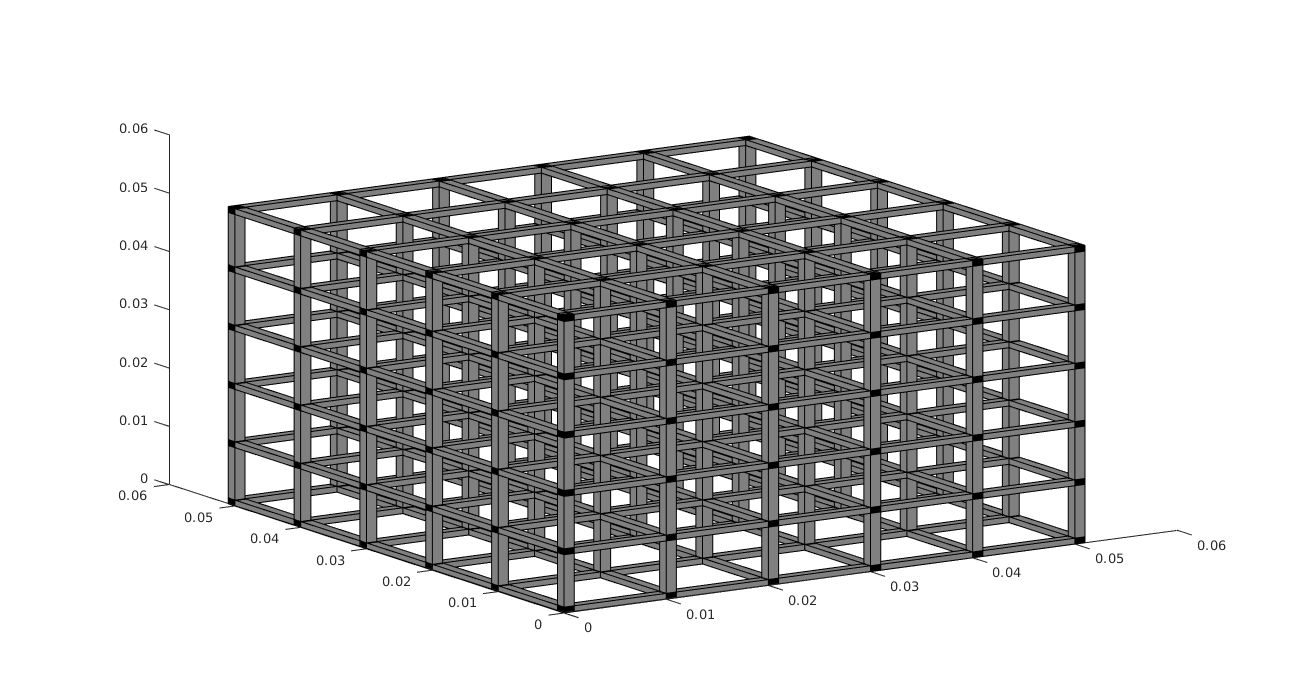}
         \caption{$\varepsilon=0.01$ and $l=0.05$;}
     \end{subfigure}
     \hfill
     \begin{subfigure}[b]{0.48\textwidth}
         \centering
         \includegraphics[width=\textwidth]{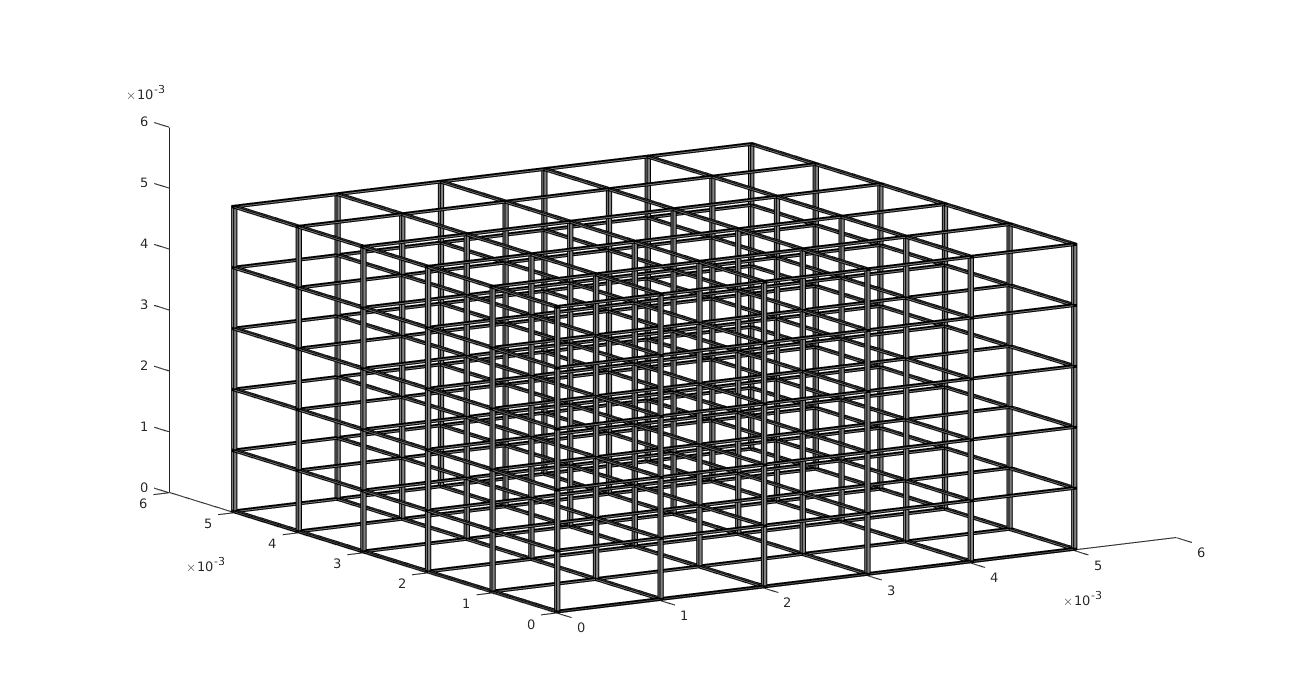}
         \caption{$\varepsilon=0.001$ and $l=0.005$.}
     \end{subfigure} 
     \caption{Cubic microlattices constructed in the box $[0,l]^3$ with $\alpha=1.4999$.}
        \label{fig:nematiccage}
\end{figure}

In order to construct such a scaffold, we use as a model particle the cube:
\begin{align}\label{defn:initial_cube}
\mathcal{C}=\bigg[-\dfrac{1}{2},\dfrac{1}{2}\bigg]^3.
\end{align}
We denote by $\partial\mathcal{C}$ the surface of the cube $\mathcal{C}$, which we also write it as:
\begin{align}\label{defn:C_x_C_y_C_z}
\partial\mathcal{C}=\mathcal{C}^x\;\cup\;\mathcal{C}^y\;\cup\;\mathcal{C}^z,
\end{align}
where $\mathcal{C}^x$ is the union of the two faces of the cube that are perpendicular to the $x$ direction and in the same way are defined $\mathcal{C}^y$ and $\mathcal{C}^z$.

Then, for a fixed value of $\varepsilon>0$ and an $\varepsilon$-independent positive constant $\alpha$, we define
\begin{equation}\label{defn:initial_cube_alpha}
\mathcal{C}^{\alpha}=\bigg[-\dfrac{\varepsilon^{\alpha}}{2p},+\dfrac{\varepsilon^{\alpha}}{2p}\bigg]\times\bigg[-\dfrac{\varepsilon^{\alpha}}{2q},+\dfrac{\varepsilon^{\alpha}}{2q}\bigg]\times\bigg[-\dfrac{\varepsilon^{\alpha}}{2r},+\dfrac{\varepsilon^{\alpha}}{2r}\bigg],
\end{equation}
with $p,\;q,\;r\in[1,+\infty)$. We call the scaffold symmetric whenever $p=q=r$. In Figures \ref{fig:nematiccage1} and \ref{fig:nematiccage}, we have $p=q=r=1$.

We construct now the lattice
\begin{equation}\label{defn:points1}
\mathcal{X}_{\varepsilon}=\{x\in\Omega\;:\;x=(x_1,x_2,x_3),\;\text{dist}(x,\partial\Omega)\geq\varepsilon\;\text{and}\;x_k/\varepsilon\in\mathbb{Z}\;\text{for}\;k\in\overline{1,3}\},
\end{equation}
which we rewrite it as:
\begin{align}\label{defn:points2}
\mathcal{X}_{\varepsilon}=\{x_{\varepsilon}^i\;:\;i\in\overline{1,N_{\varepsilon}}\},\;\text{where}\;N_{\varepsilon}=\text{card}\big(\mathcal{X}_{\varepsilon}\big).
\end{align}

Hence, the first part of the scaffold is the family of parallelipipeds 
\begin{equation}\label{defn:ceps}
\mathcal{C}_{\varepsilon}=\bigcup_{i=1}^{N_{\varepsilon}}\mathcal{C}_{\varepsilon}^i,\;\text{where}\;\mathcal{C}_{\varepsilon}^i=x^i_{\varepsilon}+\mathcal{C}^{\alpha},\;\text{for every}\;i\in\overline{1,N_{\varepsilon}},
\end{equation}
which represents the union of all black parallelipipeds from Figure \ref{fig:nematiccage1}.

We add now the lattice
\begin{align}\label{defn:Y_eps}
\mathcal{Y}_{\varepsilon}&:=\bigg\{y_{\varepsilon}\in\Omega\;:\;\exists i,j\in\overline{1,N_{\varepsilon}}\;\text{such that}\;\big|x_{\varepsilon}^i-x_{\varepsilon}^j\big|=\varepsilon\;\text{and}\;y_{\varepsilon}=\dfrac{1}{2}\big(x_{\varepsilon}^i+x_{\varepsilon}^j\big)\bigg\}\notag\\
&\hspace{-16mm}\text{and let}\;M_{\varepsilon}=\text{card}\big(\mathcal{Y}_{\varepsilon}\big).
\end{align}

The lattice $\mathcal{Y}_{\varepsilon}$ helps us construct the gray parallelipipeds from Figure \ref{fig:nematiccage1}, that is, the ``connecting boxes". We split this lattice into three parts, since the gray parallelipipeds are elongated into three different directions, granted by the axes of the Cartesian coordinate system in $\mathbb{R}^3$. We denote by $\mathcal{P}_{\varepsilon}$ the union of all of these parallelipipeds. More details regarding the construction of these objects can be found in \cref{section:constructing_lattice}.

 Let $\mathcal{N}_{\varepsilon}=\mathcal{C}_{\varepsilon}\cup\mathcal{P}_{\varepsilon}$ be the entire scaffold and $\partial\mathcal{N}_{\varepsilon}$ its surface. 

Regarding the \textit{elastic energy} that we use, we observe that having a fixed distortion of $Q$, the liquid crystal material should remain unchanged under translations and rotations. We consider the following form for the \textit{elastic energy}:
\begin{align*}
f_e(\nabla Q)&:=\sum_{i,j,k\in\{1,2,3\}}\bigg[\dfrac{L_1}{2}\bigg(\dfrac{\partial Q_{ij}}{\partial x_k}\bigg)^2+\dfrac{L_2}{2}\dfrac{\partial Q_{ij}}{\partial x_j}\dfrac{\partial Q_{ik}}{\partial x_k}+\dfrac{L_3}{2}\dfrac{\partial Q_{ik}}{\partial x_j}\dfrac{\partial Q_{ij}}{\partial x_k}\bigg],
\end{align*}
where $Q_{ij}$ is the $(ij)^{th}$ component of $Q$, $(x_1,x_2,x_3)$ represents the usual cartesian coordinates and $e_{ijk}$ represents the Levi-Civita symbol.

Regarding the \textit{surface energy densities}, we use different forms. One is the Rapini-Papoular \textit{surface energy density}, presented in \eqref{eq:intro_Rapini-Papoular}, which has a single minimum at the point in which the dependent variables take value granted by the surface treatment, which in this case represents the perpendicular \textit{alignment}. 

Another form for the \textit{surface energy density} is represented by the planar degenerate anchoring case:
\begin{align*}
f_s(Q,\nu)=k_a\;(\nu\cdot Q^2\nu)+k_b\;(\nu\cdot Q\nu)(\nu\cdot Q^2\nu)+k_c\;(\nu\cdot Q^2\nu)^2+a'\;\text{tr}(Q^2)+\dfrac{2b'}{3}\;\text{tr}(Q^3)+\dfrac{c'}{2}\;\text{tr}(Q^2)^2,
\end{align*}
where $k_a$, $k_b$, $k_c$, $a'$, $b'$ and $c'$ are constants, in which the preferred alignment for the liquid crystal material is to lie parallel to the boundary of the scaffold.

In order to describe the \textit{surface energy}, we need a better description of $\partial\mathcal{N}_{\varepsilon}$, therefore we analyse what faces from every parallelipiped constructed are in contact with the liquid crystal. More precisely, the liquid crystal is in contact with the scaffold:
\begin{itemize}
\item[•] on only four of the six faces of the parallelipipeds centered in points from $\mathcal{Y}_{\varepsilon}$, that is, on every $\mathcal{T}_x^k$, $\mathcal{T}_y^l$ and $\mathcal{T}_z^m$, defined in \eqref{defn:T_x_k}, \eqref{defn:T_y_l} and \eqref{defn:T_z_m};
\item[•] only on the edges of the parallelipipeds centered in some of the points from $\mathcal{X}_{\varepsilon}$, parallelipipeds which are in the ``interior" of the scaffold, meaning that they are not ``visible" - in this case, the interaction is neglected (the black parallelipipeds from Figure \ref{fig:nematiccage1} which have the only role of connecting the six adjacent gray parallelipipeds and they are not visible from that point of view, since they are ``inside" of the scaffold); let
\begin{align}\label{defn:N_eps_1}
N_{\varepsilon,1}\;=\text{ the total number of parallelipipeds from this case;}
\end{align}
\item[•] on at most five of the six faces of the parallelipipeds centered in some of the points from $\mathcal{X}_{\varepsilon}$, parallelipipeds which are at the ``outer" boundary of the scaffold (which are the closest to $\partial\Omega$) - in this case, we prove that the interaction is neglectable (the black parallelipipeds from Figure \ref{fig:nematiccage1} that are visible); let
\begin{align}\label{defn:N_eps_2}
N_{\varepsilon,2}\;=\text{ the total number of parallelipipeds from this case;}
\end{align}
and let
\begin{align}\label{defn:S_i}
\mathcal{S}^i\;=&\text{ the union of all the rectangles (at most five in this case) that}\notag\\
&\text{\hspace{4mm} are in contact with the liquid crystal material,}
\end{align}
for any $i\in\overline{1,N_{\varepsilon,2}}$.
\end{itemize}

From relations \eqref{defn:N_eps_1} and \eqref{defn:N_eps_2}, we have $N_{\varepsilon}=N_{\varepsilon,1}+N_{\varepsilon,2}$. Using \eqref{defn:T_x_k}, \eqref{defn:T_y_l}, \eqref{defn:T_z_m} and \eqref{defn:S_i}, we can write $\partial\mathcal{N}_{\varepsilon}=\partial\mathcal{N}_{\varepsilon}^{\mathcal{S}}\cup\partial\mathcal{N}_{\varepsilon}^{\mathcal{T}}$, where:

\begin{equation}\label{defn:surface}
\partial\mathcal{N}_{\varepsilon}^{\mathcal{S}}=\bigg(\bigcup_{i=1}^{N_{\varepsilon,2}}\mathcal{S}^i\bigg)\;\;\text{and}\;\;\partial\mathcal{N}_{\varepsilon}^{\mathcal{T}}=\bigg(\bigcup_{k=1}^{X_{\varepsilon}}\mathcal{T}_x^k\bigg)\cup\bigg(\bigcup_{l=1}^{Y_{\varepsilon}}\mathcal{T}_y^l\bigg)\cup\bigg(\bigcup_{m=1}^{Z_{\varepsilon}}\mathcal{T}_z^m\bigg).
\end{equation}

Let $J_{\varepsilon}[Q]$ be the \textit{surface energy} term from \eqref{eq:variat1} and let us split this term into two parts:

\begin{align}\label{defn:Jeps}
J_{\varepsilon}[Q]&=J_{\varepsilon}^{\mathcal{S}}[Q]+J_{\varepsilon}^{\mathcal{T}}[Q],
\end{align}
where
\begin{align}\label{defn:Jeps_S}
J_{\varepsilon}^{\mathcal{S}}[Q]&=\dfrac{\varepsilon^{3-\alpha}}{\varepsilon-\varepsilon^{\alpha}}\displaystyle{\int_{\partial\mathcal{N}_{\varepsilon}^{\mathcal{S}}}f_s(Q,\nu)\text{d}\sigma}=\dfrac{\varepsilon^{3-\alpha}}{\varepsilon-\varepsilon^{\alpha}}\sum_{i=1}^{N_{\varepsilon,2}}\int_{\mathcal{S}^i}f_s(Q,\nu)\text{d}\sigma,
\end{align}
using \eqref{defn:surface}, and
\begin{align}
J_{\varepsilon}^{\mathcal{T}}[Q]&=\dfrac{\varepsilon^{3-\alpha}}{\varepsilon-\varepsilon^{\alpha}}\displaystyle{\int_{\partial\mathcal{N}_{\varepsilon}^{\mathcal{T}}}f_s(Q,\nu)\text{d}\sigma},
\end{align}
which can be also expressed using \eqref{defn:surface} as
\begin{align}\label{defn:Jeps_T}
J_{\varepsilon}^{\mathcal{T}}[Q]&=J_{\varepsilon}^{X}[Q]+J_{\varepsilon}^{Y}[Q]+J_{\varepsilon}^{Z}[Q],
\end{align}
where:
\begin{align}\label{defn:Jeps_T_X_Y_Z}
\begin{cases}
J_{\varepsilon}^{X}[Q]&=\dfrac{\varepsilon^{3-\alpha}}{\varepsilon-\varepsilon^{\alpha}}\displaystyle{\sum_{k=1}^{X_{\varepsilon}}\int_{\mathcal{T}_x^k}f_s(Q,\nu)\text{d}\sigma};\\
\vspace{-4mm}&\\
J_{\varepsilon}^{Y}[Q]&=\dfrac{\varepsilon^{3-\alpha}}{\varepsilon-\varepsilon^{\alpha}}\displaystyle{\sum_{l=1}^{Y_{\varepsilon}}\int_{\mathcal{T}_y^l}f_s(Q,\nu)\text{d}\sigma};\\
\vspace{-4mm}&\\
J_{\varepsilon}^{Z}[Q]&=\dfrac{\varepsilon^{3-\alpha}}{\varepsilon-\varepsilon^{\alpha}}\displaystyle{\sum_{m=1}^{Z_{\varepsilon}}\int_{\mathcal{T}_z^m}f_s(Q,\nu)\text{d}\sigma}.
\end{cases}
\end{align}

The main reason why we are using relation \eqref{defn:Jeps} is because we can prove that $J_{\varepsilon}^{\mathcal{S}}[Q]$ has no influence over the homogenised material, since this surface term describes only the interaction between the liquid crystal and some parts of the ``outer" boundary of the scaffold (close to $\partial\Omega$), which are small and fewer in number compared to the interactions with the other parts of the scaffold (see \cref{subsection:nocontribution}).

\begin{remark}
In this paper, we use the notation $A\lesssim B$ for two real numbers $A$ and $B$ whenever there exists an $\varepsilon$-independent constant $C$ such that $A\leq C\cdot B$.
\end{remark}

We assume furthermore that: 
\begin{itemize}
\item[($A_1$)] $\Omega\subset\mathbb{R}^3$ is a smooth and bounded domain

\item[($A_2$)] \textit{$1<\alpha<\dfrac{3}{2}$;}
\end{itemize}

As we have seen already, the condition $1<\alpha$ grants the existence of the parallelipipeds generated in equations \eqref{defn:P_x}, \eqref{defn:P_y} and \eqref{defn:P_z}, but it also ensures the dilute regime. For a detailed proof of the last statement, consult \cref{subsection:appendix1}.

\begin{itemize}
\item[($A_3$)] \textit{There exists a constant $\lambda_{\Omega}>0$ such that} $$\text{dist}(z^i_{\varepsilon},\partial\Omega)+\dfrac{1}{2}\inf_{j\neq i}|z^j_{\varepsilon}-z^i_{\varepsilon}|\geq \lambda_{\Omega}\varepsilon$$ \textit{for any $\varepsilon>0$ and any center $z^i_{\varepsilon}$ of an object (either a black cube or a gray parallelipiped) that is contained the cubic microlattice, where $i\in\overline{1,(N_{\varepsilon}+M_{\varepsilon})}$.}
\end{itemize}

\begin{itemize}
\item[($A_4$)] \textit{As $\varepsilon\rightarrow 0$, the measures 
\begin{equation}\label{defn:measures}
\mu_{\varepsilon}^X:=\varepsilon^3\sum_{k=1}^{X_{\varepsilon}}\delta_{y^{x,k}_{\varepsilon}},\;\mu_{\varepsilon}^Y:=\varepsilon^3\sum_{l=1}^{Y_{\varepsilon}}\delta_{y^{y,l}_{\varepsilon}}\;\text{and}\;\mu_{\varepsilon}^Z:=\varepsilon^3\sum_{m=1}^{Z_{\varepsilon}}\delta_{y^{z,m}_{\varepsilon}}
\end{equation}
converge weakly* (as measures in $\mathbb{R}^3$) to the Lebesgue measure restricted on $\Omega$, denoted $\emph{dx}\mres\Omega$.}
\end{itemize}

We say that a function $f:\mathcal{S}_0\otimes\mathbb{R}^{3}\rightarrow\mathbb{R}$ is strongly convex if there exists $\theta>0$ such that the function $\tilde{f}:\mathbb{S}_0\otimes\mathbb{R}^{3}\rightarrow\mathbb{R}$ defined by $\tilde{f}(D)=f(D)-\theta|D|^2$ is convex.

\begin{itemize}
\item[($A_5$)] \textit{$f_e:\mathcal{S}_0\otimes\mathbb{R}^3\rightarrow[0,+\infty)$ is differentiable, strongly convex and there exists a constant $\lambda_e>0$ such that} $$\lambda_e^{-1}|D|^2\leq f_e(D)\leq \lambda_e|D|^2,\;\;\;\;|(\nabla f_e)(D)|\leq\lambda_e(|D|+1),$$ \textit{for any $D\in\mathcal{S}_0\times\mathbb{R}^3$. }
\item[($A_6$)] \textit{$f_b:\mathcal{S}_0\rightarrow\mathbb{R}$ is continuous, bounded from below and there exists a constant $\lambda_b>0$ such that \linebreak $|f_b(Q)|\leq\lambda_b(|Q|^6+1)$ for any $Q\in\mathcal{S}_0$.}
\item[($A_7$)] \textit{$f_s:\mathcal{S}_0\times\mathbb{S}^2\rightarrow\mathbb{R}$ is continuous and there exists a strictly positive constant $\lambda_s$ such that, for any $Q_1,Q_2\in\mathcal{S}_0$ and any $\nu\in\mathbb{S}^2$, we have $$|f_s(Q_1,\nu)-f_s(Q_2,\nu)|\leq\lambda_s|Q_1-Q_2|\big(|Q_1|^3+|Q_2|^3+1\big).$$}
\end{itemize}

It is easy to see from here that $f_s$ has a quartic growth in $Q$.
\vspace{3mm}

\textbf{The homogenised functional.}  Let $f_{hom}:\mathcal{S}_0\rightarrow\mathbb{R}$ be the function defined as:

\begin{equation}\label{defn:f_hom}
f_{hom}(Q):=\dfrac{q+r}{qr}\int_{\mathcal{C}^x}f_s(Q,\nu)\text{d}\sigma+\dfrac{p+r}{pr}\int_{\mathcal{C}^y}f_s(Q,\nu)\text{d}\sigma+\dfrac{p+q}{pq}\int_{\mathcal{C}^z}f_s(Q,\nu)\text{d}\sigma,
\end{equation}

\noindent for any $Q\in\mathcal{S}_0$, where $\mathcal{C}^x$, $\mathcal{C}^y$ and $\mathcal{C}^z$ are defined in \eqref{defn:C_x_C_y_C_z}. From ($A_7$), we can deduce that $f_{hom}$ is also continuous and that it has a quartic growth. If we work in the symmetric case, that is $p=q=r$, then relation \eqref{defn:f_hom} becomes:

\begin{equation}\label{defn:f_hom_sym}
f_{hom}(Q):=\dfrac{2}{p}\int_{\partial\mathcal{C}}f_s(Q,\nu)\text{d}\sigma.
\end{equation}

The main results of these notes concerns the asymptotic behaviour of local minimisers of the functional $\mathcal{F}_{\varepsilon}$, as $\varepsilon\rightarrow 0$.

Let $g\in H^{1/2}(\partial\Omega,\mathcal{S}_0)$ be a boundary datum. We denote by $H^1_g(\Omega,\mathcal{S}_0)$ the set of maps $Q$ from $H^1(\Omega,\mathcal{S}_0)$ such that $Q=g$ on $\partial\Omega$ in the trace sense. Similarly, we define $H^1_g(\Omega_{\varepsilon},\mathcal{S}_0)$ to be $H^1(\Omega_{\varepsilon})$ with $Q=g$ on $\partial\Omega$ in the trace sense.

We use the harmonic extension operator, $E_{\varepsilon}:H^1_g(\Omega_{\varepsilon},\mathcal{S}_0)\rightarrow H^1_g(\Omega,\mathcal{S}_0)$, defined in the following way: $E_{\varepsilon}Q:=Q$ on $\Omega_{\varepsilon}$ and inside the scaffold, $E_{\varepsilon}Q$ is the unique solution of the following problem:
\begin{equation*}
\left\{
\begin{array}{ll}
\Delta E_{\varepsilon}Q=0 & \text{in}\;\mathcal{N}_{\varepsilon}\\
E_{\varepsilon}Q\equiv Q & \text{on}\; \partial\mathcal{N}_{\varepsilon}.
\end{array}
\right. 
\end{equation*}

Using this framework, we can produce the main result of this work:

\begin{theorem}\label{th:local_min}
Suppose that the assumptions ($A_1$)-($A_7$) are satisfied. Let $Q_0\in H^1_g(\Omega,\mathcal{S}_0)$ be an isolated $H^1$-local minimiser for $\mathcal{F}_0$, that is, there exists $\delta_0>0$ such that $\mathcal{F}_0[Q_0]<\mathcal{F}_0[Q]$ for any $Q\in H^1_g(\Omega,\mathcal{S}_0)$ such that $\big\|Q-Q_0\big\|_{H^1_g(\Omega,\mathcal{S}_0)}\leq \delta_0$ and $Q\neq Q_0$. Then for any $\varepsilon$ sufficiently small enough, there exists a sequence of $H^1$-local minimisers $Q_{\varepsilon}$ of $\mathcal{F}_{\varepsilon}$ such that $E_{\varepsilon}Q_{\varepsilon}\rightarrow Q_0$ strongly in $H^1_g(\Omega,\mathcal{S}_0)$.
\end{theorem}

\subsection{Applications to the Landau-de Gennes model}

We use the classical terms from the Landau-de Gennes model for nematic liquid crystals. 

For the \textit{elastic energy density}, we take:
\begin{align*}
f_e(\nabla Q):=L_1\partial_k Q_{ij}\partial_k Q_{ij}+L_2\partial_j Q_{ij}\partial_k Q_{ik}+L_3 \partial_j Q_{ik}\partial_k Q_{ij},
\end{align*}
where the Einstein's summation convention is assumed. 

In order to fulfill assumption ($A_5$), we take as in \cite{Longa}:
\begin{align}\label{eq:ineq_for_elastic}
L_1>0, \hspace{5mm}-L_1<L_3<2 L_1, \hspace{5mm}-\dfrac{3}{5}L_1-\dfrac{1}{10}L_3<L_2. 
\end{align}

For the \textit{bulk energy density}, we use several versions of it. The first one is the classical quartic polynomial in the scalar invariants of $Q$, that is:
\begin{align}\label{defn:f_b_LDG}
f_b^{LDG}(Q):=a\,\text{tr}(Q^2)-b\,\text{tr}(Q^3)+c\,\big(\text{tr}(Q^2)\big)^2,
\end{align}
which verifies the conditions of \cref{th:local_min}.

We also prove similar results for a general polynomial in the scalar invariants of $Q$, that is:
\begin{align}\label{defn:f_b_gen}
f_b^{gen}[Q]=\sum_{k=2}^{N} a_k\,\text{tr}(Q^k),
\end{align}
where $N\in\mathbb{N}$, $N\geq 4$ is fixed, with the coefficients $a_k\in\mathbb{R}$ chosen such that the polynomial $h:\mathbb{R}\rightarrow\mathbb{R}$, defined by $h(x)=\sum_{k=2}^{N}a_k x^k$, for any $x\in\mathbb{R}$, admits at least one local minimum over $\mathbb{R}$. 

In all the cases, the coefficient of $\text{tr}(Q^2)$ depends on the temperature at which the phase transition occurs.  More specifically, $a$ from \eqref{defn:f_b_LDG} is of the form $a:=a_{*}(T-T_{*})$, in which $a_{*}$ is a material parameter and $T_{*}$ is the characteristic temperature of the nematic liquid crystal material (the temperature where the isotropic state starts losing \textit{local} stability).

For each of the versions of the \textit{bulk energy densities}, we choose suitable \textit{surface energy densities}, such that, in the homogenised functional, the surface terms have a similar form with the effective bulk energy, which is now in the same form as in the non-homogenised situation, but with different coefficients, most important of which the coefficient of $\text{tr}(Q^2)$ is now different.

\cref{th:local_min} holds for any values of $p$, $q$ and $r$, that is, for any type of parallelipiped chosen for the construction of the scaffold. In reality, 2PP (two-photon polymerization) materials with cubic symmetry properties have been obtained (for example, \cite{nematiccage}) and they do represent an object of interest in the construction of nematic scaffolds.

If the scaffold presents symmetries, then the physical invariances require
\begin{equation*}
f_s(UQU^T,Uu)=f_s(Q,u),\;\forall (Q,u)\in\mathcal{S}_0\times\mathbb{R}^3,\;U\in\mathcal{O}(3)
\end{equation*}
\noindent and this leads, according to Proposition 2.6 from \cite{CanevariZarnescu1}, to a \textit{surface energy} of the form 
\begin{equation*}
f_s(Q,\nu)=\tilde{f}_s(\text{tr}(Q^2),\text{tr}(Q^3),\nu\cdot Q\nu,\nu\cdot Q^2\nu),\;\forall (Q,\nu)\in\mathcal{S}_0\times\mathbb{R}^3.
\end{equation*}

Even though we have this result, we still use a more general form for $f_s(Q,\nu)$, in which we include terms of the form $\nu\cdot Q^k\nu$, since they grant an easier way to compute the homogenised functional, in this case with this type of scaffold. Still, according to \cref{prop:int_en_dens_LDG}, we obtain in the homogenised functional terms of the form $\text{tr}(Q^k)$, with $k\geq 4$, which depend only on $\text{tr}(Q^2)$ and $\text{tr}(Q^3)$, since $\text{tr}(Q)=0$. In order to prove this statement, let $\lambda_1$, $\lambda_2$ and $\lambda_3$ the eigenvalues of $Q$. Then they satisfy the system:
\begin{align*}
\begin{cases}
\lambda_1+\lambda_2+\lambda_3=0\\
\lambda_1^2+\lambda_2^2+\lambda_3^2=\text{tr}(Q^2)\\
\lambda_1^3+\lambda_2^3+\lambda_3^3=\text{tr}(Q^3)
\end{cases}
\end{align*}
and, by solving the system, we can see that $\lambda_1$, $\lambda_2$ and $\lambda_3$ can be viewed as functions of $\text{tr}(Q^2)$ and $\text{tr}(Q^3)$. Since $\text{tr}(Q^k)=\lambda_1^k+\lambda_2^k+\lambda_3^k$, for any $k\in\mathbb{N}$, $k\geq 1$, then it is easy to see from here that $\text{tr}(Q^k)$, for $k\geq 4$, is depending only on $\text{tr}(Q^2)$ and $\text{tr}(Q^3)$. Indeed, by Cayley-Hamilton theorem, the identity:
\begin{align*}
Q^3-\dfrac{1}{2}\text{tr}(Q^2)Q-\dfrac{1}{3}\text{tr}(Q^3)\mathbb{I}_3=0
\end{align*}
becomes valid for any Q-tensor $Q$, where $\mathbb{I}_3$ is the $3\times 3$ identity matrix. Multiplying this identity succcessively by $Q$, $Q^2$, $Q^3$ and so on and taking the trace we obtain the claim. 

\subsubsection{The case $p=q=r$}

Assuming $p=q=r$ implies that the parallelipipeds constructed in \eqref{defn:ceps} are actually cubes and that the ``cells" of the nematic scaffold are also cubes. 

For each of the \textit{bulk energies} presented before, we take different \textit{surface energy densities}. One of the corresponding choices of the \textit{surface energy} $f_s$ in the case of \eqref{defn:f_b_LDG} is:
\begin{align}\label{defn:f_s_LDG}
f_s^{LDG}(Q,\nu)=\dfrac{p}{4}\bigg((a'-a)(\nu\cdot Q^2\nu)-(b'-b)(\nu\cdot Q^3\nu)+2(c'-c)(\nu\cdot Q^4\nu)\bigg)
\end{align}
where $a'$, $b'$ and $c'$ are the desired coefficients in the homogenised bulk potential, such that in the homogenised material, we have:
\begin{align}\label{defn:f_hom_LDG}
f_{hom}^{LDG}(Q)=(a'-a)\,\text{tr}(Q^2)-(b'-b)\,\text{tr}(Q^3)+(c'-c)\,\big(\text{tr}(Q^2)\big)^2.
\end{align}

We are interested in studying the behaviour of the whole material when $\varepsilon\rightarrow 0$, that is, studying the following functionals:
\begin{align}\label{defn:F_eps_LDG}
\mathcal{F}_{\varepsilon}^{LDG}[Q_{\varepsilon}]&:=\int_{\Omega_{\varepsilon}}\big(f_e(\nabla Q_{\varepsilon})+a\,\text{tr}(Q_{\varepsilon}^2)-b\,\text{tr}(Q_{\varepsilon}^3)+c\,\big(\text{tr}(Q_{\varepsilon}^2)\big)^2\big)\text{d}x+\dfrac{\varepsilon^{3-\alpha}}{\varepsilon-\varepsilon^{\alpha}}\int_{\partial\mathcal{N}_{\varepsilon}}f_s^{LDG}(Q_{\varepsilon},\nu)\text{d}\sigma 
\end{align}
and
\begin{align}\label{defn:F_0_LDG}
\mathcal{F}_0^{LDG}[Q]&:=\int_{\Omega}\big(f_e(\nabla Q)+a'\,\text{tr}(Q^2)-b'\,\text{tr}(Q^3)+c'\,\big(\text{tr}(Q^2)\big)^2\big)\text{d}x, 
\end{align}
with $a'$, $b'$ and $c'$ being the desired parameters.

\begin{theorem}\label{th:LDG}
Let $(a,b,c)$ and $(a',b',c')$ be two set of parameters with $c>0$ and $c'>0$. Suppose that the assumptions ($A_1$)-($A_7$) are satisfied and also the inequalities from \eqref{eq:ineq_for_elastic}. Then, for any isolated $H^1$-local minimiser $Q_0$ of the functional $\mathcal{F}_0^{LDG}$ defined by \eqref{defn:F_0_LDG}, and for $\varepsilon>0$ sufficiently small enough, there exists a sequence of local minimisers $Q_{\varepsilon}$ of the functionals $\mathcal{F}_{\varepsilon}^{LDG}$, defined by \eqref{defn:F_eps_LDG}, such that $E_{\varepsilon}Q_{\varepsilon}\rightarrow Q_0$ strongly in $H^1_g(\Omega,\mathcal{S}_0)$.  
\end{theorem}

\begin{proof}
This theorem is a particular case of \cref{th:local_min}. It is sufficient to prove that relation \eqref{defn:f_hom_LDG} can be obtained via \eqref{defn:f_hom_sym}, that is:
\begin{align*}
f_{hom}^{LDG}(Q)=\dfrac{2}{p}\int_{\partial\mathcal{C}}f_s^{LDG}(Q,\nu)\text{d}\sigma=(a'-a)\;\text{tr}(Q^2)-(b'-b)\;\text{tr}(Q^3)+(c'-c)\;\big(\text{tr}(Q^2)\big)^2.
\end{align*}

Using \cref{prop:int_en_dens_LDG}, we have:
\begin{align*}
\int_{\partial\mathcal{C}}\nu\cdot Q^2\nu\text{d}\sigma =2\text{tr}(Q^2),\hspace{3mm}
\int_{\partial\mathcal{C}}\nu\cdot Q^3\nu\text{d}\sigma =2\text{tr}(Q^3)\hspace{3mm}\text{and}\hspace{3mm}\int_{\partial\mathcal{C}}\nu\cdot Q^4\nu\text{d}\sigma =2\text{tr}(Q^4),
\end{align*}
from which we get
\begin{align*}
\dfrac{2}{p}\int_{\partial\mathcal{C}}f_s^{LDG}(Q,\nu)\text{d}\sigma &=\dfrac{2}{p}\cdot\dfrac{p}{4}\big((a'-a)\cdot 2\text{tr}(Q^2)-(b'-b)\cdot 2\text{tr}(Q^3)+2(c'-c)\cdot 2\text{tr}(Q^4)\big)\Rightarrow\\
\Rightarrow f_{hom}^{LDG}(Q)&=(a'-a)\text{tr}(Q^2)-(b'-b)\text{tr}(Q^3)+(c'-c)\cdot 2\text{tr}(Q^4).
\end{align*}

Since $Q\in\mathcal{S}_0$, then, by Cayley-Hamilton theorem, if $\lambda_1$, $\lambda_2$ and $\lambda_3$ are the eigenvalues of $Q$, we have:
\begin{align*}
\begin{cases}
\lambda_1+\lambda_2+\lambda_3=\text{tr}(Q)=0\\
\lambda_1\lambda_2+\lambda_2\lambda_3+\lambda_3\lambda_1=\dfrac{1}{2}\big(\big(\text{tr}(Q)\big)^2-\text{tr}(Q^2)\big)=-\dfrac{1}{2}\text{tr}(Q^2)
\end{cases}
\end{align*}
and
\begin{align*}
\text{tr}(Q^4)=\lambda_1^4+\lambda_2^4+\lambda_3^4&=(\lambda_1^2+\lambda_2^2+\lambda_3^2)^2-2(\lambda_1^2\lambda_2^2+\lambda_2^2\lambda_3^2+\lambda_3^2\lambda_1^2)\\
&=(\lambda_1^2+\lambda_2^2+\lambda_3^2)^2-2\big((\lambda_1\lambda_2+\lambda_2\lambda_3+\lambda_3\lambda_1)^2-2\lambda_1\lambda_2\lambda_3(\lambda_1+\lambda_2+\lambda_3)\big)\\
&=\big(\text{tr}(Q^2)\big)^2-2\bigg(-\dfrac{1}{2}\text{tr}(Q^2)\bigg)^2\\
&=\dfrac{1}{2}\big(\text{tr}(Q^2)\big)^2
\end{align*}
from which we get the relation $2\text{tr}(Q^4)=\big(\text{tr}(Q^2)\big)^2$.

Hence, we conclude that:
\begin{align*}
f_{hom}^{LDG}(Q)&=(a'-a)\text{tr}(Q^2)-(b'-b)\text{tr}(Q^3)+(c'-c)\big(\text{tr}(Q^2)\big)^2.
\end{align*}
\end{proof}

If we assume that in the case of \eqref{defn:f_b_LDG} we have $b=c=0$ ($f_b^{RP}=a\;\text{tr}(Q^2)$) and we desire $b'=c'=0$, that is, the only nonzero coefficients are $a$ and $a'$, then another suitable choice for $f_s$ is given by the Rapini-Papoular form \eqref{eq:intro_Rapini-Papoular}:
\begin{align}\label{defn:f_s_RP}
f_s^{RP}(Q,\nu)=\dfrac{p}{12}(a'-a)\;\text{tr}(Q-Q_{\nu})^2,
\end{align}
where $Q_{\nu}=\nu\otimes\nu-\mathbb{I}_3/3$ and $\mathbb{I}_3$ is the $3\times 3$ identity matrix.

In this case, we have:
\begin{align}\label{defn:F_eps_RP}
\mathcal{F}_{\varepsilon}^{RP}[Q_{\varepsilon}]&:=\int_{\Omega_{\varepsilon}}\big(f_e(\nabla Q_{\varepsilon})+a\,\text{tr}(Q_{\varepsilon}^2)\big)\text{d}x+\dfrac{p}{2}\cdot(a'-a)\cdot\bigg(\dfrac{\varepsilon^{3-\alpha}}{\varepsilon-\varepsilon^{\alpha}}\int_{\partial\mathcal{N}_{\varepsilon}}\;\text{tr}(Q_{\varepsilon}-Q_{\nu})^2\text{d}\sigma\bigg) 
\end{align}
and we prove that
\begin{align}\label{defn:f_hom_RP}
f_{hom}^{RP}(Q)=(a'-a)\;\text{tr}(Q^2),
\end{align}
and
\begin{align}\label{defn:F_0_RP}
\mathcal{F}_0^{RP}[Q]&:=\int_{\Omega}\big(f_e(\nabla Q)+a'\,\text{tr}(Q^2)\big)\text{d}x.
\end{align}

\begin{theorem}\label{th:RP}
Let $a$ and $a'$ be two parameters. Suppose that the assumptions ($A_1$)-($A_7$) are satisfied and also the inequalities from \eqref{eq:ineq_for_elastic}. Then, for any isolated $H^1$-local minimiser $Q_0$ of the functional $\mathcal{F}_0^{RP}$ defined by \eqref{defn:F_0_RP}, and for $\varepsilon>0$ sufficiently small enough, there exists a sequence of local minimisers $Q_{\varepsilon}$ of the functionals $\mathcal{F}_{\varepsilon}^{RP}$, defined by \eqref{defn:F_eps_RP}, such that $E_{\varepsilon}Q_{\varepsilon}\rightarrow Q_0$ strongly in $H^1_g(\Omega,\mathcal{S}_0)$.  
\end{theorem}

\begin{proof}
The proof follows the same steps as in the proof of \cref{th:LDG}, using \cref{prop:int_en_dens_RP}. We only have to prove that relation \eqref{defn:f_hom_RP} can be obtained using \eqref{defn:f_hom_sym}, knowing that \eqref{defn:f_s_RP} holds.

From \eqref{defn:f_hom_sym} and \cref{prop:int_en_dens_RP}, we have:
\begin{align*}
f_{hom}^{RP}(Q)&=\dfrac{2}{p}\int_{\partial\mathcal{C}}f_s^{RP}(Q,\nu)\text{d}\sigma=\dfrac{2}{p}\cdot\dfrac{p}{12}(a'-a)\int_{\partial\mathcal{C}}\text{tr}(Q-Q_{\nu})\text{d}\sigma \\
&=\dfrac{(a'-a)}{6}\big(6\text{tr}(Q^2)+4\big)=(a'-a)\text{tr}(Q^2)+\dfrac{2}{3}(a'-a).
\end{align*}

We can eliminate the constant $\dfrac{2}{3}(a'-a)$ from $f_{hom}^{RP}$, since it does not influence the minimisers of the functional $\mathcal{F}_{\varepsilon}^{RP}$, so we obtain: $f_{hom}^{RP}(Q)=(a'-a)\text{tr}(Q^2)$.
\end{proof}

For the more general case described by \eqref{defn:f_b_gen}, we can choose:
\begin{align}\label{defn:f_s_gen}
f_s^{gen}(Q,\nu)=\dfrac{p}{4}\sum_{k=2}^{M}b_k(\nu\cdot Q^k\nu),
\end{align}
where $(b_k)_{k\in\overline{2,M}}$ are the coefficients of the polynomial $i:\mathbb{R}\rightarrow\mathbb{R}$ of degree $M\in\mathbb{N}$, $M\geq 4$, defined by $i(x)=\sum_{k=2}^{M}b_k x^k$, for any $x\in\mathbb{R}$, with the property that $i$ admits at least one local minimum over $\mathbb{R}$.

In the same manner, we have 
\begin{align*}
f_{hom}^{gen}(Q)=\sum_{k=2}^{\text{max}\{M,N\}}c_k\;\text{tr}(Q^k),
\end{align*}
where, for any $k\in\overline{2,\text{max}\{M,N\}}$, we have
\begin{align*}
c_k=\begin{cases}
a_k+b_k,\;\text{if}\;2\leq k\leq\text{min}\{M,N\}\\
a_k,\;\text{if}\;\text{min}\{M,N\}<k\leq\text{max}\{M,N\}\;\text{and}\;M\leq N\\
b_k,\;\text{if}\;\text{min}\{M,N\}<k\leq\text{max}\{M,N\}\;\text{and}\;M\geq N.
\end{cases}
\end{align*}

In this case, $\mathcal{F}_{\varepsilon}$ and $\mathcal{F}_0$ become:
\begin{align}\label{defn:F_eps_gen}
\mathcal{F}_{\varepsilon}^{gen}[Q_{\varepsilon}]:=\int_{\Omega_{\varepsilon}}\bigg(f_e(\nabla Q_{\varepsilon})+\sum_{k=2}^{N}a_k\;\text{tr}(Q_{\varepsilon}^k)\bigg)\text{d}x+\dfrac{p}{4}\cdot\sum_{k=2}^{M}b_k\cdot\bigg(\dfrac{\varepsilon^{3-\alpha}}{\varepsilon-\varepsilon^{\alpha}}\int_{\partial\mathcal{N}_{\varepsilon}}(\nu\cdot Q_{\varepsilon}^k\nu)\text{d}\sigma\bigg)
\end{align}
and
\begin{align}\label{defn:F_0_gen}
\mathcal{F}_{0}^{gen}[Q]=\int_{\Omega}\bigg(f_e(\nabla Q)+\sum_{k=2}^{\text{max}\{M,N\}}c_k\;\text{tr}(Q^k)\bigg)\text{d}x.
\end{align}

\begin{theorem}\label{th:gen}
Let $(a_k)_{k\in\overline{2,N}}$ and $(b_k)_{k\in\overline{2,M}}$ be such that the polynomials $h$ and $i$ defined earlier admit at least one local minimum over $\mathbb{R}$. Suppose that the assumptions ($A_1$)-($A_7$) are satisfied and also the inequalities from \eqref{eq:ineq_for_elastic}. Then, for any isolated $H^1$-local minimiser $Q_0$ of the functional $\mathcal{F}_0^{gen}$ defined by \eqref{defn:F_0_gen}, and for $\varepsilon>0$ sufficiently small enough, there exists a sequence of local minimisers $Q_{\varepsilon}$ of the functionals $\mathcal{F}_{\varepsilon}^{gen}$, defined by \eqref{defn:F_eps_gen}, such that $E_{\varepsilon}Q_{\varepsilon}\rightarrow Q_0$ strongly in $H^1_g(\Omega,\mathcal{S}_0)$.  
\end{theorem}

\begin{proof}
This theorem is a particular case of \cref{th:local_min}. Using once again \cref{prop:int_en_dens_LDG}, the proof is finished. 
\end{proof}

\subsubsection{The case $p\neq q\neq r\neq p$}

Assume now that $p$, $q$ and $r$ are three different real values, each greater than or equal to 1. For this case, we only present one of the theorems stated in the last subsection.

Relation \eqref{defn:f_hom} states that:
\begin{equation*}
f_{hom}(Q)=\dfrac{q+r}{qr}\int_{\mathcal{C}^x}f_s(Q,\nu)\text{d}\sigma+\dfrac{p+r}{pr}\int_{\mathcal{C}^y}f_s(Q,\nu)\text{d}\sigma+\dfrac{p+q}{pq}\int_{\mathcal{C}^z}f_s(Q,\nu)\text{d}\sigma.
\end{equation*}

We only illustrate how to proceed for the case in which we have 
\begin{align*}
f_b(Q)=a\;\text{tr}(Q^2)-b\;\text{tr}(Q^3)+c\;\text{tr}(Q^4)=a\;\text{tr}(Q^2)-b\;\text{tr}(Q^3)+\dfrac{c}{2}\;\big(\text{tr}(Q^2)\big)^2,
\end{align*} 
with $c>0$, and similar results can be obtained for the other cases in which we modify the form of $f_b$. 

Let 
\begin{align*}
A=\dfrac{1}{3}\begin{pmatrix}
-\dfrac{2}{p}+\dfrac{1}{q}+\dfrac{1}{r} & 0 & 0\\
0 & \dfrac{1}{p}-\dfrac{2}{q}+\dfrac{1}{r} & 0\\
0 & 0 & \dfrac{1}{p}+\dfrac{1}{q}-\dfrac{2}{r}
\end{pmatrix}\hspace{3mm}\text{and}\hspace{3mm}B=\begin{pmatrix}
\dfrac{1}{q}+\dfrac{1}{r} & 0 & 0\\
0 & \dfrac{1}{p}+\dfrac{1}{r} & 0\\
0 & 0 & \dfrac{1}{p}+\dfrac{1}{q}
\end{pmatrix}.
\end{align*}
and $\omega=\dfrac{2}{3}\bigg(\dfrac{1}{p}+\dfrac{1}{q}+\dfrac{1}{r}\bigg)$. Note that $A$, $B$ and $\omega$ are constants depending only on the choice of $p$, $q$ and $r$. Moreover, we have $\text{tr}(A)=0$ and $B=\omega\mathbb{I}_3+A$, where $\mathbb{I}_3$ is the $3\times 3$ identity matrix.

Consider now
\begin{align}\label{defn:f_s_asym}
f_s^{asym}(Q,\nu)=\dfrac{1}{2\omega}\big((a'-a)(\nu\cdot Q^2\nu)-(b'-b)(\nu\cdot Q^3\nu)+(c'-c)(\nu\cdot Q^4\nu)\big),
\end{align}
with $a'$, $b'$ and $c'$ real parameters such that $c'>0$ and the associated free energy functional:
\begin{align}\label{defn:F_eps_asym}
\mathcal{F}_{\varepsilon}^{asym}[Q_{\varepsilon}]:=\int_{\Omega}\big(f_e(\nabla Q_{\varepsilon})+a\;\text{tr}(Q_{\varepsilon}^2)-b\;\text{tr}(Q_{\varepsilon}^3)+c\;\text{tr}(Q_{\varepsilon}^4)\big)\text{d}x+\dfrac{\varepsilon^{3-\alpha}}{\varepsilon-\varepsilon^{\alpha}}\int_{\partial\mathcal{N}_{\varepsilon}}f_s^{asym}(Q_{\varepsilon},\nu).
\end{align}

We prove in the next theorem that the homogenised functional is:
\begin{align}\label{defn:f_hom_asym}
f_{hom}^{asym}(Q)&=\big((a'-a)\text{tr}(Q^2)-(b'-b)\text{tr}(Q^3)+(c'-c)\text{tr}(Q^4)\big)+\notag\\
&\hspace{2mm}+\dfrac{1}{\omega}\big((a'-a)\text{tr}(A\cdot Q^2)-(b'-b)\text{tr}(A\cdot Q^3)+(c'-c)\text{tr}(A\cdot Q^4)\big).
\end{align}

\begin{theorem}\label{th:asym}
Let $(a,b,c)$ and $(a',b',c')$ be two set of parameters with $c>0$ and $c'>0$. Suppose that the assumptions ($A_1$)-($A_7$) are satisfied and also the inequalities from \eqref{eq:ineq_for_elastic}. Then, for $\varepsilon>0$ sufficiently small enough and for any isolated $H^1$-local minimiser $Q_0$ of the functional:
\begin{align*}
\mathcal{F}_0^{asym}[Q]&:=\int_{\Omega}\big(f_e(\nabla Q)+a'\emph{tr}(Q^2)-b'\emph{tr}(Q^3)+c'\big(\emph{tr}(Q^2)\big)^2\big)\emph{d}x+\notag\\
&\hspace{2mm}+\dfrac{1}{\omega}\int_{\Omega}\big((a'-a)\emph{tr}(A\cdot Q^2(x))-(b'-b)\emph{tr}(A\cdot Q^3(x))+(c'-c)\emph{tr}(A\cdot Q^4(x))\big)\emph{d}x
\end{align*}
there exists a sequence of local minimisers $Q_{\varepsilon}$ of the functionals $\mathcal{F}_{\varepsilon}^{asym}$, defined by \eqref{defn:F_eps_asym}, such that $E_{\varepsilon}Q_{\varepsilon}\rightarrow Q_0$ strongly in $H^1_g(\Omega,\mathcal{S}_0)$.  
\end{theorem}

\begin{proof}
We follow the same steps as in \cref{th:LDG} and in \cref{th:RP}, that is, we prove that relation \eqref{defn:f_hom_asym} can be obtained using \eqref{defn:f_hom} and \eqref{defn:f_s_asym}.

In the proof of \cref{prop:int_en_dens_LDG}, we obtain that:
\begin{align*}
\int_{\mathcal{C}^x}\nu\cdot Q^k\nu\text{d}\sigma =2q_{11,k},\hspace{3mm}
\int_{\mathcal{C}^y}\nu\cdot Q^k\nu\text{d}\sigma =2q_{22,k}\hspace{3mm}\text{and}\hspace{3mm}\int_{\mathcal{C}^z}\nu\cdot Q^k\nu\text{d}\sigma =2q_{33,k},
\end{align*}
for any $k\in\mathbb{N}$, $k\neq 0$, where $q_{ij,k}$ is the $ij$-th component of $Q^k$, from which we get:
\begin{align*}
\int_{\mathcal{C}^x}f_s^{asym}(Q,\nu)\text{d}\sigma=\dfrac{1}{\omega}((a'-a)q_{11,2}-(b'-b)q_{11,3}+(c'-c)q_{11,4})\\
\int_{\mathcal{C}^y}f_s^{asym}(Q,\nu)\text{d}\sigma=\dfrac{1}{\omega}((a'-a)q_{22,2}-(b'-b)q_{22,3}+(c'-c)q_{22,4})\\
\int_{\mathcal{C}^z}f_s^{asym}(Q,\nu)\text{d}\sigma=\dfrac{1}{\omega}((a'-a)q_{33,2}-(b'-b)q_{33,3}+(c'-c)q_{33,4}).
\end{align*}

Using now \eqref{defn:f_hom}, we obtain:
\begin{align*}
f_{hom}^{asym}(Q)&=\dfrac{1}{\omega}(a'-a)\bigg(q_{11,2}\bigg(\dfrac{1}{q}+\dfrac{1}{r}\bigg)+q_{22,2}\bigg(\dfrac{1}{p}+\dfrac{1}{r}\bigg)+q_{33,2}\bigg(\dfrac{1}{p}+\dfrac{1}{q}\bigg)\bigg)-\\
&-\dfrac{1}{\omega}(b'-b)\bigg(q_{11,3}\bigg(\dfrac{1}{q}+\dfrac{1}{r}\bigg)+q_{22,3}\bigg(\dfrac{1}{p}+\dfrac{1}{r}\bigg)+q_{33,3}\bigg(\dfrac{1}{p}+\dfrac{1}{q}\bigg)\bigg)+\\
&+\dfrac{1}{\omega}(c'-c)\bigg(q_{11,4}\bigg(\dfrac{1}{q}+\dfrac{1}{r}\bigg)+q_{22,4}\bigg(\dfrac{1}{p}+\dfrac{1}{r}\bigg)+q_{33,4}\bigg(\dfrac{1}{p}+\dfrac{1}{q}\bigg)\bigg)\\
\end{align*}
which we can see as:
\begin{align*}
f_{hom}^{asym}(Q)&=\dfrac{1}{\omega}\big((a'-a)\text{tr}(B\cdot Q^2)-(b'-b)\text{tr}(B\cdot Q^3)+(c'-c)\text{tr}(B\cdot Q^4)\big)\\
\end{align*}
and since $B=\omega\mathbb{I}_3+A$, we obtain:
\begin{align*}
f_{hom}^{asym}(Q)&=\big((a'-a)\text{tr}(Q^2)-(b'-b)\text{tr}(Q^3)+(c'-c)\text{tr}(Q^4)\big)+\\
&+\dfrac{1}{\omega}\big((a'-a)\text{tr}(A\cdot Q^2)-(b'-b)\text{tr}(A\cdot Q^3)+(c'-c)\text{tr}(A\cdot Q^4)\big),
\end{align*}
from which we conclude.
\end{proof}

\begin{remark}
We have obtained in this case a part which is exactly the same as in the case in which we have cubic symmetry, but also three terms of the form $\emph{tr}(A\cdot Q^k)$ which describe a new preferred alignment of the liquid crystal particles inside of the domain, given by the loss of the cubic symmetry of the scaffold.
\end{remark}

\section{Properties of the functional \texorpdfstring{$\mathcal{F}_{\varepsilon}$}{Fe}}\label{section:prop_F_eps}

\subsection{Analytical tools: trace and extension}\label{subsection:extension}

The main result of this subsection consists on a $L^p$ inequality, which is adapted from lemma 3.1. from \cite{CanevariZarnescu1}, because our scaffold now consists on inter-connected particles and the interaction between the liquid crystal and the cubic microlattice happens only up to five faces of the particles of the scaffold. 

In the following, given a set $\mathcal{P}\subset\mathbb{R}^{2}$ and a real number $a>0$, we define $a\mathcal{P}=\{ax:x\in\mathcal{P}\}$.

\begin{lemma}\label{lemma:bP-aP}
Let $\mathcal{P}\subseteq\mathbb{R}^2$ be a compact, convex set whose interior contains the origin. Let $a$ and $b$ be positive numbers such that $a<b$. Then there exists a bijective, Lipschitz map $\phi:b\mathcal{P}\setminus a\mathcal{P}\rightarrow \overline{B}_b\setminus\overline{B}_a$ that has a Lipschitz inverse and satisfies $$\|\nabla\phi\|_{L^{\infty}(b\mathcal{P}\setminus a\mathcal{P})}+\|\nabla(\phi^{-1})\|_{L^{\infty}(\overline{B}_b\setminus\overline{B}_a)}\leq C(\mathcal{P}),$$ where $C(\mathcal{P})$ is a positive constant that depends only on $\mathcal{P}$ and neither on $a$ nor $b$. 
\end{lemma}

The proof of \cref{lemma:bP-aP} follows the same steps as Lemma 3.2. from \cite{CanevariZarnescu1}, the only difference being that now we are in the case of $\mathbb{R}^2$ instead of $\mathbb{R}^3$.

\begin{lemma}\label{lemma:bP-aP-ineq}
Let $\mathcal{P}\subseteq\mathbb{R}^2$ be a compact, convex set whose interior contains the origin and $n\in[2,4]$. Then, there exists $C=C(\mathcal{P},\phi)>0$, such that for any $0<a\leq b$ and any $u\in H^1(b\mathcal{P}\setminus a\mathcal{P})$, there holds $$\oint_{\partial\big(a\mathcal{P}\big)}|u|^n\text{d}s\lesssim \dfrac{2aC}{b^2-a^2}\int_{b\mathcal{P}\setminus a\mathcal{P}}|u|^n\text{d}x+\dfrac{nC}{2}\int_{b\mathcal{P}\setminus a\mathcal{P}}\big(|u|^{2n-2}+|\nabla u|^2\big)\text{d}x,$$
where $\displaystyle{\oint}$ represents the curvilinear integral in $\mathbb{R}^2$.
\end{lemma}

\begin{proof}
Using \cref{lemma:bP-aP}, we can restrict without loss of generality to the case in which $\mathcal{P}=\overline{B}_1$, which is the two dimensional unit disk, centered in origin. Then $\partial B_{\tau}=\{x\in\mathbb{R}^2:|x|=\tau\}=\{(\rho,\theta):\rho=\tau,\;\theta\in[0,2\pi]\}$, for any $\tau>0$, and we can write, for any $\rho\in[a,b]$ and any $\theta\in[0,2\pi]$:
\begin{align*}
|u|^n(a,\theta)&=|u|^n(\rho,\theta)-\int_{a}^{\rho}\partial_{\tau}\big(|u|^n\big)(\tau,\theta)\text{d}\tau\\
&\leq |u|^n(\rho,\theta)+n\int_a^{\rho}\big(|u|^{n-1}\cdot |\partial_{\tau}u|\big)(\tau,\theta)\text{d}\tau\\
&\leq |u|^n(\rho,\theta)+\dfrac{n}{2}\int_a^{\rho}\big(|u|^{2n-2}+|\partial_{\tau}u|^2\big)(\tau,\theta)\text{d}\tau\\
|u|^n(a,\theta)&\leq |u|^n(\rho,\theta)+\dfrac{n}{2}\int_a^b\big(|u|^{2n-2}+|\nabla u|^2\big)(\tau,\theta)\text{d}\tau
\end{align*}

If we multiply both sides by $\rho$ and integrate over $[a,b]$ with respect to $\rho$, we get:
\begin{align*}
|u|^n(a,\theta)\int_a^b\rho\;\text{d}\rho&\leq \int_a^b|u|^n(\rho,\theta)\cdot\rho\;\text{d}\rho+\dfrac{n}{2}\int_a^b\rho\;\text{d}\rho\int_a^b\big(|u|^{2n-2}+|\nabla u|^2\big)(\tau,\theta)\text{d}\tau\\
\dfrac{b^2-a^2}{2}|u|^n(a,\theta)&\leq \int_a^b|u|^n(\rho,\theta)\cdot\rho\;\text{d}\rho+\dfrac{n(b^2-a^2)}{4}\int_a^b\big(|u|^{2n-2}+|\nabla u|^2\big)(\tau,\theta)\text{d}\tau.
\end{align*}

Since for any $\tau\in[a,b]$ we have $\tau>a$, then:
\begin{align*}
\dfrac{b^2-a^2}{2a}|u|^n(a,\theta)\cdot a&\leq \int_a^b|u|^n(\rho,\theta)\cdot\rho\;\text{d}\rho+\dfrac{n(b^2-a^2)}{4a}\int_a^b\big(|u|^{2n-2}+|\nabla u|^2\big)(\tau,\theta)\cdot \tau\;\text{d}\tau.
\end{align*}

Now we integrate with respect to $\theta$ over $[0,2\pi]$ and we get:
\begin{align*}
\dfrac{b^2-a^2}{2a}\int_0^{2\pi}|u|^n(a,\theta)\cdot a\;\text{d}\theta&\leq \int_0^{2\pi}\int_a^b|u|^n(\rho,\theta)\cdot\rho\;\text{d}\rho\text{d}\theta+\dfrac{n(b^2-a^2)}{4a}\int_0^{2\pi}\int_a^b\big(|u|^{2n-2}+|\nabla u|^2\big)(\tau,\theta)\cdot \tau\;\text{d}\tau\text{d}\theta\\
\dfrac{b^2-a^2}{2a}\oint_{\partial B_a}|u|^n\text{d}s &\leq \int_{B_b\setminus B_a}|u|^n\text{d}x+\dfrac{n(b^2-a^2)}{4a}\int_{B_b\setminus B_a}\big(|u|^{2n-2}+|\nabla u|^2\big)\text{d}x,
\end{align*}
therefore
\begin{align*}
\oint_{\partial B_a}|u|^n\text{d}s\leq \dfrac{2a}{b^2-a^2}\int_{B_b\setminus B_a}|u|^n\text{d}x+\dfrac{n}{2}\int_{B_b\setminus B_a}\big(|u|^{2n-2}+|\nabla u|^2\big)\text{d}x.
\end{align*}

If we apply now the Lipschitz homeomorphism $\phi$ defined by \cref{lemma:bP-aP}, the conclusion follows.
\end{proof}

\begin{lemma}\label{lemma:surface}
For any $Q\in H^1(\Omega_{\varepsilon},\mathcal{S}_0)$ and any $n\in[2,4]$, there holds:
\begin{equation*}
\dfrac{\varepsilon^3}{\varepsilon^{\alpha}(\varepsilon-\varepsilon^{\alpha})}\int_{\partial\mathcal{N}_{\varepsilon}^{\mathcal{T}}}|Q|^n\emph{d}\sigma \lesssim \dfrac{n}{2}\cdot \dfrac{\varepsilon^{2-\alpha}}{1-\varepsilon^{\alpha-1}}\int_{\Omega_{\varepsilon}}\big(|Q|^{2n-2}+|\nabla Q|^2\big)\emph{d}x + \dfrac{1}{2(1-\varepsilon^{\alpha-1})^2}\int_{\Omega_{\varepsilon}}|Q|^n \emph{d}x.
\end{equation*}
\end{lemma}

\begin{proof}
Let $I_{\varepsilon}[Q]=\dfrac{\varepsilon^3}{\varepsilon^{\alpha}(\varepsilon-\varepsilon^{\alpha})}\displaystyle{\int_{\partial\mathcal{N}_{\varepsilon}^{\mathcal{T}}}|Q|^n\text{d}\sigma}$ and $$I_{\varepsilon}^{X}[Q]=\dfrac{\varepsilon^{3}}{\varepsilon^{\alpha}(\varepsilon-\varepsilon^{\alpha})}\sum_{k=1}^{X_{\varepsilon}}\int_{\mathcal{T}_x^k} |Q|^n \text{d}\sigma.$$

Let $\overline{e_1}=(1,0,0)^T$, $\overline{e_2}=(0,1,0)^T$, $\overline{e_3}=(0,0,1)^T$ and $k\in\overline{1,X_{\varepsilon}}$. Then, according to previous definitions, $y_{\varepsilon}^{x,k}$ is the center of the parallelipiped $\mathcal{P}_{\varepsilon}^{x,k}$ with the ``contact" faces $\mathcal{T}_x^k$. If this is an ``inner" parallelipiped of the scaffold, then Figure \eqref{fig:crosss} shows a cross section of a neighbourhood of $\mathcal{P}_{\varepsilon}^{x,k}$, surrounding $\mathcal{T}_x^k$, a section which is parallel to the $yOz$ plane and which is passing through $y_{\varepsilon}^{x,k}+\delta\overline{e_1}$, where $\delta\in I_p:=\bigg[-\dfrac{p\varepsilon-\varepsilon^{\alpha}}{2p},\dfrac{p\varepsilon-\varepsilon^{\alpha}}{2p}\bigg]$. If the parallelipiped is next to $\partial\Omega$, then the same argument works, since we have relations \eqref{defn:points1} and \eqref{defn:Y_eps}. 

Let $\mathcal{T}_x^k(\delta)$ be
\begin{equation*}
\mathcal{T}_x^k(\delta)=\bigg\{y_{\varepsilon}^{x,k}+\delta\overline{e_1}+y\overline{e_2}+z\overline{e_3}\;\bigg|\; -\dfrac{\varepsilon^{\alpha}}{2q}\leq y\leq\dfrac{\varepsilon^{\alpha}}{2q};\;-\dfrac{\varepsilon^{\alpha}}{2r}\leq z\leq\dfrac{\varepsilon^{\alpha}}{2r}\bigg\},
\end{equation*}
which represents the centered white rectangle from Figure \eqref{fig:crosss}.

Let $\mathcal{V}_x^k(\delta)$ be
\begin{equation*}
\mathcal{V}_x^k(\delta)=\bigg\{y_{\varepsilon}^{x,k}+\delta\overline{e_1}+y\overline{e_2}+z\overline{e_3}\;\bigg|\; -\varepsilon+\dfrac{\varepsilon^{\alpha}}{2q}\leq y\leq\varepsilon-\dfrac{\varepsilon^{\alpha}}{2q};\;-\varepsilon+\dfrac{\varepsilon^{\alpha}}{2r}\leq z\leq\varepsilon-\dfrac{\varepsilon^{\alpha}}{2r}\bigg\}\setminus\mathcal{T}_x^k(\delta),
\end{equation*}
which represents the darker shaded area from Figure \eqref{fig:crosss}, containing only liquid crystal particles, that is $\mathcal{V}_x^k(\delta)\subset\Omega_{\varepsilon}$, for any $\delta\in I_p$.

In our case, $\mathcal{V}_x^k(\delta)$ plays the role of $b\mathcal{P}\setminus a\mathcal{P}$ from \cref{lemma:surface}.

If for every $\delta\in I_p$, we apply the translation $y_{\varepsilon}^k+\delta\overline{e_1}$ to the origin of the system, then for 
\begin{align*}
\mathcal{P}=\{0\}\times\bigg[-\dfrac{1}{2q},\dfrac{1}{2q}\bigg]\times\bigg[-\dfrac{1}{2r},\dfrac{1}{2r}\bigg],
\end{align*}
we can choose $a=\varepsilon^{\alpha}$, therefore $\varepsilon^{\alpha}\mathcal{P}=\mathcal{T}_x^k(\delta)$. In order to choose $b$, we assume: $\dfrac{b}{2r}\leq \varepsilon-\dfrac{\varepsilon^{\alpha}}{2r}$ and $\dfrac{b}{2q}\leq \varepsilon-\dfrac{\varepsilon^{\alpha}}{2q}$, that is: $b\leq 2q\varepsilon-\varepsilon^{\alpha}$ and $b\leq 2r\varepsilon-\varepsilon^{\alpha}$. Since $p,q,r\geq 1$, we can choose $b=2\varepsilon-\varepsilon^{\alpha}$. In this way, we have $b\mathcal{P}\setminus a\mathcal{P}\subset\mathcal{V}_x^k(\delta)$ and we also have $b\geq a\Leftrightarrow 2\varepsilon-\varepsilon^{\alpha}\geq \varepsilon^{\alpha}\Leftrightarrow \alpha\geq 1$.

\begin{figure}[ht]
\centering
\includegraphics[scale=0.35]{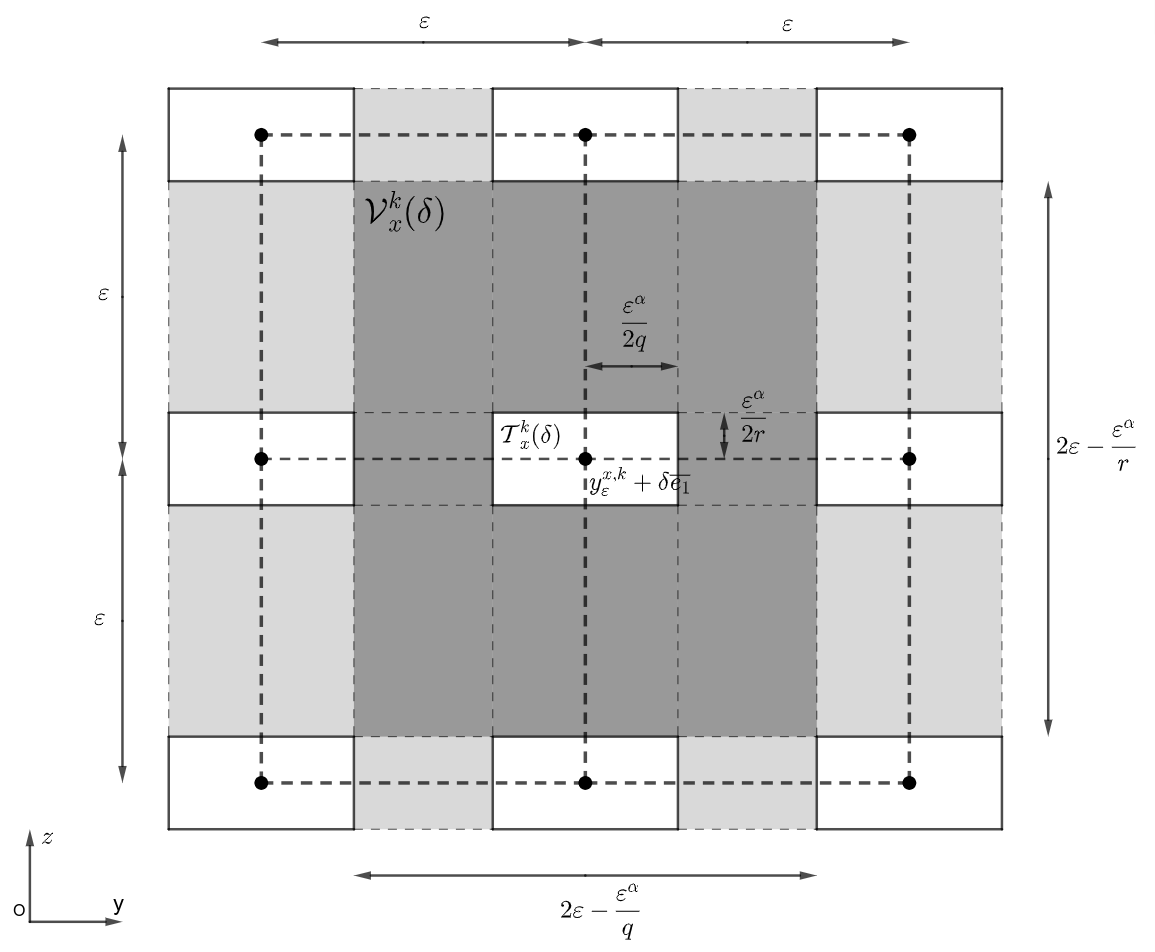}
\caption{Cross section of the scaffold, parallel to $yOz$ plane, passing through $y_{\varepsilon}^{x,k}+\delta\overline{e_1}$. The gray shaded areas represent the liquid crystal and the white rectangles represent the sections of the parts of the scaffold nearby.}
\label{fig:crosss}
\end{figure}

Therefore, we can apply \cref{lemma:bP-aP-ineq} for $Q$ with $a=\varepsilon^{\alpha}$, $b=2\varepsilon-\varepsilon^{\alpha}$ and $\mathcal{P}$ defined as before, hence:
\begin{align*}
\oint_{\partial\mathcal{T}_x^k(\delta)}|Q|^n\text{d}s &\lesssim \dfrac{2\varepsilon^{\alpha}}{(2\varepsilon-\varepsilon^{\alpha})^2-\varepsilon^{2\alpha}}\int_{b\mathcal{P}\setminus\mathcal{T}_x^k(\delta)}|Q|^n\text{d}x+\dfrac{n}{2}\int_{b\mathcal{P}\setminus\mathcal{T}_x^k(\delta)}\big(|Q|^{2n-2}+|\nabla Q|^2\big)\text{d}x
\end{align*}
and since $b\mathcal{P}\subset\mathcal{V}_x^k(\delta)$, we have:
\begin{align*}
\oint_{\partial\mathcal{T}_x^k(\delta)}|Q|^n\text{d}s &\lesssim \dfrac{\varepsilon^{\alpha}}{2\varepsilon(\varepsilon-\varepsilon^{\alpha})}\int_{\mathcal{V}_x^k(\delta)}|Q|^n\text{d}x+\dfrac{n}{2}\int_{\mathcal{V}_x^k(\delta)}\big(|Q|^{2n-2}+|\nabla Q|^2\big)\text{d}x,
\end{align*}
for every $\delta\in I_p$. Integrating now with respect to $\delta$ over $I_p$, we get:
\begin{align*}
\int_{I_p}\bigg(\oint_{\partial\mathcal{T}_x^k(\delta)}|Q|^n\text{d}s\bigg)\text{d}\delta &\lesssim \dfrac{\varepsilon^{\alpha}}{2\varepsilon(\varepsilon-\varepsilon^{\alpha})}\int_{I_p}\bigg(\int_{\mathcal{V}_x^k(\delta)}|Q|^n\text{d}x\bigg)\text{d}\delta+\dfrac{n}{2}\int_{I_p}\bigg(\int_{\mathcal{V}_x^k(\delta)}\big(|Q|^{2n-2}+|\nabla Q|^2\big)\text{d}x\bigg)\text{d}\delta\\
\int_{\mathcal{T}_x^k}|Q|^n\text{d}\sigma &\lesssim \dfrac{\varepsilon^{\alpha}}{2\varepsilon(\varepsilon-\varepsilon^{\alpha})}\int_{\mathcal{U}_x^k}|Q|^n\text{d}x+\dfrac{n}{2}\int_{\mathcal{U}_x^k}\big(|Q|^{2n-2}+|\nabla Q|^2\big)\text{d}x,
\end{align*}
where $\mathcal{U}_x^k:=\bigcup_{\delta\in I_p}\mathcal{V}_x^k(\delta)\subset\Omega_{\varepsilon}$ is now a three dimensional object. Hence:
\begin{align*}
\dfrac{\varepsilon^{3-\alpha}}{\varepsilon-\varepsilon^{\alpha}}\int_{\mathcal{T}_x^k}|Q|^n\text{d}\sigma &\lesssim \dfrac{n}{2}\cdot \dfrac{\varepsilon^{2-\alpha}}{1-\varepsilon^{\alpha-1}}\int_{\mathcal{U}_x^k}\big(|Q|^{2n-2}+|\nabla Q|^2\big)\text{d}x+\dfrac{1}{2(1-\varepsilon^{\alpha-1})^2}\int_{\mathcal{U}_x^k}|Q|^n\text{d}x.
\end{align*}

Repeating the same argument for all the other parallelipipeds of type $\mathcal{T}$ and considering the fact that parts of $\mathcal{U}_x^k$ are added only up to four times (by constructing the same sets for the nearby particles from the scaffold), then the conclusion follows.
\end{proof}

Since we are interested in the homogenised material, it is useful to consider maps defined on the entire $\Omega$ and for this we use the harmonic extension operator $E_{\varepsilon}:H^1_g(\Omega_{\varepsilon},\mathcal{S}_0)\rightarrow H^1_g(\Omega,\mathcal{S}_0)$, defined as follows: for $Q\in H^1_g(\Omega_{\varepsilon},\mathcal{S}_0)$, we take $E_{\varepsilon}Q\equiv Q$ in $\Omega_{\varepsilon}$ and inside $\mathcal{N}_{\varepsilon}$, $E_{\varepsilon}Q$ solves the following PDE:

\begin{equation}\label{eq:extensionpb}
\left\{
\begin{array}{ll}
\Delta E_{\varepsilon}Q=0 & \text{in}\;\mathcal{N}_{\varepsilon}\\
E_{\varepsilon}Q\equiv Q & \text{on}\; \partial\mathcal{N}_{\varepsilon}
\end{array}
\right. 
\end{equation}

Since $\mathcal{N}_{\varepsilon}$ has a Lipschitz boundary, we can apply Theorem 4.19 from \cite{CioranescuDonato} and see that there exists a unique solution $E_{\varepsilon}Q\in H^{1}(\mathcal{N}_{\varepsilon})$ to the problem \eqref{eq:extensionpb}. Hence the operator $E_{\varepsilon}$ is well defined. Moreover, from \eqref{eq:extensionpb}, we can see that $E_{\varepsilon}Q$ verifies:
\begin{equation}\label{eq:extensionmin}
\|\nabla E_{\varepsilon}Q\|_{L^2(\mathcal{N}_{\varepsilon})}=\text{min}\big\{\|\nabla u\|_{L^2(\mathcal{N}_{\varepsilon})}\;\big|\;u\in H^1(\mathcal{N}_{\varepsilon}),\;u=Q\;\text{on}\;\partial\mathcal{N}_{\varepsilon}\big\}.
\end{equation}

Our aim is now to prove that the extension operator $E_{\varepsilon}$ is uniformly bounded with respect to $\varepsilon>0$. More specifically, we prove that the following lemma holds.

\begin{lemma}\label{lemma:extensionineq}
There exists a constant $C>0$ such that $\|\nabla E_{\varepsilon}Q\|_{L^2(\Omega)}\leq C\|\nabla Q\|_{L^2(\Omega_{\varepsilon})}$ for any $\varepsilon\in(0,\varepsilon_0)$, where $\varepsilon_0$ is suitably small enough, and for any $Q\in H^1_g(\Omega_{\varepsilon},\mathcal{S}_0)$. 
\end{lemma}

\begin{proof}
By \cref{subsection:existence_of_extension}, we know that there exists $v\in H^1(\Omega)$ such that:
\begin{align*}
\begin{cases}
v\equiv Q\;\text{in}\;\Omega_{\varepsilon}\\
v=Q\;\text{on}\;\partial\mathcal{N}_{\varepsilon}\\
\big\|\nabla v\big\|_{L^2(\Omega)}\lesssim \big\|\nabla Q\big\|_{L^2(\Omega_{\varepsilon})}.
\end{cases}
\end{align*}

Using relation \eqref{eq:extensionmin}, we see that
\begin{align*}
\big\|\nabla E_{\varepsilon}Q\big\|_{L^2(\mathcal{N}_{\varepsilon})}\leq \big\|\nabla v\big\|_{L^2(\mathcal{N}_{\varepsilon})}
\end{align*}
and because $E_{\varepsilon}Q\equiv Q$ in $\Omega_{\varepsilon}$, we have $E_{\varepsilon}Q\equiv v\equiv Q$ in $\Omega_{\varepsilon}$ and therefore:
\begin{align*}
\big\|\nabla E_{\varepsilon}Q\big\|_{L^2(\Omega)}\leq \big\|\nabla v\big\|_{L^2(\Omega)}\lesssim \big\|\nabla Q\big\|_{L^2(\Omega_{\varepsilon})}.
\end{align*}
\end{proof}

\subsection{Zero contribution from the surface terms depending on $\mathcal{S}^i$}\label{subsection:nocontribution}

In this section, we prove that the surface term $J_{\varepsilon}^{\mathcal{S}}$ has a neglectable contribution to the homogenised material, that is $J_{\varepsilon}^{\mathcal{S}}[Q]\rightarrow 0$ as $\varepsilon\rightarrow 0$, for any $Q\in H^{1}_{g}(\Omega,\mathcal{S}_0)$, since we can use the extension operator $E_{\varepsilon}$ defined in the previous subsection.

We start by proving if $Q:\overline{\Omega}\rightarrow\mathcal{S}_0$ is a bounded, Lipschitz map, then $J_{\varepsilon}^{\mathcal{S}}[Q]\rightarrow 0$ as $\varepsilon\rightarrow 0$ and then, by a density argument, for all $Q\in H^1_g(\Omega,\mathcal{S}_0)$.

\begin{lemma}\label{lemma:zerocontrbounded}
Let $Q:\overline{\Omega}\rightarrow\mathcal{S}_0$ be a bounded, Lipschitz map. Then $J_{\varepsilon}^{\mathcal{S}}[Q]\rightarrow 0$, as $\varepsilon\rightarrow 0$, where $J_{\varepsilon}^{\mathcal{S}}$ is defined in \eqref{defn:Jeps} and in \eqref{defn:Jeps_S}.
\end{lemma}

\begin{proof}
By \eqref{defn:Jeps_S}, we have:
\begin{align*}
\bigg|\dfrac{\varepsilon^{3-\alpha}}{\varepsilon-\varepsilon^{\alpha}}\sum_{i=1}^{N_{\varepsilon,2}}\int_{\mathcal{S}^i}f_s(Q(t),\nu)d\sigma(t)\bigg| & \leq \dfrac{\varepsilon^{3-\alpha}}{\varepsilon-\varepsilon^{\alpha}}\sum_{i=1}^{N_{\varepsilon,2}}\int_{\mathcal{S}^i}\big|f_s(Q(t),\nu)\big|d\sigma(t)\\
&\leq \dfrac{C\varepsilon^{3-\alpha}}{\varepsilon-\varepsilon^{\alpha}}\sum_{i=1}^{N_{\varepsilon,2}}\int_{\mathcal{S}^i}\big(|Q|^4(t)+1\big)d\sigma(t) \\
&\leq \dfrac{C\varepsilon^{3-\alpha}}{\varepsilon-\varepsilon^{\alpha}}\sum_{i=1}^{N_{\varepsilon,2}}\int_{\partial\mathcal{C}^i_{\varepsilon}}\big(|Q|^4(t)+1\big)d\sigma(t) \\
&\leq \dfrac{\varepsilon^{3+\alpha}}{\varepsilon-\varepsilon^{\alpha}}\cdot\dfrac{2C(p+q+r)}{pqr}\cdot\big(\|Q\|_{L^{\infty}(\overline{\Omega})}^4+1\big)\cdot\sum_{i=1}^{N_{\varepsilon,2}}\int_{\partial\mathcal{C}}d\sigma(t)\\
&\leq \dfrac{\varepsilon^{3+\alpha}}{\varepsilon-\varepsilon^{\alpha}}\cdot\big(\|Q\|^4_{L^{\infty}(\overline{\Omega})}+1\big)\cdot\dfrac{2C(p+q+r)}{pqr}\cdot\sigma(\partial\mathcal{C})\cdot N_{\varepsilon,2},\\
\end{align*}
where $\partial\mathcal{C}$ represents the surface of the model particle $\mathcal{C}$ defined in \eqref{defn:initial_cube}, $\mathcal{C}^{i}_{\varepsilon}$ represents the parallelipipeds constructed in relation \eqref{defn:ceps}, $N_{\varepsilon,2}$ is defined in \eqref{defn:N_eps_2} and $C$ is the $\varepsilon$-independent constant given from the inequality that states that $f_s$ has a quartic growth in $Q$, which can be obtained from assumption ($A_7$). We have also used that $Q$ is bounded on $\overline{\Omega}$. In the proof of \cref{prop:outer_surfaces_go_to_0}, we obtain $N_{\varepsilon,2}\leq\dfrac{L_0l_0+l_0h_0+h_0L_0}{\varepsilon^2}$, hence:
\begin{align*}
\bigg|\dfrac{\varepsilon^{3-\alpha}}{\varepsilon-\varepsilon^{\alpha}}\sum_{i=1}^{N_{\varepsilon,2}}\int_{\mathcal{S}^i}f_s(Q(t),\nu)d\sigma(t)\bigg| &< C'\cdot\dfrac{\varepsilon^{3+\alpha}}{\varepsilon-\varepsilon^{\alpha}}\cdot\dfrac{L_0l_0+L_0h_0+l_0h_0}{\varepsilon^{2}}\Rightarrow\\
\Rightarrow\bigg|\dfrac{\varepsilon^{3-\alpha}}{\varepsilon-\varepsilon^{\alpha}}\sum_{i=1}^{N_{\varepsilon,2}}\int_{\mathcal{S}^i}f_s(Q(t),\nu)d\sigma(t)\bigg| &< C''\cdot\dfrac{\varepsilon^{\alpha}}{1-\varepsilon^{\alpha-1}}\rightarrow 0\;\text{as}\;\varepsilon\rightarrow 0,
\end{align*}
since $\alpha\in\bigg(1,\dfrac{3}{2}\bigg)$, where $L_0$, $l_0$ and $h_0$ are defined in \eqref{defn:L_0_l_0_h_0} and $C'$ and $C''$ are $\varepsilon$-independent constants.

\end{proof}

\begin{lemma}
For any $Q\in H^1_g(\Omega_{\varepsilon},\mathcal{S}_0)$, we have $J_{\varepsilon}^{\mathcal{S}}[Q]\rightarrow 0$ as $\varepsilon\rightarrow 0$.
\end{lemma}
\begin{proof}
Let $(Q_j)_{j\geq 1}$ be a sequence of smooth maps that converge strongly in $H^1_g(\Omega_{\varepsilon},\mathcal{S}_0)$ to $Q$. By \cref{lemma:zerocontrbounded}, we have $J_{\varepsilon}^{\mathcal{S}}[Q_j]\rightarrow 0$ as $\varepsilon\rightarrow 0$, for any $j\geq 1$. By assumption ($A_7$), we have on $\mathcal{N}_{\varepsilon}^{\mathcal{S}}$:
\begin{align*}
|f_s(Q_j,\nu)-f_s(Q,\nu)|&\leq |Q_j-Q|\big(|Q_j|^3+|Q|^3+1\big)\\
&\lesssim |Q_j-Q|\big(|Q_j-Q|^3+|Q|^3+1\big)\\
&\lesssim |Q_j-Q|^4+|Q_j-Q|\big(|Q|^3+1\big)
\end{align*}

Thanks to the continuity of the trace operator from $H^1(\Omega_{\varepsilon})$ to $H^{1/2}(\partial\Omega_{\varepsilon})$, the Sobolev embedding \linebreak $H^{1/2}(\partial\Omega_{\varepsilon})\hookrightarrow L^4(\partial\Omega_{\varepsilon})$ and the strong convergence $Q_j\rightarrow Q$ in $H^1_g(\Omega_{\varepsilon},\mathcal{S}_0)$, we get that $Q_j\rightarrow Q$ a.e. on $\partial\mathcal{N}_{\varepsilon}^{\mathcal{S}}$, since $\partial\mathcal{N}_{\varepsilon}^{\mathcal{S}}\subset\partial\Omega_{\varepsilon}$. Therefore, there exists $\psi\in L^4(\partial\mathcal{N}_{\varepsilon}^{\mathcal{S}})$ such that $|Q_j-Q|\leq \psi$ a.e. in $\partial\mathcal{N}_{\varepsilon}^{\mathcal{S}}$ and we can write:
\begin{align}\label{eq:noname}
|f_s(Q_j,\nu)-f_s(Q,\nu)|&\lesssim \psi^4+\psi\big(|Q|^3+1)
\end{align}
on $\partial\mathcal{N}_{\varepsilon}^{\mathcal{S}}$, for every $j\geq 1$. 

At the same time, we have the compact Sobolev embedding $H^{1/2}(\partial\Omega_{\varepsilon})\hookrightarrow L^3(\partial\Omega_{\varepsilon})$, therefore $|Q|^3$ is in $L^1(\partial\mathcal{N}_{\varepsilon}^{\mathcal{S}})$. Hence, the right hand side from \eqref{eq:noname} is in $L^1(\partial\mathcal{N}_{\varepsilon}^{\mathcal{S}})$ and we can apply the Lebesgue dominated convergence theorem and get:
\begin{align*}
\lim_{j\rightarrow +\infty}\int_{\partial\mathcal{N}_{\varepsilon}^{\mathcal{S}}}|f_s(Q_j,\nu)-f_s(Q,\nu)|\text{d}\sigma&=0,
\end{align*}
for any $\varepsilon>0$ fixed.

Now, because for $\varepsilon\rightarrow 0$ we get $|\partial\mathcal{N}_{\varepsilon}^{\mathcal{S}}|\rightarrow 0$ (according to \cref{prop:outer_surfaces_go_to_0}) and $\dfrac{\varepsilon^{3-\alpha}}{\varepsilon-\varepsilon^{\alpha}}\rightarrow 0$, the conclusion follows.
\end{proof}

Therefore, from now on we omit the term $J_{\varepsilon}^{\mathcal{S}}$ from the free energy functional and we only study the behaviour of:
\begin{align*}
\mathcal{F}_{\varepsilon}^{\mathcal{T}}[Q]:=\int_{\Omega_{\varepsilon}}\big(f_e(\nabla Q)+f_b(Q)\big)\text{d}x+\dfrac{\varepsilon^{3-\alpha}}{\varepsilon-\varepsilon^{\alpha}}\int_{\partial\mathcal{N}_{\varepsilon}^{\mathcal{T}}}f_s(Q,\nu)\text{d}\sigma,
\end{align*}
which we denote simply by $\mathcal{F}_{\varepsilon}[Q]$, but we keep the same notation for surfaces generated by the scaffold.

\subsection{Equicoercivity of \texorpdfstring{$\mathcal{F}_{\varepsilon}$}{Fe}}

\begin{prop}\label{prop:equicoercivity}
Suppose that the assumptions ($A_1$)-($A_7$) hold and also that there exists $\mu>0$ such that $f_b(Q)\geq\mu|Q|^6-C$, for any $Q\in\mathcal{S}_0$. Let $Q\in H^1_g(\Omega_{\varepsilon},\mathcal{S}_0)$ satisfy $\mathcal{F}_{\varepsilon}[Q]\leq M$, for some $\varepsilon$-independent constant. Then there holds $$\int_{\Omega_{\varepsilon}}|\nabla Q|^2\leq C_M$$ for $\varepsilon>0$ small enough and for some $C_M>0$ depending only on $M$, $f_e$, $f_b$, $f_s$ and $\Omega$.
\end{prop}

\begin{proof}
Assumption $(A_6)$ ensures that $|f_s(Q,\nu)|\lesssim |Q|^4+1$, therefore:
\begin{equation*}
J_{\varepsilon}^{\mathcal{T}}[Q]\geq -C_1\cdot \dfrac{\varepsilon^{3-\alpha}}{\varepsilon-\varepsilon^{\alpha}}\int_{\partial\mathcal{N}_{\varepsilon}^{\mathcal{T}}}(|Q|^4+1)\text{d}\sigma\geq -C_1\cdot \dfrac{\varepsilon^{3-\alpha}}{\varepsilon-\varepsilon^{\alpha}}\int_{\partial\mathcal{N}_{\varepsilon}^{\mathcal{T}}}|Q|^4\text{d}\sigma-C_1\cdot C_s,
\end{equation*}

\noindent according to \cref{prop:C_s}. Using \cref{lemma:surface} with $n=4$, we have:
\begin{equation*}
\dfrac{\varepsilon^{3-\alpha}}{\varepsilon-\varepsilon^{\alpha}}\int_{\partial\mathcal{N}_{\varepsilon}^{\mathcal{T}}}|Q|^4\text{d}\sigma \lesssim \dfrac{2\varepsilon^{2-\alpha}}{(1-\varepsilon^{\alpha-1})}\int_{\Omega_{\varepsilon}}\big(|Q|^6+|\nabla Q|^2\big)\text{d}x+\dfrac{1}{2(1-\varepsilon^{\alpha-1})^2}\int_{\Omega_{\varepsilon}}|Q|^4\text{d}x,
\end{equation*}
hence
\begin{equation*}
J_{\varepsilon}^{\mathcal{T}}[Q]\geq -C_1\cdot C_2\cdot \dfrac{2\varepsilon^{2-\alpha}}{1-\varepsilon^{\alpha-1}}\int_{\Omega_{\varepsilon}}\big(|Q|^6+|\nabla Q|^2\big)\text{d}x-C_1\cdot C_2\cdot\dfrac{1}{2(1-\varepsilon^{\alpha-1})^2}\int_{\Omega_{\varepsilon}}|Q|^4\text{d}x-C_1\cdot C_s.
\end{equation*}

At the same time, from the generalised version of the Hölder's inequality and from the fact that $\Omega$ is bounded, we have
\begin{equation*}
\bigg(\int_{\Omega_{\varepsilon}}|Q|^4\text{d}x\bigg)^{1/4}\leq|\Omega_{\varepsilon}|^{1/12}\cdot\bigg(\int_{\Omega_{\varepsilon}}|Q|^6\text{d}x\bigg)^{1/6}\Rightarrow \int_{\Omega_{\varepsilon}}|Q|^4\text{d}x<|\Omega|^{1/3}\bigg(\int_{\Omega_{\varepsilon}} |Q|^6\text{d}x\bigg)^{2/3}
\end{equation*}
and so
\begin{equation}\label{eq:equi1}
J_{\varepsilon}^{\mathcal{T}}[Q]\geq -C_3\cdot\dfrac{\varepsilon^{2-\alpha}}{1-\varepsilon^{\alpha-1}}\int_{\Omega_{\varepsilon}}\big(|Q|^6+|\nabla Q|^2)\text{d}x-C_3\cdot\dfrac{1}{2(1-\varepsilon^{\alpha-1})^2}\cdot  |\Omega|^{1/3}\bigg(\int_{\Omega_{\varepsilon}}|Q|^6\text{d}x\bigg)^{2/3}-C_3,
\end{equation}
where $C_3=\text{max}\{C_1\cdot C_2,\;C_1\cdot C_s\}$.
\vspace{2mm}

Since $f_b(Q)\geq\mu |Q|^6-C$, $f_e(\nabla Q)\geq \lambda_e^{-1}|\nabla Q|^2$ (according to ($A_5$) and ($A_6$)) and $|\Omega_{\varepsilon}|\leq|\Omega|$, we have:
\begin{equation}\label{eq:equi2}
\int_{\Omega_{\varepsilon}}\big(f_b(Q)+f_e(\nabla Q)\big)\text{d}x\geq \mu\int_{\Omega_{\varepsilon}}|Q|^6\text{d}x+\lambda_e^{-1}\int_{\Omega_{\varepsilon}}|\nabla Q|^2\text{d}x-C|\Omega|
\end{equation}
and because $\mathcal{F}_{\varepsilon}[Q]\leq M$, combining \eqref{eq:equi1} and \eqref{eq:equi2}, we obtain
\begin{equation}\label{eq:equi3}
\bigg(\lambda_e^{-1}-C_3\cdot \dfrac{\varepsilon^{2-\alpha}}{1-\varepsilon^{\alpha-1}}\bigg)\int_{\Omega_{\varepsilon}}|\nabla Q|^2\text{d}x \leq h_{\varepsilon}\bigg(\bigg(\int_{\Omega_{\varepsilon}}|Q|^6\text{d}x\bigg)^{1/3}\bigg),
\end{equation}
where
\begin{equation*}
h_{\varepsilon}(t)=t^2\cdot\big(C_4(\varepsilon)+t\cdot C_5(\varepsilon)\big)+C_6,
\end{equation*}
for any $t\geq 0$, with $C_4(\varepsilon)=\dfrac{C_3\cdot|\Omega|^{1/3}}{2(1-\varepsilon^{\alpha-1})^2}$,  $C_5(\varepsilon)=C_3\cdot \dfrac{\varepsilon^{2-\alpha}}{1-\varepsilon^{\alpha-1}}-\mu$ and $C_6=\big(M+C_3+C|\Omega|\big)$. As $\varepsilon\rightarrow 0$, we have $C_4(\varepsilon)\searrow \dfrac{C_3\cdot|\Omega|^{1/3}}{2}>0$ and $C_5(\varepsilon)\searrow (-\mu)<0$. Hence, for $\varepsilon>0$ small enough, we have:
\begin{align}\label{eq:ineq_C4_C5}
\dfrac{C_3\cdot|\Omega|^{1/3}}{2}<C_4(\varepsilon)<C_3\cdot|\Omega|^{1/3}\;\text{and}\;-\mu<C_5(\varepsilon)<-\dfrac{\mu}{2}<0.
\end{align}

Let $t_0(\varepsilon)$ be the solution of the equation $C_4(\varepsilon)+t\cdot C_5(\varepsilon)=0$. We prove that $h_{\varepsilon}(t)$ is bounded from above on $[0,+\infty)$. Computing the critical points of $h_{\varepsilon}$, it is easy to check that $2t_0(\varepsilon)/3$ is the point in which the function attains its maximum over $[0,+\infty)$, which is:
\begin{align*}
\text{max}\big\{h_{\varepsilon}(t):t\in[0,+\infty)\big\}&=\dfrac{4C_4^3(\varepsilon)}{27C_5^2(\varepsilon)}+C_6<\dfrac{4}{27}\cdot C_3^3\cdot |\Omega|\cdot \dfrac{4}{\mu^2}+C_6,
\end{align*}
using \eqref{eq:ineq_C4_C5}. Therefore, the function $h_{\varepsilon}$ is bounded from above on $[0,+\infty)$.

Using the same arguments we can see that $\lambda_e^{-1}-C_3\cdot\dfrac{\varepsilon^{2-\alpha}}{1-\varepsilon^{\alpha-1}}$ is also bounded from below, away from $0$, for $\varepsilon>0$ small enough, and from here the conclusion follows, based on relation \eqref{eq:equi3}.
\end{proof}

\subsection{Lower semi-continuity of \texorpdfstring{$\mathcal{F}_{\varepsilon}$}{Fe}}

\begin{prop}\label{prop:lsc}
Suppose that the assumptions ($A_1$)-($A_7$) are satisfied. Then, the following statement holds: for any positive $M>0$, there exists $\varepsilon_0(M)>0$ such that for any $\varepsilon\in\big(0,\varepsilon_0(M)\big)$ and for any sequence $(Q_j)_{j\in\mathbb{N}}$ from $H^1(\Omega_{\varepsilon},\mathcal{S}_0)$ that converges $H^1$-weakly to a function $Q\in H^1(\Omega_{\varepsilon},\mathcal{S}_0)$ and which satisfies $\|\nabla Q_j\|_{L^{2}(\Omega_{\varepsilon})}\leq M$ for any $j\in\mathbb{N}$, then $$\mathcal{F}_{\varepsilon}[Q]\leq\liminf_{j\rightarrow +\infty} \mathcal{F}_{\varepsilon}[Q_j].$$
\end{prop}

\begin{proof}
The proof of \cref{prop:lsc} follows the same steps as in \cite{CanevariZarnescu1}. We prove this proposition on each component of $\mathcal{F}_{\varepsilon}$. Before that, let $$\omega=\liminf_{j\rightarrow +\infty}\int_{\Omega_{\varepsilon}}|\nabla Q_j|^2 \text{d}x-\int_{\Omega_{\varepsilon}}|\nabla Q|^2\text{d}x.$$

Since $Q_j\rightharpoonup Q$ in $H^1$, then $\nabla Q_j\rightharpoonup \nabla Q$ in $L^2$, therefore $\omega\geq 0$. Moreover, up to extracting a subsequence, we can assume that 
\begin{equation}\label{eq:Q_j_goes_to_Q_(nabla)}
\int_{\Omega_{\varepsilon}}|\nabla Q_j|^2\text{d}x\rightarrow\int_{\Omega_{\varepsilon}}|\nabla Q|^2\text{d}x+\omega
\end{equation} 
as $j\rightarrow +\infty$.

From the assumption ($A_5$), we have that $f_e$ is strongly convex, that is for $\theta>0$ small enough, $\tilde{f_e}(D):=f_e(D)-\theta|D|^2$ is a convex function from $\mathcal{S}_0\otimes \mathbb{R}^3$ to $[0,+\infty)$. In this case, the functional $\int_{\Omega_{\varepsilon}}\tilde{f_e}(\cdot)\text{d}x$ is lower semicontinuous. Therefore $$\liminf_{j\rightarrow +\infty
}\int_{\Omega_{\varepsilon}}\tilde{f_e}(\nabla Q_j)\text{d}x\geq\int_{\Omega_{\varepsilon}}\tilde{f_e}(\nabla Q)\text{d}x,$$ from which we get 
\begin{equation}\label{eq:liminffe}
\liminf_{j\rightarrow +\infty}\int_{\Omega_{\varepsilon}}f_e(\nabla Q_j)\text{d}x-\int_{\Omega_{\varepsilon}}f_e(\nabla Q)\text{d}x\geq \bigg(\liminf_{j\rightarrow +\infty}\int_{\Omega_{\varepsilon}}\tilde{f_e}(\nabla Q_j)\text{d}x-\int_{\Omega_{\varepsilon}}\tilde{f_e}(\nabla Q)\text{d}x\bigg)+\theta\omega\geq 0.
\end{equation}

Since $Q_j\rightharpoonup Q$ in $H^1(\Omega_{\varepsilon})$ and the injection $H^1(\Omega_{\varepsilon})\subset L^2(\Omega_{\varepsilon})$ is compact, then we can assume, up to extracting a subsequence, that $Q_j\rightarrow Q$ a.e. in $\Omega_{\varepsilon}$. Then, from the assumption ($A_6$), we can see that the sequence $\big(f_b(Q_j)\big)_{j\in\mathbb{N}}$ satisfies all the conditions from Fatou's lemma, therefore:
\begin{equation}\label{eq:liminffb}
\liminf_{j\rightarrow +\infty}\int_{\Omega_{\varepsilon}}f_b(Q_j)\text{d}x\geq\int_{\Omega_{\varepsilon}}\liminf_{j\rightarrow +\infty}f_b(Q_j)\text{d}x=\int_{\Omega_{\varepsilon}}f_b(Q)\text{d}x.
\end{equation} 

Regarding the \textit{surface energy}, we split $\partial\mathcal{N}_{\varepsilon}^{\mathcal{T}}$ into:
\begin{align*}
A_j&=\{x\in\partial\mathcal{N}_{\varepsilon}^{\mathcal{T}}\;:\;|Q_j(x)-Q(x)|\leq |Q(x)|+1\}\\
B_j=\partial\mathcal{N}_{\varepsilon}^{\mathcal{T}}\setminus A_j&=\{x\in\partial\mathcal{N}_{\varepsilon}^{\mathcal{T}}\;:\;|Q_j(x)-Q(x)|>|Q(x)|+1\},
\end{align*}
for any $j\in\mathbb{N}$.

Using ($A_7$), we have
\begin{align*}
\int_{A_j}|f_s(Q_j,\nu)-f_s(Q,\nu)|\text{d}\sigma&\leq \int_{A_j}\big(|Q_j|^3+|Q|^3+1\big)\cdot |Q_j-Q|\text{d}\sigma\\
&\leq \int_{A_j}\big((|Q_j-Q|+|Q|)^3+|Q|^3+1\big)\cdot\big(|Q|+1\big)\text{d}\sigma\\
&\lesssim \int_{A_j}(|Q|^3+1)(|Q|+1)\text{d}\sigma\lesssim \int_{A_j}(|Q|^4+1)\text{d}\sigma.
\end{align*}

Then due to the continuous embedding of $H^{1/2}(\partial\mathcal{N}_{\varepsilon})$ into $L^{4}(\partial\mathcal{N}_{\varepsilon})$:
\begin{align*}
\dfrac{\varepsilon^{3-\alpha}}{\varepsilon-\varepsilon^{\alpha}}\int_{A_j}|f_s(Q_j,\nu)-f_s(Q,\nu)|\text{d}\sigma\lesssim \dfrac{\varepsilon^{3-\alpha}}{\varepsilon-\varepsilon^{\alpha}}\int_{A_j}\big(|Q|^4+1\big)\text{d}\sigma<+\infty,
\end{align*}
according also to \cref{prop:C_s}. At the same time, the compact embedding $H^{1/2}(\partial\mathcal{N}_{\varepsilon})\hookrightarrow L^{2}(\partial\mathcal{N}_{\varepsilon})$ and the continuity of the trace operator from $H^{1}(\Omega_{\varepsilon})$ into $H^{1/2}(\partial\mathcal{N}_{\varepsilon})$ grants that $Q_j\rightarrow Q$ a.e. on $\partial\mathcal{N}_{\varepsilon}$, up to extracting a subsequence. We can now apply the dominated convergence theorem and get:
\begin{equation}\label{eq:liminfaj}
\dfrac{\varepsilon^{3-\alpha}}{\varepsilon-\varepsilon^{\alpha}}\int_{A_j}|f_s(Q_j,\nu)-f_s(Q,\nu)|\text{d}\sigma\rightarrow 0\;\text{as}\;j\rightarrow +\infty.
\end{equation}

Regarding the $B_j$ sets, we have, according to ($A_7$):
\begin{align*}
|f_s(Q_j,\nu)-f_s(Q,\nu)|&\leq \lambda_s |Q_j-Q|\big(|Q_j|^3+|Q|^3+1\big)\\
& \lesssim |Q_j-Q|\big(|Q_j|^3+|Q_j-Q|^3+1\big)\;\text{(using}\;|Q|+1<|Q_j-Q|)\\
& \lesssim |Q_j-Q|\big(|Q_j-Q|^3+|Q|^3+|Q_j-Q|^3+1\big)\\
& \lesssim |Q_j-Q|\big(|Q_j-Q|^3+1) \lesssim |Q_j-Q|^4.
\end{align*}

Using \cref{lemma:surface} for $(Q_j-Q)$ with $n=4$, we have:
\begin{align*}
\dfrac{\varepsilon^{3-\alpha}}{\varepsilon-\varepsilon^{\alpha}}\int_{B_j}|f_s(Q_j,\nu)-f_s(Q,\nu)|\text{d}\sigma &\lesssim \dfrac{\varepsilon^{3-\alpha}}{\varepsilon-\varepsilon^{\alpha}}\int_{B_j}|Q_j-Q|^4\text{d}\sigma\lesssim \dfrac{\varepsilon^{3-\alpha}}{\varepsilon-\varepsilon^{\alpha}}\int_{\partial\mathcal{N}_{\varepsilon}^{\mathcal{T}}}|Q_j-Q|^4\text{d}\sigma
\end{align*}
\begin{equation*}
\lesssim \dfrac{2\varepsilon^{2-\alpha}}{1-\varepsilon^{\alpha-1}}\int_{\Omega_{\varepsilon}}|Q_j-Q|^6+|\nabla Q_j-\nabla Q|^2\text{d}x+\dfrac{1}{2(1-\varepsilon^{\alpha-1})^2}\int_{\Omega_{\varepsilon}}|Q_j-Q|^4\text{d}x.
\end{equation*}

Since $\text{H}^1(\Omega_{\varepsilon})$ is compactly embedded into $\text{L}^{4}(\Omega_{\varepsilon})$ and $Q_j\rightharpoonup Q$ in $\text{H}^1(\Omega_{\varepsilon})$, then $Q_j\rightarrow Q$ in $\text{L}^4(\Omega_{\varepsilon})$ and so 
\begin{equation*}
\dfrac{1}{2(1-\varepsilon^{\alpha-1})^2}\int_{\Omega_{\varepsilon}}|Q_j-Q|^4\text{d}x\rightarrow 0,\;\text{as}\;j\rightarrow +\infty.
\end{equation*}

For the term containing $|Q_j-Q|^6$, we proceed in the following way:
\begin{align*}
\int_{\Omega_{\varepsilon}}|Q_j-Q|^6\text{d}x &=\int_{\Omega_{\varepsilon}}|E_{\varepsilon}(Q_j-Q)|^6\text{d}x \leq \int_{\Omega}|E_{\varepsilon}(Q_j-Q)|^6\text{d}x=\big\|E_{\varepsilon}(Q_j-Q)\big\|^6_{L^{6}(\Omega)}\\
&\lesssim \big\|E_{\varepsilon}(Q_j-Q)\big\|^6_{\text{H}^1(\Omega)}\;\text{(by the continuous injection}\;\text{H}^1(\Omega)\subset \text{L}^6(\Omega))\\
&\lesssim \big\|E_{\varepsilon}(Q_j-Q)\big\|^6_{\text{H}^1_0(\Omega)}\;\text{(because}\;Q_j\equiv Q\;\text{on}\;\partial\Omega)\\
&\lesssim \big\|\nabla E_{\varepsilon}(Q_j-Q)\big\|^6_{\text{L}^{2}(\Omega)} = \bigg(\int_{\Omega}|\nabla E_{\varepsilon}(Q_j-Q)|^2\text{d}x\bigg)^3\\
&\lesssim \bigg(\int_{\Omega_{\varepsilon}}|\nabla Q_j-\nabla Q|^2\text{d}x\bigg)^3\;\text{(using}\;\text{\cref{lemma:extensionineq}}).
\end{align*}

Now, because $\|\nabla Q_j\|_{\text{L}^2(\Omega_{\varepsilon})}\leq M$, then:
\begin{align*}
\int_{\Omega_{\varepsilon}}|\nabla Q_j-\nabla Q|^2\text{d}x &\leq \int_{\Omega_{\varepsilon}} \big(|\nabla Q_j|^2+|\nabla Q|^2\big)\text{d}x \lesssim M^2
\end{align*}
and therefore
\begin{equation}\label{eq:liminfbj}
\dfrac{\varepsilon^{3-\alpha}}{\varepsilon-\varepsilon^{\alpha}}\int_{B_j}|f_s(Q_j,\nu)-f_s(Q,\nu)|\text{d}\sigma \lesssim \dfrac{2\varepsilon^{2-\alpha}}{1-\varepsilon^{\alpha-1}}(1+M^4)\int_{\Omega_{\varepsilon}}|\nabla Q_j-\nabla Q|^2\text{d}x+o(1).
\end{equation}

Using that $Q_j\rightharpoonup Q$ in $\text{H}^1(\Omega_{\varepsilon})$ and \eqref{eq:Q_j_goes_to_Q_(nabla)}, we obtain that $\displaystyle{\int_{\Omega_{\varepsilon}}|\nabla Q_j-\nabla Q|^2\text{d}x\rightarrow \omega}$ as $j\rightarrow +\infty$ and combining this with \eqref{eq:liminfaj} and \eqref{eq:liminfbj}, we get:
\begin{equation}\label{eq:liminffs}
\liminf_{j\rightarrow +\infty}J_{\varepsilon}^{\mathcal{T}}[Q_j]-J_{\varepsilon}^{\mathcal{T}}[Q]\geq -C_M\cdot\omega\cdot\dfrac{\varepsilon^{2-\alpha}}{1-\varepsilon^{\alpha-1}},
\end{equation}
where $C_M$ is a constant dependent of $M$ and independent of $\varepsilon$.

According to \eqref{eq:liminffe}, \eqref{eq:liminffb} and \eqref{eq:liminffs}, we finally obtain that
\begin{equation*}
\liminf_{j\rightarrow +\infty}\mathcal{F}_{\varepsilon}[Q_j]-\mathcal{F}_{\varepsilon}[Q]\geq \bigg(\theta -C_M\dfrac{\varepsilon^{2-\alpha}}{1-\varepsilon^{\alpha-1}}\bigg)\omega
\end{equation*}
and since $\dfrac{\varepsilon^{2-\alpha}}{1-\varepsilon^{\alpha-1}}\rightarrow 0$ as $\varepsilon\rightarrow 0$, the conclusion follows.

\end{proof}

\section{Convergence of local minimisers}\label{section:conv_local_min}

\subsection{Pointwise convergence of the surface integral}

The aim of this section is to prove the following statement:

\begin{theorem}\label{th:Je_goes_to_J0}
Suppose that the assumptions ($A_1$)-($A_7$) are satisfied. Then, for any bounded, Lipschitz map $Q:\overline{\Omega}\rightarrow\mathcal{S}_0$, there holds $J_{\varepsilon}^{\mathcal{T}}[Q]\rightarrow J_0[Q]$ as $\varepsilon\rightarrow 0$, where 
\begin{align}\label{defn:J_0}
J_0[Q]=\displaystyle{\int_{\Omega}f_{hom}(Q)\emph{d}x}.
\end{align}
\end{theorem}

\begin{proof}
Let us fix a bounded, Lipschitz map $Q:\overline{\Omega}\rightarrow\mathcal{S}_0$ and let $\tilde{J}_{\varepsilon}$ be the following functional:
\begin{equation}\label{defn:J_tilde_eps}
\tilde{J}_{\varepsilon}[Q]=\dfrac{\varepsilon^{3-\alpha}}{\varepsilon-\varepsilon^{\alpha}}\bigg(\sum_{k=1}^{X_{\varepsilon}}\int_{\mathcal{T}_x^k}f_s(Q(y^{x,k}_{\varepsilon}),\nu)\text{d}\sigma+\sum_{l=1}^{Y_{\varepsilon}}\int_{\mathcal{T}_y^l}f_s(Q(y^{y,l}_{\varepsilon}),\nu)\text{d}\sigma+\sum_{m=1}^{Z_{\varepsilon}}\int_{\mathcal{T}_z^m}f_s(Q(y^{z,m}_{\varepsilon}),\nu)\text{d}\sigma\bigg),
\end{equation}
where $y_{\varepsilon}^{x,k}$, $y_{\varepsilon}^{y,l}$ and $y_{\varepsilon}^{z,m}$ are defined in \eqref{defn:Y_eps_x}, \eqref{defn:Y_eps_y} and \eqref{defn:Y_eps_z}, $\mathcal{T}_x^k$, $\mathcal{T}_y^l$ and $\mathcal{T}_z^m$ are defined in \eqref{defn:T_x_k}, \eqref{defn:T_y_l} and \eqref{defn:T_z_m} and $X_{\varepsilon}$, $Y_{\varepsilon}$ and $Z_{\varepsilon}$ are defined in \eqref{defn:Y_eps_x_cardinal}, \eqref{defn:Y_eps_y_cardinal} and \eqref{defn:Y_eps_z_cardinal}.

We prove that $\tilde{J}_{\varepsilon}[Q]\rightarrow J_0[Q]$ and that $\big|J_{\varepsilon}^{\mathcal{T}}[Q]-\tilde{J}_{\varepsilon}[Q]\big|\rightarrow 0$ as $\varepsilon\rightarrow 0$, for any $Q$ with the properties set earlier. 

Let 
\begin{align}\label{defn:PSI}
\begin{cases}
\displaystyle{\Psi^{X}(Q(\tau_0))=\int_{\mathcal{C}^x}f_s(Q(\tau_0),\nu(\tau))\text{d}\sigma(\tau)}\\
\vspace{-4mm}\\
\displaystyle{\Psi^Y(Q(\tau_0))=\int_{\mathcal{C}^y}f_s(Q(\tau_0),\nu(\tau))\text{d}\sigma(\tau)}\\
\vspace{-4mm}\\
\displaystyle{\Psi^Z(Q(\tau_0))=\int_{\mathcal{C}^z}f_s(Q(\tau_0),\nu(\tau))\text{d}\sigma(\tau)}
\end{cases}
\end{align}
for any $\tau_0\in\Omega$, where $\mathcal{C}^x$, $\mathcal{C}^y$ and $\mathcal{C}^z$ are defined in \eqref{defn:C_x_C_y_C_z}. Because $f_s$ is continuous on $\mathcal{S}_0\times\mathbb{S}^2$, then $\Psi^X$, $\Psi^Y$ and $\Psi^Z$ are also continuous. In this case, for example, the first sum from \eqref{defn:J_tilde_eps}, denoted as $\tilde{J}_{\varepsilon}^{X}$, becomes: 
\vspace{-3mm}
\begin{align}\label{defn:J_tilde_eps_nice_form}
\tilde{J}_{\varepsilon}^{X}[Q]&=\dfrac{\varepsilon^{3-\alpha}}{\varepsilon-\varepsilon^{\alpha}}\sum_{k=1}^{X_{\varepsilon}}\int_{\mathcal{T}^k_x}f_s(Q(y^{x,k}_{\varepsilon}),\nu)\text{d}\sigma\notag\\
&=\dfrac{(p-\varepsilon^{\alpha-1})}{pr(1-\varepsilon^{\alpha-1})}\cdot\varepsilon^{3}\sum_{k=1}^{X_{\varepsilon}}\int_{\mathcal{C}^{y}}f_s(Q(y_{\varepsilon}^{x,k}),\nu)\text{d}\sigma+\dfrac{(p-\varepsilon^{\alpha-1})}{pq(1-\varepsilon^{\alpha-1})}\cdot\varepsilon^{3}\sum_{k=1}^{X_{\varepsilon}}\int_{\mathcal{C}^{z}}f_s(Q(y_{\varepsilon}^{x,k}),\nu)\text{d}\sigma\notag\\
&=\dfrac{(p-\varepsilon^{\alpha-1})}{pr(1-\varepsilon^{\alpha-1})}\cdot\int_{\Omega}\Psi^Y(Q(\tau))\text{d}\mu_{\varepsilon}^{X}(\tau)+\dfrac{(p-\varepsilon^{\alpha-1})}{pq(1-\varepsilon^{\alpha-1})}\cdot\int_{\Omega}\Psi^Z(Q(\tau))\text{d}\mu_{\varepsilon}^{X}(\tau),
\end{align}
where $\mu_{\varepsilon}^{X}$ is defined in \eqref{defn:measures}, that is, assumption ($A_4$).

According to ($A_4$), as $\varepsilon\rightarrow 0$, $\mu_{\varepsilon}^{X}$ converges weakly* to the Lebesgue measure restricted to $\Omega$ and because $\Psi^Y$ and $\Psi^Z$ are continuous, then:
\begin{equation*}
\tilde{J}_{\varepsilon}^{X}[Q]\rightarrow\dfrac{1}{r}\int_{\Omega}\Psi^Y(Q(\tau))\text{d}\tau+\dfrac{1}{q}\int_{\Omega}\Psi^Z(Q(\tau))\text{d}\tau.
\end{equation*}

Computing in a similar way $\tilde{J}_{\varepsilon}^y$ and $\tilde{J}_{\varepsilon}^z$, we get:
\begin{equation*}
\tilde{J}_{\varepsilon}[Q]\rightarrow \int_{\Omega}\bigg(\dfrac{q+r}{qr}\Psi^X(Q(\tau))+\dfrac{p+r}{pr}\Psi^Y(Q(\tau))+\dfrac{p+q}{pq}\Psi^Z(Q(\tau))\bigg)\text{d}\tau=\int_{\Omega}f_{hom}(Q(\tau))\text{d}\tau=J_0[Q]
\end{equation*}
which implies that $\tilde{J}_{\varepsilon}[Q]\rightarrow J_0[Q]$. For $J_{\varepsilon}^{X}[Q]$ and $\tilde{J}_{\varepsilon}^{X}[Q]$, we have:
\begin{align}\label{eq:J_eps_J_eps_tilde_control1}
\big|J_{\varepsilon}^{X}[Q]-\tilde{J}_{\varepsilon}^{X}[Q]\big|&\leq \dfrac{\varepsilon^{3-\alpha}}{\varepsilon-\varepsilon^{\alpha}}\sum_{k=1}^{X_{\varepsilon}}\int_{\mathcal{T}_x^k}\big|f_s(Q(\tau),\nu(\tau))-f_s(Q(y^{x,k}_{\varepsilon}),\nu(\tau))\big|\text{d}\sigma(\tau)\notag\\
&\lesssim \dfrac{\varepsilon^{3-\alpha}}{\varepsilon-\varepsilon^{\alpha}}\sum_{k=1}^{X_{\varepsilon}}\int_{\mathcal{T}_x^k}\big(|Q(\tau)|^3+|Q(y^{x,k}_{\varepsilon})|^3+1\big)|Q(\tau)-Q(y^{x,k}_{\varepsilon})|\text{d}\sigma(\tau)\notag\\
&\lesssim \dfrac{\varepsilon^{3-\alpha}}{\varepsilon-\varepsilon^{\alpha}}\cdot\big(\|Q\|^3_{L^{\infty}(\overline{\Omega})}+1\big)\cdot \text{Lip}(Q)\cdot \sum_{k=1}^{X_{\varepsilon}} \sigma(\mathcal{T}_x^k)\cdot \text{diam}(\mathcal{T}_x^k),
\end{align}
using that $Q$ is bounded on $\overline{\Omega}$ and where $\text{Lip}(Q)$ is the Lipschitz constant of $Q$, $\sigma(\mathcal{T}_x^k)$ is the total area of $\mathcal{T}_x^k$ and $\text{diam}(\mathcal{T}_x^k)$ is the diameter of $\mathcal{T}_x^k$, which coincides with the diameter of the parallelipiped $\mathcal{P}_{\varepsilon}^{x,k}$, defined in \eqref{defn:P_x}. Hence 
\begin{align}\label{eq:J_eps_J_eps_tilde_control2}
\big|J_{\varepsilon}^{X}[Q]-\tilde{J}_{\varepsilon}^{X}[Q]\big|&\lesssim\dfrac{2(r+q)}{pqr}\cdot  \dfrac{p\varepsilon-\varepsilon^{\alpha}}{\varepsilon-\varepsilon^{\alpha}}\cdot X_{\varepsilon}\cdot\varepsilon^{3}\cdot\sqrt{\bigg(\varepsilon-\dfrac{\varepsilon^{\alpha}}{p}\bigg)^2+\bigg(\dfrac{\varepsilon^{\alpha}}{q}\bigg)^2+\bigg(\dfrac{\varepsilon^{\alpha}}{r}\bigg)^2}
\end{align}

Now, as $\varepsilon\rightarrow 0$, we have: $\sqrt{\bigg(\varepsilon-\dfrac{\varepsilon^{\alpha}}{p}\bigg)^2+\bigg(\dfrac{\varepsilon^{\alpha}}{q}\bigg)^2+\bigg(\dfrac{\varepsilon^{\alpha}}{r}\bigg)^2}\rightarrow 0$; $\dfrac{p\varepsilon-\varepsilon^{\alpha}}{\varepsilon-\varepsilon^{\alpha}}=\dfrac{p-\varepsilon^{\alpha-1}}{1-\varepsilon^{1-\alpha}}\rightarrow p$ because $1<\alpha$; $X_{\varepsilon}\cdot\varepsilon^{3}<\bigg(\dfrac{L_0}{\varepsilon}-1\bigg)\cdot\dfrac{l_0}{\varepsilon}\cdot\dfrac{h_0}{\varepsilon}\cdot\varepsilon^{3}=L_0l_0h_0-\varepsilon l_0h_0$, according to \cref{prop:volumegoesto0}, where $L_0$, $l_0$ and $h_0$ are defined in \eqref{defn:L_0_l_0_h_0}. Since $X_{\varepsilon}$ is positive, we see that: 
\begin{align}\label{eq:control_over_X_eps}
0\leq \lim_{\varepsilon\rightarrow 0} \big(X_{\varepsilon}\cdot \varepsilon^{3}\big)\leq\lim_{\varepsilon\rightarrow 0}\big(L_0l_0h_0-\varepsilon l_0h_0)=L_0l_0h_0<+\infty.
\end{align}.

Therefore $J_{\varepsilon}^{X}[Q]\rightarrow\tilde{J}_{\varepsilon}^{X}[Q]$ as $\varepsilon\rightarrow 0$. We get the same result for the other two components, from which we conclude.

\end{proof}

\begin{remark}
It is easy to see that if we replace the coefficient $\dfrac{\varepsilon^{3}}{\varepsilon^{\alpha}(\varepsilon-\varepsilon^{\alpha})}$ of the \textit{surface energy} term $J_{\varepsilon}$ with:
\begin{itemize}
\item[•] $\dfrac{\varepsilon^3}{(\varepsilon-\varepsilon^{\alpha})^2}\rightarrow 0$, then $J_{\varepsilon}^{X}[Q]\rightarrow 0$ as $\varepsilon\rightarrow 0$;
\item[•] $\dfrac{\varepsilon^3}{\varepsilon^{2\alpha}}=\varepsilon^{3-2\alpha}\rightarrow 0$, then $J_{\varepsilon}^{X}[Q]\rightarrow +\infty$ as $\varepsilon\rightarrow 0$.
\end{itemize}

In both cases, we lose the convergence $J_{\varepsilon}^{\mathcal{T}}[Q]\rightarrow J_0[Q]$.
\end{remark}

\subsection{$\Gamma$-convergence of the approximating free energies}

\begin{lemma}\label{lemma:boundednabla}
Suppose that the assumption ($A_7$) is satisfied. Let $Q_1$ and $Q_2$ from $H^1_g(\Omega,\mathcal{S}_0)$ be such that 
\begin{equation}\label{eq:controlovernabla}
max\{\|\nabla Q_1\|_{\text{L}^2(\Omega)},\|\nabla Q_2\|_{\text{L}^{2}(\Omega)}\}\leq M
\end{equation}
for some $\varepsilon$-independent constant $M$. Then, for $\varepsilon$ sufficiently small, we have:
\begin{equation}\label{eq:alpha/4}
\big|J_{\varepsilon}^{\mathcal{T}}[Q_2]-J_{\varepsilon}^{\mathcal{T}}[Q_1]\big|\leq C_M\big(\varepsilon^{1/2-\alpha/4}+\|Q_2-Q_1\|_{\text{L}^4(\Omega)}\big)
\end{equation}
for some $C_M>0$ depending only on $M$, $f_s$, $\Omega$, $\mathcal{C}$ and $g$.
\end{lemma}

\begin{proof}
According to ($A_7$) and H\"{o}lder inequality, we have:
\begin{align*}
\big|J_{\varepsilon}^{\mathcal{T}}[Q_2]-J_{\varepsilon}^{\mathcal{T}}[Q_1]\big|&\lesssim \dfrac{\varepsilon^{3-\alpha}}{\varepsilon-\varepsilon^{\alpha}}\int_{\partial\mathcal{N}_{\varepsilon}^{\mathcal{T}}} \big(|Q_1|^3+|Q_2|^3+1\big)|Q_2-Q_1|\text{d}\sigma 
\end{align*}
\begin{equation*}
\lesssim \dfrac{\varepsilon^{3-\alpha}}{\varepsilon-\varepsilon^{\alpha}}\bigg(\int_{\partial\mathcal{N}_{\varepsilon}^{\mathcal{T}}}|Q_2-Q_1|^4\text{d}\sigma\bigg)^{1/4}\bigg(\bigg(\int_{\partial\mathcal{N}_{\varepsilon}^{\mathcal{T}}}|Q_1|^4\text{d}\sigma\bigg)^{3/4}+\bigg(\int_{\partial\mathcal{N}_{\varepsilon}^{\mathcal{T}}}|Q_2|^4\text{d}\sigma\bigg)^{3/4}+|\partial\mathcal{N}_{\varepsilon}^{\mathcal{T}}|^{3/4}\bigg).
\end{equation*}

If we make use of \cref{lemma:surface}, then:
\begin{align*}
\dfrac{\varepsilon^{3-\alpha}}{\varepsilon-\varepsilon^{\alpha}}\int_{\partial\mathcal{N}_{\varepsilon}^{\mathcal{T}}}|Q_i|^4\text{d}\sigma &\lesssim \dfrac{2\varepsilon^{2-\alpha}}{1-\varepsilon^{\alpha-1}}\int_{\Omega_{\varepsilon}}\big(|Q_i|^6+|\nabla Q_i|^2\big)\text{d}x+\dfrac{1}{2(1-\varepsilon^{\alpha-1})^2}\int_{\Omega_{\varepsilon}}|Q_i|^4\text{d}x,
\end{align*}
for any $i\in\{1,2\}$. By the continuous injection $\text{H}^1(\Omega_{\varepsilon})$ into $\text{L}^6(\Omega_{\varepsilon})$, we have 
\begin{align}\label{eq:control_Q_i}
\dfrac{\varepsilon^{3-\alpha}}{\varepsilon-\varepsilon^{\alpha}}\int_{\partial\mathcal{N}_{\varepsilon}^{\mathcal{T}}}|Q_i|^4\text{d}\sigma &\lesssim \dfrac{2\varepsilon^{2-\alpha}}{1-\varepsilon^{\alpha-1}}\bigg(\big\|\nabla Q_i\big\|^2_{\text{L}^2(\Omega_{\varepsilon})} + \big\|Q_i\big\|^6_{\text{H}^{1}(\Omega_{\varepsilon})}\bigg)+\dfrac{1}{2(1-\varepsilon^{\alpha-1})^2}\big\|Q_i\big\|^4_{\text{L}^4(\Omega_{\varepsilon})}.
\end{align}

Using the Poincar\' e inequality as in Theorem 4.4.7, page 193, from \cite{Ziemer}, the compact embedding \linebreak $H^1(\Omega_{\varepsilon})\hookrightarrow L^4(\Omega_{\varepsilon})$ and the fact that $\Omega_{\varepsilon}\subset\Omega$, we get for $\varepsilon$ small enough:
\begin{align*}
\dfrac{\varepsilon^{3-\alpha}}{\varepsilon-\varepsilon^{\alpha}}\int_{\partial\mathcal{N}_{\varepsilon}^{\mathcal{T}}}|Q_i|^4\text{d}\sigma &\lesssim \dfrac{2\varepsilon^{2-\alpha}}{1-\varepsilon^{\alpha-1}}\bigg(\big\|\nabla Q_i\big\|^2_{L^2(\Omega)}+\big\|\nabla Q_i\big\|^6_{L^2(\Omega)}\bigg)+\dfrac{1}{2(1-\varepsilon^{\alpha-1})^2}\big\|\nabla Q_i\big\|^4_{L^2(\Omega)}\\
&\lesssim \dfrac{2\varepsilon^{2-\alpha}}{1-\varepsilon^{\alpha-1}}\big(M^2+M^6\big)+\dfrac{1}{2(1-\varepsilon^{\alpha-1})^2}M^4,
\end{align*}
where we used \eqref{eq:controlovernabla} and we can see that the right-hand side from the last inequality can be bounded in terms of $M$, since $\dfrac{\varepsilon^{2-\alpha}}{1-\varepsilon^{\alpha-1}}\searrow 0$ and $\dfrac{1}{(1-\varepsilon^{\alpha-1})^2}\searrow 1$ as $\varepsilon\rightarrow 0$. But since $\dfrac{1}{(1-\varepsilon^{\alpha-1})^2}\searrow 1$ as $\varepsilon\rightarrow 0$, we can choose $\varepsilon>0$ such that $\dfrac{1}{(1-\varepsilon^{\alpha-1})^2}<2$ and we can move the constant $2$ under the ``$\lesssim$" sign. Hence, the last relation can be written as:
\begin{align*}
\dfrac{\varepsilon^{3-\alpha}}{\varepsilon-\varepsilon^{\alpha}}\int_{\partial\mathcal{N}_{\varepsilon}^{\mathcal{T}}}|Q_i|^4\text{d}\sigma &\lesssim \varepsilon^{2-\alpha}\big(M^2+M^6\big)+M^4.
\end{align*}

 In a similar fashion, using the same arguments as before for \eqref{eq:control_Q_i}, we get in the case of $\big(Q_2-Q_1\big)$:
\begin{align*}
\dfrac{\varepsilon^{3-\alpha}}{\varepsilon-\varepsilon^{\alpha}}\int_{\partial\mathcal{N}_{\varepsilon}^{\mathcal{T}}}\big|Q_2-Q_1\big|^4\text{d}\sigma &\lesssim \varepsilon^{2-\alpha}\big(M^2+M^6\big)+\big\|Q_2-Q_1\big\|^4_{L^4(\Omega)}.
\end{align*}

Using the same bounds as in \eqref{eq:control_Q_i}, we conclude by observing that there exists a constant $C_M>0$ such that:
\begin{equation*}
\big|J_{\varepsilon}^{\mathcal{T}}[Q_2]-J_{\varepsilon}^{\mathcal{T}}[Q_1]\big|\leq C_M\cdot\big(\big(\varepsilon^{2-\alpha}\big)^{1/4}+\|Q_2-Q_1\|_{\text{L}^4(\Omega)}\big)
\end{equation*}
\end{proof}

\begin{lemma}\label{lemma:J_eps_goes_to_J_0_for_any_Q}
For any $Q\in H^1_g(\Omega,\mathcal{S}_0)$, there holds $J_{\varepsilon}^{\mathcal{T}}[Q]\rightarrow J_0[Q]$ as $\varepsilon\rightarrow 0$.
\end{lemma}

\begin{proof}
Let $(Q_j)_{j\geq 1}$ be a sequence of smooth functions that converge strongly to $Q$ in $H^1_g(\Omega,\mathcal{S}_0)$. Then there holds:
\begin{equation*}
\big|J_{\varepsilon}^{\mathcal{T}}[Q]-J_0[Q]\big|\leq\big|J_{\varepsilon}^{\mathcal{T}}[Q]-J_{\varepsilon}^{\mathcal{T}}[Q_j]\big|+\big|J_{\varepsilon}^{\mathcal{T}}[Q_j]-J_0[Q_j]\big|+\big|J_0[Q_j]-J_0[Q]\big|.
\end{equation*}

From \cref{lemma:boundednabla}, we have that
\begin{equation*}
\big|J_{\varepsilon}^{\mathcal{T}}[Q]-J_{\varepsilon}^{\mathcal{T}}[Q_j]\big|\lesssim \varepsilon^{1/2-\alpha/4}+\|Q-Q_j\|_{\text{L}^4(\Omega)}.
\end{equation*}
and we recall that for $\varepsilon\rightarrow 0$ we have $\varepsilon^{1/2-\alpha/4}\rightarrow 0$ because $\alpha\in(1,3/2)$.

Since the $\big(Q_j)_{j\geq 1}$ converge strongly in $H^1_g(\Omega)$, from the compact Sobolev embedding, we get that $Q_j\rightarrow Q$ in $L^4(\Omega)$ as $j\rightarrow +\infty$, therefore $Q_j\rightarrow Q$ a.e. in $\Omega$. 

From \cref{th:Je_goes_to_J0}, we obtain that $J_{\varepsilon}^{\mathcal{T}}[Q_j]\rightarrow J_0[Q_j]$ as $\varepsilon\rightarrow 0$, for any $j\geq 1$.

For the last term, we can write $\big|J_0[Q_j]-J_0[Q]\big|\leq\int_{\Omega}\big|f_{hom}[Q_j]-f_{hom}[Q]\big|\text{d}x$. In here, we have: $f_{hom}$ is continuous, $Q_j\rightarrow Q$ a.e. in $\Omega$ and $f_{hom}$ has a quartic growth in $Q$ (because $f_s$ has the same growth), which implies that: $\big|f_{hom}[Q_j]\big|\lesssim|Q_j|^4+1$. At the same time, we can assume that there exists $\psi \in L^1(\Omega)$ such that $|Q_j|^4\leq \psi$, for any $j\geq 1$, a.e. in $\Omega$. Therefore, we can apply the Lebesgue dominated convergence theorem and get that $J_0[Q_j]\rightarrow J_0[Q]$ as $j\rightarrow +\infty$.

Combining the results from above, we obtain:
\begin{align*}
\limsup_{\varepsilon\rightarrow 0}\big|J_{\varepsilon}^{\mathcal{T}}[Q]-J_{\varepsilon}^{\mathcal{T}}[Q_j]\big|\lesssim \|Q-Q_j\|_{L^4(\Omega)}+\big|J_0[Q_j]-J_0[Q]\big|\rightarrow 0,\hspace{5mm}\text{as }j\rightarrow +\infty,
\end{align*}
from which we conclude.

\end{proof}

We now prove that $\mathcal{F}_{\varepsilon}$ $\Gamma$-converges to $\mathcal{F}_0$ as $\varepsilon\rightarrow 0$, with respect to the weak $H^1$-topology.

\begin{prop}\label{prop:gamma_conv_inf}
Suppose that the assumptions ($A_1$)-($A_7$) are satisfied. Let $Q_{\varepsilon}\in H^1_g(\Omega_{\varepsilon},\mathcal{S}_0)$ be such that $E_{\varepsilon}Q_{\varepsilon}\rightharpoonup Q$ weakly in $H^1(\Omega)$ as $\varepsilon\rightarrow 0$. Then: 
\begin{equation*}
\liminf_{\varepsilon\rightarrow 0}\mathcal{F}_{\varepsilon}[Q_{\varepsilon}]\geq \mathcal{F}_0[Q],\hspace{10mm} \lim_{\varepsilon\rightarrow 0}J_{\varepsilon}^{\mathcal{T}}[Q_{\varepsilon}]=J_0[Q].
\end{equation*}
\end{prop}

\begin{proof}
The proof follows the same steps as in Proposition 4.2. from \cite{CanevariZarnescu1}.

Since $E_{\varepsilon}Q_{\varepsilon}\rightharpoonup Q$ in $H^1(\Omega)$, then $\big(E_{\varepsilon}Q_{\varepsilon}\big)_{\varepsilon >0}$ is a bounded sequence in $H^1(\Omega)$. Therefore, we can choose a subsequence $\big(E_{\varepsilon_j}Q_{\varepsilon_j}\big)_{j\geq 1}\subset\big(E_{\varepsilon}Q_{\varepsilon}\big)_{\varepsilon >0}$ such that 
\begin{align*}
\liminf_{\varepsilon\rightarrow 0}\mathcal{F}_{\varepsilon}[Q_{\varepsilon}]&=\lim_{j\rightarrow +\infty}\mathcal{F}_{\varepsilon_j}[Q_{\varepsilon_j}].
\end{align*}

Furthermore, by the compact embeddings $H^1(\Omega)\hookrightarrow L^s(\Omega)$, with $s\in[1,6)$, we have that $E_{\varepsilon_j}Q_{\varepsilon_j}\rightarrow Q$ strongly in $L^s(\Omega)$, for any $s\in[1,6)$. As a result, we also obtain that $E_{\varepsilon_j}Q_{\varepsilon_j}\rightarrow Q$ a.e. in $\Omega$. We denote the subsequence $E_{\varepsilon_j}Q_{\varepsilon_j}$ as $E_{\varepsilon}Q_{\varepsilon}$ for the ease of notation.

Now, according to ($A_5$), we have:
\begin{align}\label{eq:nabla_f_e}
\int_{\Omega_{\varepsilon}}\big(f_e(\nabla Q_{\varepsilon})-f_e(\nabla Q)\big)\text{d}x &\geq \int_{\Omega_{\varepsilon}}\nabla f_e(\nabla Q):(\nabla Q_{\varepsilon}-\nabla Q)\text{d}x=\notag\\
&=\int_{\Omega}\nabla f_e(\nabla Q):(\nabla Q_{\varepsilon}-\nabla Q)\text{d}x-\int_{\mathcal{N}_{\varepsilon}}\nabla f_e(\nabla Q):(\nabla Q_{\varepsilon}-\nabla Q)\text{d}x \geq\notag\\
&\geq \int_{\Omega}\nabla f_e(\nabla Q):(\nabla Q_{\varepsilon}-\nabla Q)\text{d}x-\big\|\nabla f_e(\nabla Q)\big\|_{L^2(\mathcal{N}_{\varepsilon})}\cdot\big\|\nabla Q_{\varepsilon}-\nabla Q\big\|_{L^2(\mathcal{N}_{\varepsilon})}
\end{align}

Because $Q\in H^1(\Omega)$, then $\nabla Q\in L^2(\Omega)$ and, according to ($A_5$), the relation $\big|\nabla f_e(\nabla Q)\big|\lesssim |\nabla Q|+1$ implies that $\nabla f_e(\nabla Q)\in L^2(\Omega)$. Therefore, by the weak convergence $E_{\varepsilon}Q_{\varepsilon}\rightharpoonup Q$ in $H^1(\Omega)$, the first term from the right hand side in \eqref{eq:nabla_f_e} goes to $0$ as $\varepsilon\rightarrow 0$. The second term goes to $0$ as well thanks additionally to the fact that the volume of the scaffold $\mathcal{N}_{\varepsilon}$ tends to $0$ as $\varepsilon\rightarrow 0$, according to \cref{prop:volumegoesto0}. Hence:
\begin{align}\label{eq:liminf_f_e}
\liminf_{\varepsilon\rightarrow 0}\int_{\Omega_{\varepsilon}}f_e(\nabla Q_{\varepsilon})\text{d}x &\geq \lim_{\varepsilon\rightarrow 0}\int_{\Omega_{\varepsilon}}f_e(\nabla Q)\text{d}x=\int_{\Omega}f_e(\nabla Q)\text{d}x.
\end{align}

For the bulk potential we apply Fatou's lemma, since $f_b(Q_{\varepsilon})\chi_{\Omega_{\varepsilon}}\rightarrow f_b(Q)$ a.e. in $\Omega$ (because $f_b$ is continuous, $E_{\varepsilon}Q_{\varepsilon}\rightarrow Q$ a.e. in $\Omega$ and $|\mathcal{N}_{\varepsilon}|\rightarrow 0$, according to \cref{prop:volumegoesto0}) and $f_b$ is bounded from below (according to ($A_6$)), in order to obtain:
\begin{align}\label{eq:liminf_f_b}
\liminf_{\varepsilon\rightarrow 0}\int_{\Omega_{\varepsilon}}f_b(Q_{\varepsilon})\text{d}x &\geq \int_{\Omega}f_b(Q)\text{d}x.
\end{align}

For the \textit{surface energy}, we first use \cref{lemma:boundednabla} in the following inequality:
\begin{align*}
\big|J_{\varepsilon}^{\mathcal{T}}[E_{\varepsilon}Q_{\varepsilon}]-J_0[Q]\big| &\leq \big|J_{\varepsilon}^{\mathcal{T}}[E_{\varepsilon}Q_{\varepsilon}]-J_{\varepsilon}^{\mathcal{T}}[Q]\big|+\big|J_{\varepsilon}^{\mathcal{T}}[Q]-J_0[Q]\big|\\
& \lesssim \varepsilon^{1/2-\alpha/4}+\big\|E_{\varepsilon}Q_{\varepsilon}-Q\big\|_{L^4(\Omega)}+\big|J_{\varepsilon}^{\mathcal{T}}[Q]-J_0[Q]\big|.
\end{align*}

Since we have $\varepsilon^{1/2-\alpha/4}\rightarrow 0$ for $\varepsilon\rightarrow 0$ (because $\alpha\in(1,3/2)$), then combining the result from \cref{lemma:J_eps_goes_to_J_0_for_any_Q} with the fact that $E_{\varepsilon}Q_{\varepsilon}\rightarrow Q$ strongly in $L^4(\Omega)$, we obtain 
\begin{align}\label{eq:liminf_f_s}
\lim_{\varepsilon\rightarrow 0} J_{\varepsilon}^{\mathcal{T}}[Q_{\varepsilon}]=J_0[Q].
\end{align}

The proof is now complete, considering \eqref{eq:liminf_f_e}, \eqref{eq:liminf_f_b} and \eqref{eq:liminf_f_s}. 
\end{proof}

\begin{prop}\label{prop:gamma_conv_sup}
Suppose that the assumptions ($A_1$)-($A_7$) are verified. Then, for any $Q\in H^1_g(\Omega,\mathcal{S}_0)$, there exists a sequence $\big(Q_{\varepsilon}\big)_{\varepsilon>0}$ such that $Q_{\varepsilon}\in H^1(\Omega_{\varepsilon})$, for any $\varepsilon>0$, $E_{\varepsilon}Q_{\varepsilon}\rightharpoonup Q$ in $H^1(\Omega)$ and:
\begin{align*}
\limsup_{\varepsilon\rightarrow 0}\mathcal{F}_{\varepsilon}[Q_{\varepsilon}]&\leq \mathcal{F}_0[Q].
\end{align*}
The sequence $\big(Q_{\varepsilon}\big)_{\varepsilon>0}$ is called a recovery sequence.
\end{prop}

\begin{proof}
Let us define in this case $Q_{\varepsilon}=Q\cdot\chi_{\Omega_{\varepsilon}}$. Since $|\mathcal{N}_{\varepsilon}|\rightarrow 0$ as $\varepsilon\rightarrow 0$ (according to \cref{prop:volumegoesto0}), then $\chi_{\Omega_{\varepsilon}}\rightarrow 1$ strongly in $L^1(\Omega)$ and we can apply Lebesgue's dominated converge theorem in order to obtain that:
\begin{align*}
\lim_{\varepsilon\rightarrow 0}\int_{\Omega_{\varepsilon}}f_e(\nabla Q_{\varepsilon})+f_b(Q_{\varepsilon})\text{d}x &=\int_{\Omega}f_e(\nabla Q)+f_b(Q)\text{d}x\\
\lim_{\varepsilon\rightarrow 0}\big(\mathcal{F}_{\varepsilon}[Q]-J_{\varepsilon}^{\mathcal{T}}[Q]\big)&=\mathcal{F}_0[Q]-J_0[Q].
\end{align*}

By \cref{prop:gamma_conv_inf}, we have that $\displaystyle{\lim_{\varepsilon\rightarrow 0}J_{\varepsilon}^{\mathcal{T}}[Q_{\varepsilon}]=J_0[Q]}$, hence the conclusion follows.
\end{proof}

\cref{prop:gamma_conv_inf} and \cref{prop:gamma_conv_sup} show that $\mathcal{F}_{\varepsilon}$ $\Gamma$-converges to $\mathcal{F}_0$, as $\varepsilon\rightarrow 0$, with respect to the weak $H^1$ topology.

\subsection{Proof of main theorems}

\begin{proof}[Proof of \cref{th:local_min}]
Let $Q_0$ from $H^1_g(\Omega,\mathcal{S}_0)$ be an isolated $H^1$-local minimiser for $\mathcal{F}_0$, that is, there exists $\delta_0 >0$ such that $\mathcal{F}_0[Q_0]<\mathcal{F}_0[Q]$, for any $Q\in H^1_g(\Omega,\mathcal{S}_0)$, such that $0<\big\|Q-Q_0\big\|_{H^1(\Omega)}\leq \delta_0$.

We would like to prove that for any $\varepsilon>0$, there exists $Q_{\varepsilon}\in H^1_g(\Omega_{\varepsilon},\mathcal{S}_0)$, which is a $H^1$-local minimiser for $\mathcal{F}_{\varepsilon}$, such that $E_{\varepsilon}Q_{\varepsilon}\rightarrow Q_0$ strongly in $H^1_g(\Omega,\mathcal{S}_0)$ as $\varepsilon\rightarrow 0$.

For this, let 
\begin{align*}
\mathcal{B}_{\varepsilon}:=\big\{Q\in H^1_g(\Omega_{\varepsilon},\mathcal{S}_0)\;:\:\big\|E_{\varepsilon}Q-Q_0\big\|_{H^1(\Omega)}\leq \delta_0\big\}.
\end{align*}

Using Mazur's lemma, we can show that the set $\mathcal{B}_{\varepsilon}$ is sequentially weakly closed in $H^1(\Omega_{\varepsilon})$. Then, by \cref{prop:lsc}, we can see that, for $\varepsilon$ small enough, $\mathcal{F}_{\varepsilon}$ is lower semicontinuous on $\mathcal{B}_{\varepsilon}$ and, by \cref{prop:equicoercivity}, is also coercive on $\mathcal{B}_{\varepsilon}$, since any $Q\in\mathcal{B}_{\varepsilon}$ has $\big\|\nabla Q\big\|_{L^2(\Omega_{\varepsilon})}<\big\|\nabla Q_0\big\|_{L^2(\Omega)}+\delta_0$. Hence, for any $\varepsilon$ sufficiently small, the functional $\mathcal{F}_{\varepsilon}$ admits at least one minimiser $Q_{\varepsilon}$ from $\mathcal{B}_{\varepsilon}$.

Firstly, we prove that $E_{\varepsilon}Q_{\varepsilon}\rightharpoonup Q_0$ weakly in $H^1(\Omega)$, as $\varepsilon\rightarrow 0$.

Let $\mathcal{B}_0:=\big\{Q\in H^1_g(\Omega,\mathcal{S}_0)\;:\;\big\|Q-Q_0\big\|_{H^1(\Omega)}\leq\delta_0\big\}$. Because $Q_{\varepsilon}\in\mathcal{B}_{\varepsilon}$, then $(E_{\varepsilon}Q_{\varepsilon})_{\varepsilon>0}$ represents a bounded sequence in $H^1(\Omega)$, hence there exists a subsequence, which we still denote $(E_{\varepsilon}Q_{\varepsilon})_{\varepsilon>0}$ for the ease of notation, that converges weakly to a $\tilde{Q}\in \mathcal{B}_0$. We show that $\tilde{Q}=Q_0$.

Since $E_{\varepsilon}Q_{\varepsilon}\rightharpoonup\tilde{Q}$ in $H^1_g(\Omega,\mathcal{S}_0)$, we can apply \cref{prop:gamma_conv_inf} and get:
\begin{align*}
\mathcal{F}_0[\tilde{Q}]&\leq\liminf_{\varepsilon\rightarrow 0}\mathcal{F}_{\varepsilon}[Q_{\varepsilon}]\leq\limsup_{\varepsilon\rightarrow 0}\mathcal{F}_{\varepsilon}[Q_{\varepsilon}].
\end{align*}

But $Q_{\varepsilon}$ is a minimiser of $\mathcal{F}_{\varepsilon}$ on $\mathcal{B}_{\varepsilon}$, therefore, since $Q_0\big|_{\Omega_{\varepsilon}}\in\mathcal{B}_{\varepsilon}$, we get that
\begin{align*}
\limsup_{\varepsilon\rightarrow 0}\mathcal{F}_{\varepsilon}[Q_{\varepsilon}]&\leq \lim_{\varepsilon\rightarrow 0}\mathcal{F}_{\varepsilon}[Q_0]=\lim_{\varepsilon\rightarrow 0}\bigg(\int_{\Omega_{\varepsilon}}f_e(\nabla Q_0)+f_b(Q)\text{d}x+\mathcal{J}_{\varepsilon}^{\mathcal{T}}[Q_0]\bigg)=\mathcal{F}_0[Q_0].
\end{align*}

Hence, we have $\mathcal{F}_0[\tilde{Q}]\leq\mathcal{F}_0[Q_0]$. Because $\tilde{Q}$ is in $\mathcal{B}_0$, that is $\|\tilde{Q}-Q_0\|_{H^1(\Omega)}\leq\delta_0$, then by the definition of $Q_0$, we get that $\tilde{Q}=Q_0$.

We now prove that $E_{\varepsilon}Q_{\varepsilon}\rightarrow Q_0$ strongly in $H^1(\Omega)$, as $\varepsilon\rightarrow 0$.

By ($A_5$), there exists $\theta >0$ such that the function $\tilde{f_e}(D)=f_e(D)-\theta|D|^2$ is convex. We can repeat the same arguments from \cref{prop:gamma_conv_inf}, more specifically, steps \eqref{eq:liminf_f_e} and \eqref{eq:liminf_f_b}, to get:
\begin{align*}
\liminf_{\varepsilon\rightarrow 0}\int_{\Omega_{\varepsilon}}\tilde{f_e}(\nabla Q_{\varepsilon})\text{d}x &\geq \int_{\Omega}\tilde{f_e}(\nabla Q_0)\text{d}x,\\
\theta\liminf_{\varepsilon\rightarrow 0}\int_{\Omega_{\varepsilon}}|\nabla Q_{\varepsilon}|^2\text{d}x &\geq \theta\int_{\Omega}|\nabla Q_0|^2\text{d}x,\\
\liminf_{\varepsilon\rightarrow 0}\int_{\Omega_{\varepsilon}}f_b(Q_{\varepsilon})\text{d}x &\geq \int_{\Omega}f_b(Q_0)\text{d}x.
\end{align*}

From \cref{prop:gamma_conv_inf}, we have that $J_{\varepsilon}^{\mathcal{T}}[Q_{\varepsilon}]\rightarrow J_0[Q_0]$ as $\varepsilon\rightarrow 0$. Also, from the proof that $\tilde{Q}=Q_0$, we can see that
\begin{align*}
\lim_{\varepsilon\rightarrow 0}\mathcal{F}_{\varepsilon}[Q_{\varepsilon}]=\mathcal{F}_0[Q_0],
\end{align*}
which implies that 
\begin{align*}
\lim_{\varepsilon\rightarrow 0}\int_{\Omega_{\varepsilon}}|\nabla Q_{\varepsilon}|^2\text{d}x=\int_{\Omega}|\nabla Q_0|^2\text{d}x.
\end{align*} 

This shows us that $\nabla (E_{\varepsilon}Q_{\varepsilon})\chi_{\Omega_{\varepsilon}}$ converges strongly to $\nabla Q_0$ in $L^2(\Omega)$, where $\chi_{\Omega_{\varepsilon}}$ is the characteristic function of $\Omega_{\varepsilon}$. We now show that $\nabla (E_{\varepsilon}Q_{\varepsilon})\chi_{\mathcal{N}_{\varepsilon}}$ converges strongly to $0$ in $L^2(\Omega)$. We have:
\begin{align*}
\big\|\nabla(E_{\varepsilon}Q_{\varepsilon})\chi_{\mathcal{N}_{\varepsilon}}\big\|_{L^2(\Omega)} &=\bigg(\int_{\Omega}|\nabla E_{\varepsilon}Q_{\varepsilon}|^2\cdot |\chi_{\mathcal{N}_{\varepsilon}}|^2\text{d}x\bigg)^{1/2}\leq \bigg(\int_{\Omega}|\nabla E_{\varepsilon}Q_{\varepsilon}|^2\text{d}x\bigg)^{1/2}\cdot \big\|\chi_{\mathcal{N}_{\varepsilon}}\big\|_{L^{\infty}(\Omega)}\\
&\leq \big\|\nabla E_{\varepsilon}Q_{\varepsilon}\big\|_{L^2(\Omega)}\cdot\big\|\chi_{\mathcal{N}_{\varepsilon}}\big\|_{L^{\infty}(\Omega)}\leq C\cdot \big\|\nabla Q_{\varepsilon}\big\|_{L^2(\Omega_{\varepsilon})}\cdot\big\|\chi_{\mathcal{N}_{\varepsilon}}\big\|_{L^{\infty}(\Omega)}\\
&\leq C\cdot\big(\big\|\nabla Q_0\big\|_{L^2(\Omega)}+\delta_0\big)\cdot\big\|\chi_{\mathcal{N}_{\varepsilon}}\big\|_{L^{\infty}(\Omega)},
\end{align*}
where we have used \cref{lemma:extensionineq} and the fact that $Q_{\varepsilon}\in\mathcal{B}_{\varepsilon}$. According to \cref{prop:volumegoesto0}, we have $|\mathcal{N}_{\varepsilon}|\rightarrow 0$, therefore we have $\nabla (E_{\varepsilon}Q_{\varepsilon})\chi_{\mathcal{N}_{\varepsilon}}\rightarrow 0$ strongly in $L^2(\Omega)$.

Combining all the results, we obtain that $\nabla E_{\varepsilon}Q_{\varepsilon}\rightarrow \nabla Q_0$ strongly in $L^2(\Omega)$, hence $E_{\varepsilon}Q_{\varepsilon}$ converges strongly to $Q_0$ in $H^1(\Omega)$, since the weak convergence $E_{\varepsilon}Q_{\varepsilon}\rightharpoonup Q_0$ in $H^1(\Omega)$ automatically implies the strong convergence $E_{\varepsilon}Q_{\varepsilon}\rightarrow Q_0$ in $L^2(\Omega)$.

\end{proof}

\section{Rate of convergence}\label{section:rate_of_conv}

The aim of this section is the study the rate of convergence of the sequence $J_{\varepsilon}^{\mathcal{T}}[Q_{\varepsilon}]$ to $J_0[Q_0]$, where $J_{\varepsilon}^{\mathcal{T}}$ is defined in \eqref{defn:Jeps_T} and in \eqref{defn:Jeps_T_X_Y_Z}, $J_0$ is defined in \eqref{defn:J_0} and $(Q_{\varepsilon})_{\varepsilon>0}$ is a sequence from $H^1_g(\Omega,\mathcal{S}_0)$ that converges $H^1$-strongly to $Q\in H^1_g(\Omega,\mathcal{S}_0)$. We omit the term $J_{\varepsilon}^{\mathcal{S}}$ because in \cref{subsection:nocontribution} we proved that this term has no contribution to the homogenised functional. 

First, we recall some notations used in the previous sections. For a $Q\in H^1_g(\Omega,\mathcal{S}_0)$, we write $J_0[Q]$ in the following form:
\begin{align}\label{defn:J_0_new_form}
J_0[Q]&=\int_{\Omega}\bigg(\dfrac{1}{r}\Psi^{Y}(Q)+\dfrac{1}{q}\Psi^{Z}(Q)\bigg)\text{d}x+\notag\\
&+\int_{\Omega}\bigg(\dfrac{1}{r}\Psi^{X}(Q)+\dfrac{1}{p}\Psi^{Z}(Q)\bigg)\text{d}x+\notag\\
&+\int_{\Omega}\bigg(\dfrac{1}{q}\Psi^{X}(Q)+\dfrac{1}{p}\Psi^{Y}(Q)\bigg)\text{d}x,
\end{align}
where $\Psi^{X}$, $\Psi^{Y}$ and $\Psi^{Z}$ are defined in \eqref{defn:PSI}. We also write $\tilde{J}_{\varepsilon}[Q]$, defined in \eqref{defn:J_tilde_eps}, as:
\begin{align}\label{defn:J_tilde_eps_new_form}
\tilde{J}_{\varepsilon}[Q]&=\int_{\Omega}\bigg(\dfrac{p\varepsilon-\varepsilon^{\alpha}}{pr(\varepsilon-\varepsilon^{\alpha})}\Psi^{Y}(Q)+\dfrac{p\varepsilon-\varepsilon^{\alpha}}{pq(\varepsilon-\varepsilon^{\alpha})}\Psi^{Z}(Q)\bigg)\text{d}\mu_{\varepsilon}^{X}+\notag\\
&+\int_{\Omega}\bigg(\dfrac{q\varepsilon-\varepsilon^{\alpha}}{qr(\varepsilon-\varepsilon^{\alpha})}\Psi^{X}(Q)+\dfrac{q\varepsilon-\varepsilon^{\alpha}}{pq(\varepsilon-\varepsilon^{\alpha})}\Psi^{Z}(Q)\bigg)\text{d}\mu_{\varepsilon}^{Y}+\notag\\
&+\int_{\Omega}\bigg(\dfrac{r\varepsilon-\varepsilon^{\alpha}}{qr(\varepsilon-\varepsilon^{\alpha})}\Psi^{X}(Q)+\dfrac{r\varepsilon-\varepsilon^{\alpha}}{pr(\varepsilon-\varepsilon^{\alpha})}\Psi^{Y}(Q)\bigg)\text{d}\mu_{\varepsilon}^{Z},
\end{align}
using \eqref{defn:J_tilde_eps_nice_form} and the analogous formulae.

We suppose now that:
\begin{itemize}
\item[($H_1$)] the \textit{surface energy density} $f_s$ is locally Lipschitz continuous.
\end{itemize}  

Using the assumption ($A_7$), from \cref{section:assumptions}, we have: 
\begin{align}\label{eq:properties_f_s_1}
|f_s(Q_1,\nu)-f_s(Q_2,\nu)|&\lesssim |Q_2-Q_1|(|Q_1|^3+|Q_2|^3+1), 
\end{align}
for any $Q_1,Q_2\in\mathcal{S}_0$ and any $\nu\in\mathbb{S}^2$, and
\begin{align}\label{eq:properties_f_s_2}
|f_s(Q,\nu)|&\lesssim |Q|^4+1,
\end{align}
for any $Q\in\mathcal{S}_0$ and any $\nu\in\mathbb{S}^2$.

We now have the following lemma:
\begin{lemma}\label{lemma:control_Psi}
For any $K\in\{X,Y,Z\}$, the function $\Psi^{K}$ is locally Lipschitz continuous and there holds:
\begin{align*}
|\Psi^{K}(Q)|\lesssim |Q|^4+1\hspace{10mm} \big|\nabla\Psi^{K}(Q)\big|\lesssim |Q|^3+1,
\end{align*}
for any $Q\in\mathcal{S}_0$. Moreover, the function $\Psi^K$ satisfies:
\begin{align*}
|\Psi^K(Q_1)-\Psi^K(Q_2)|\lesssim |Q_2-Q_1|(|Q_1|^3+|Q_2|^3+1),
\end{align*}
for any $Q_1,Q_2\in\mathcal{S}_0$.
\end{lemma}

\begin{proof}
The proof of this lemma follows immediatly, using the definitions of the functions $\Psi^X$, $\Psi^Y$ and $\Psi^Z$ from \eqref{defn:PSI}, the assumption ($H_1$) and the properties of the function $f_s$ from \eqref{eq:properties_f_s_1} and \eqref{eq:properties_f_s_2}.
\end{proof}

We recall now that the measures $\mu_{\varepsilon}^{X}$, $\mu_{\varepsilon}^Y$ and $\mu_{\varepsilon}^Z$, which are defined in \eqref{defn:measures}, converge weakly*, as measures in $\mathbb{R}^3$, to the Lebesgue measure restricted to $\Omega$, according to ($A_4$) from \cref{section:assumptions}. We need to prescribe a rate of convergence and for this we use the $W^{-1,1}$-norm (that is, the dual Lipschitz norm, also known as flat norm in some contexts): 

\begin{align*}
\mathbb{F}_{\varepsilon}:=\max_{K\in\{X,Y,Z\}}\sup\bigg\{\int_{\Omega}\varphi\text{d}\mu_{\varepsilon}^{K}-\int_{\Omega}\varphi\text{d}x\;:\;\varphi\in W^{1,\infty}(\Omega), \|\nabla\varphi\|_{L^{\infty}(\Omega)}+\|\varphi\|_{L^{\infty}(\Omega)}\leq 1\bigg\}.
\end{align*}

\begin{lemma}\label{lemma:Fbb_eps}
There exists a constant $\lambda_{\text{flat}}>0$ such that $\mathbb{F}_{\varepsilon}\leq \lambda_{\text{flat}}\varepsilon$ for any $\varepsilon>0$.
\end{lemma}

\begin{proof}
Let $\varphi\in W^{1,\infty}(\Omega)$. Then, according to the definition of $\mu_{\varepsilon}^{X}$ from \eqref{defn:measures}, we have:
\begin{align*}
\int_{\Omega}\varphi\text{d}\mu_{\varepsilon}^X=\varepsilon^{3}\sum_{k=1}^{X_{\varepsilon}}\varphi(y_{\varepsilon}^{x,k})=\sum_{k=1}^{X_{\varepsilon}}\int_{y_{\varepsilon}^{x,k}+[-\varepsilon/2,\varepsilon/2]^3}\varphi(y_{\varepsilon}^{x,k})\text{d}x,
\end{align*}
where in the last equality we integrate over the cube with length $\varepsilon$ centered in $y_{\varepsilon}^{x,k}$. Let $\Omega_{\varepsilon}^X$ be the following set:
\begin{align*}
\Omega_{\varepsilon}^{X}:=\bigcup_{k=1}^{X_{\varepsilon}}\big(y_{\varepsilon}^{x,k}+[-\varepsilon/2,\varepsilon/2]^3\big).
\end{align*}
Hence, we can write:
\begin{align*}
\int_{\Omega}\varphi\text{d}\mu_{\varepsilon}^{X}=\int_{\Omega_{\varepsilon}^X}\varphi(y_{\varepsilon}^{x,k})\text{d}x.
\end{align*}

Then:
\begin{align}
\bigg|\int_{\Omega}\varphi\text{d}\mu_{\varepsilon}^X-\int_{\Omega}\varphi\text{d}x\bigg|&\leq \int_{\Omega_{\varepsilon}^X}\big|\varphi-\varphi(y_{\varepsilon}^{x,k})\big|\text{d}x+\int_{\Omega\setminus\Omega_{\varepsilon}^X}\big|\varphi\big|\text{d}x\notag\\
&\leq \dfrac{\varepsilon\sqrt{3}}{2}\|\nabla\varphi\|_{L^{\infty}(\Omega)}\cdot |\Omega_{\varepsilon}^X|+\|\varphi\|_{L^{\infty}(\Omega)}\cdot |\Omega\setminus\Omega_{\varepsilon}^{X}|\label{eq:diff_of_varphi},
\end{align}
where $\dfrac{\varepsilon\sqrt{3}}{2}$ comes from the largest possible value for $|x-y_{\varepsilon}^{x,k}|$, with $x\in\big(y_{\varepsilon}^{x,k}+[-\varepsilon/2,\varepsilon/2]^3\big)$. 

If we look now at the definition of the points $y_{\varepsilon}^{x,k}$ in \eqref{defn:y_eps_x_k}, hence also at the definition of the points $x_{\varepsilon}^{i}$ in \eqref{defn:points1} and \eqref{defn:points2}, we observe that $\Omega\setminus\Omega_{\varepsilon}^{X}\subset\{x\in\Omega\;:\;\text{dist}(x,\partial\Omega)<\varepsilon\}$, therefore, we have $|\Omega\setminus\Omega_{\varepsilon}^X|\leq C\cdot\varepsilon$, where $C$ is an $\varepsilon$-independent constant. At the same time, we have $\Omega_{\varepsilon}^X\subset\Omega\Rightarrow |\Omega_{\varepsilon}^X|\leq |\Omega|$, so \eqref{eq:diff_of_varphi} becomes:
\begin{align*}
\bigg|\int_{\Omega}\varphi\text{d}\mu_{\varepsilon}^X-\int_{\Omega}\varphi\text{d}x\bigg|&\leq \dfrac{\varepsilon\sqrt{3}}{2}\|\nabla\varphi\|_{L^{\infty}(\Omega)}\cdot|\Omega|+C\cdot\varepsilon\|\varphi\|_{L^{\infty}(\Omega)}\lesssim \varepsilon\cdot\|\varphi\|_{W^{1,\infty}(\Omega)}.
\end{align*}

Computing in the same fashion for $\mu_{\varepsilon}^Y$ and $\mu_{\varepsilon}^Z$, we obtain the conclusion.
\end{proof}

We also suppose that:
\begin{itemize}
\item[($H_2$)] $g$ is bounded and Lipschitz, where $g$ represents the prescribed boundary data.
\end{itemize}

Since $\Omega$ is bounded and smooth (by assumption ($A_1$) from \cref{section:assumptions}), we can extend the function $g$ to a bounded and Lipschitz map from $\mathbb{R}^3$ to $\mathcal{S}_0$, denoted still as $g$.

We present an auxiliary result proved in \cite{CanevariZarnescu2}:

\begin{lemma}\label{lemma:Q_eps_beta}
Let $\Omega\subseteq\mathbb{R}^3$ a bounded, smooth domain, and let $g:\Omega\rightarrow\mathcal{S}_0$ be a bounded, Lipschitz map. For any $Q\in H^1_g(\Omega,\mathcal{S}_0)$ and $\sigma\in(0,1)$, there exists a bounded, Lipschitz map $Q_{\sigma}:\overline{\Omega}\rightarrow\mathcal{S}_0$ that satisfies the following properties:
\begin{equation*}
Q_{\sigma}=g\hspace{3mm}\text{on}\;\partial\Omega
\end{equation*}
\begin{equation}\label{eq:Q_sigma}
\|Q_{\sigma}\|_{L^{\infty}(\Omega)}\lesssim \sigma^{-1/2}\big(\|Q\|_{H^1(\Omega)}+\|g\|_{L^{\infty}(\Omega)}\big)
\end{equation}
\begin{equation}\label{eq:nabla_Q_sigma}
\|\nabla Q_{\sigma}\|_{L^{\infty}(\Omega)}\lesssim \sigma^{-3/2}\big(\|Q\|_{H^1(\Omega)}+\|g\|_{W^{1,\infty}(\Omega)}\big)
\end{equation}
\begin{equation}\label{eq:Q-Q_sigma_L^2}
\|Q-Q_{\sigma}\|_{L^2(\Omega)}\lesssim \sigma\|Q\|_{H^1(\Omega)}
\end{equation}
\begin{equation}\label{eq:nabla_Q-Q_sigma_L^2}
\|\nabla Q-\nabla Q_{\sigma}\|_{L^2(\Omega)}\rightarrow 0\hspace{3mm}\text{as}\;\sigma\rightarrow 0.
\end{equation}
\end{lemma}

The main result from this section is the following:

\begin{prop}\label{prop:rate_of_conv}
Suppose that assumptions ($A_1$)-($A_7$) (from \cref{section:assumptions}) and ($H_1$)-($H_2$) (from this section) hold. Then, for any $Q\in H^1_g(\Omega,\mathcal{S}_0)$, there exists a sequence $(Q_{\varepsilon})_{\varepsilon>0}$ in $H^1_g(\Omega,\mathcal{S}_0)$ that converges $H^1(\Omega)$-strongly to $Q$ and satisfies
\begin{equation*}
|J_{\varepsilon}^{\mathcal{T}}[Q_{\varepsilon}]-J_{0}[Q]|\lesssim \varepsilon^{(\alpha-1)/3}\big(\|Q\|^4_{H^1(\Omega)}+1\big),
\end{equation*}
for $\varepsilon$ small enough, where $J_{\varepsilon}^{\mathcal{T}}$ is defined in \eqref{defn:Jeps_T} and $J_0$ is defined in \eqref{defn:J_0_new_form}. The constant that merged with the sign ``$\lesssim$" depends only on the $L^{\infty}$-norms of $g$ and $\nabla g$, but also on $\Omega$, $f_s$ and the initial cube $\mathcal{C}$.
\end{prop}

\begin{remark}\label{remark:rate_of_convergence}
The previous proposition allows us to obtain, as claimed,  a rate of convergence for the minimisers $\bar Q_\varepsilon$ of $\mathcal{F}_\varepsilon$ to a minimiser $Q$ of $\mathcal{F}_0$ in terms of $\|\bar Q_\varepsilon-Q\|_{H^1(\Omega)}=o(1)$ as $\varepsilon\to 0$ (i.e. relation \eqref{eq:rate_of_convergence}).

Indeed, this is obtained in the following way. First, let us fix a value for $0<\varepsilon<1$ such that equation \eqref{eq:alpha/4} holds. Then we use the inequality
\begin{align*}
|J_{\varepsilon}^{\mathcal{T}}[\bar Q_{\varepsilon}]-J_0[Q]|&\leq |J_{\varepsilon}^{\mathcal{T}}[\bar Q_{\varepsilon}]-J_{\varepsilon}^{\mathcal{T}}[Q_{\varepsilon}]|+|J_{\varepsilon}^{\mathcal{T}}[Q_{\varepsilon}]-J_0[Q]|,
\end{align*}
where $Q_{\varepsilon}$ is the function from $H^1_g(\Omega,\mathcal{S}_0)$ granted by \cref{lemma:Q_eps_beta}, with $\sigma=\varepsilon^{(\alpha-1)/3}$.

For the first term from the right-hand side from the last inequality, we use relation \eqref{eq:alpha/4} and we obtain, for a fixed $\varepsilon$ sufficiently small:
\begin{align*}
|J_{\varepsilon}^{\mathcal{T}}[\bar Q_{\varepsilon}]-J_{\varepsilon}^{\mathcal{T}}[Q_{\varepsilon}]|&\leq C\cdot\big(\varepsilon^{1/2-\alpha/4}+\|\bar Q_{\varepsilon}-Q_{\varepsilon}\|_{L^4(\Omega)}\big),
\end{align*}
where $C$ is $\varepsilon$-independent.

From the compact Sobolev embedding $H^1(\Omega)\hookrightarrow L^4(\Omega)$, we obtain:
\begin{align*}
|J_{\varepsilon}^{\mathcal{T}}[\bar Q_{\varepsilon}]-J_{\varepsilon}^{\mathcal{T}}[Q_{\varepsilon}]|&\leq C\cdot\big(\varepsilon^{1/2-\alpha/4}+\|\bar Q_{\varepsilon}-Q_{\varepsilon}\|_{H^1(\Omega)}\big).
\end{align*}
Now, we observe that:
\begin{align*}
\|\bar Q_{\varepsilon}-Q_{\varepsilon}\|_{H^1(\Omega)}&\leq \|\bar Q_{\varepsilon}-Q\|_{H^1(\Omega)}+\|Q_{\varepsilon}-Q\|_{H^1(\Omega)}\\
&\leq \|\bar Q_{\varepsilon}-Q\|_{H^1(\Omega)}+\|Q_{\varepsilon}-Q\|_{L^2(\Omega)}+\|\nabla Q_{\varepsilon}-\nabla Q\|_{L^2(\Omega)}\\
&\leq \|\bar Q_{\varepsilon}-Q\|_{H^1(\Omega)}+\varepsilon^{(\alpha-1)/3}\|Q\|_{H^1(\Omega)}+\|\nabla Q_{\varepsilon}-\nabla Q\|_{L^2(\Omega)},
\end{align*}
where we have used relation \eqref{eq:Q-Q_sigma_L^2} in the last row. Relation \eqref{eq:nabla_Q-Q_sigma_L^2} tells us that $\|\nabla Q_{\varepsilon}-\nabla Q\|_{L^2(\Omega)}\rightarrow 0$ as $\varepsilon\rightarrow 0$, hence, by the choice of $\varepsilon$, we can control it with a constant. Since $Q$ is fixed, we can also control $\|Q\|_{H^1(\Omega)}$ with an $\varepsilon$-independent constant. Therefore, we can write:
\begin{align*}
\|\bar Q_{\varepsilon}-Q_{\varepsilon}\|_{H^1(\Omega)}&\lesssim \|\bar Q_{\varepsilon}-Q\|_{H^1(\Omega)}+\varepsilon^{(\alpha-1)/3}.
\end{align*}

Hence, we have:
\begin{align*}
|J_{\varepsilon}^{\mathcal{T}}[\bar Q_{\varepsilon}]-J_{\varepsilon}^{\mathcal{T}}[Q_{\varepsilon}]|&\lesssim \varepsilon^{1/2-\alpha/4}+\varepsilon^{(\alpha-1)/3}+\|\bar Q_{\varepsilon}-Q\|_{H^1(\Omega)}.
\end{align*}
and if we denote by $m_{\alpha}=\text{min}\{1/2-\alpha/4,(\alpha-1)/3\}$ (which is defined depending whether $\alpha$ is bigger or smaller than $10/7$), we can rewrite:
\begin{align*}
|J_{\varepsilon}^{\mathcal{T}}[\bar Q_{\varepsilon}]-J_{\varepsilon}^{\mathcal{T}}[Q_{\varepsilon}]|&\lesssim \varepsilon^{m_{\alpha}}+\|\bar Q_{\varepsilon}-Q\|_{H^1(\Omega)},
\end{align*} 
since $\varepsilon$ is chosen from $(0,1)$.

For the term $|J_{\varepsilon}^{\mathcal{T}}[Q_{\varepsilon}]-J_0[Q]|$, we apply \cref{prop:rate_of_conv}:
\begin{align*}
|J_{\varepsilon}^{\mathcal{T}}[Q_{\varepsilon}]-J_0[Q]|&\lesssim \varepsilon^{(\alpha-1)/3}(\|Q\|^4_{H^1(\Omega)}+1)
\end{align*}
and since $Q$ is fixed, we obtain:
\begin{align*}
|J_{\varepsilon}^{\mathcal{T}}[Q_{\varepsilon}]-J_0[Q]|&\lesssim \varepsilon^{(\alpha-1)/3}\lesssim \varepsilon^{m_{\alpha}},
\end{align*}
using the definition of $m_{\alpha}$.

If we go back to our initial inequality, we obtain:
\begin{align}\label{eq:rate_of_convergence}
|J_{\varepsilon}^{\mathcal{T}}[\bar Q_{\varepsilon}]-J_0[Q]|&\lesssim \varepsilon^{\text{min}\{1/2-\alpha/4,(\alpha-1)/3\}}+\|\bar Q_{\varepsilon}-Q\|_{H^1(\Omega)}.
\end{align}
\end{remark}

\begin{proof}[Proof of \cref{prop:rate_of_conv}]
Let us fix a small $\varepsilon\in(0,1)$ such that:
\begin{align}\label{eq:choice_of_eps_2p}
\dfrac{p\varepsilon-\varepsilon^{\alpha}}{\varepsilon-\varepsilon^{\alpha}}<2p,\;\dfrac{q\varepsilon-\varepsilon^{\alpha}}{\varepsilon-\varepsilon^{\alpha}}<2q\;\text{and}\;\dfrac{r\varepsilon-\varepsilon^{\alpha}}{\varepsilon-\varepsilon^{\alpha}}<2r.
\end{align}
This is possible since $\dfrac{p\varepsilon-\varepsilon^{\alpha}}{\varepsilon-\varepsilon^{\alpha}}\searrow p$, $\dfrac{q\varepsilon-\varepsilon^{\alpha}}{\varepsilon-\varepsilon^{\alpha}}\searrow q$ and $\dfrac{r\varepsilon-\varepsilon^{\alpha}}{\varepsilon-\varepsilon^{\alpha}}\searrow r$  as $\varepsilon\rightarrow 0$ and $p,q,r\geq 1$.

Let now $\beta$ be a positive parameter, to be chosen later, and let $Q_{\varepsilon}:=Q_{\varepsilon^{\beta}}\in H^1_g(\Omega,\mathcal{S}_0)$ be the Lipschitz map given by \cref{lemma:Q_eps_beta}. Then, we have:
\begin{align}\label{eq:control_J_eps_Q_eps_J_0_Q}
|J_{\varepsilon}^{\mathcal{T}}[Q_{\varepsilon}]-J_0[Q]|\leq |J_{\varepsilon}^{\mathcal{T}}[Q_{\varepsilon}]-\tilde{J}_{\varepsilon}[Q_{\varepsilon}]|+|\tilde{J}_{\varepsilon}[Q_{\varepsilon}]-J_0[Q_{\varepsilon}]|+|J_0[Q_{\varepsilon}]-J_0[Q]|,
\end{align}
where $\tilde{J}_{\varepsilon}$ is defined in \eqref{defn:J_tilde_eps}. 

We analyse the first term from the right-hand side from \eqref{eq:control_J_eps_Q_eps_J_0_Q}. Using the same notations as in \cref{th:Je_goes_to_J0}, replacing $\text{Lip}(Q_{\varepsilon})$ (the Lipschitz constant) with $\|\nabla Q_{\varepsilon}\|_{L^{\infty}(\Omega)}$ and combining relations \eqref{eq:J_eps_J_eps_tilde_control1} and \eqref{eq:J_eps_J_eps_tilde_control2}, we obtain:
\begin{align*}
|J_{\varepsilon}^X[Q_{\varepsilon}]-\tilde{J}^X_{\varepsilon}[Q_{\varepsilon}]|&\lesssim \big(\|Q_{\varepsilon}\|^3_{L^{\infty}(\Omega)}+1\big)\cdot\|\nabla Q_{\varepsilon}\|_{L^{\infty}(\Omega)}\cdot\dfrac{2(r+q)}{pqr}\cdot\dfrac{p\varepsilon-\varepsilon^{\alpha}}{\varepsilon-\varepsilon^{\alpha}}\cdot\\
&\hspace{2mm}\cdot X_{\varepsilon}\cdot \varepsilon^{3}\cdot \sqrt{\bigg(\varepsilon-\dfrac{\varepsilon^{\alpha}}{p}\bigg)^2+\bigg(\dfrac{\varepsilon^{\alpha}}{q}\bigg)^2+\bigg(\dfrac{\varepsilon^{\alpha}}{r}\bigg)^2}.
\end{align*}

Using \eqref{eq:control_over_X_eps} and \eqref{eq:choice_of_eps_2p}, we can rewrite the last inequality as follows:
\begin{align}\label{eq:control_J_eps_X_J_tilde_eps_X_sqrt}
|J_{\varepsilon}^X[Q_{\varepsilon}]-\tilde{J}^X_{\varepsilon}[Q_{\varepsilon}]|&\lesssim \big(\|Q_{\varepsilon}\|^3_{L^{\infty}(\Omega)}+1\big)\cdot\|\nabla Q_{\varepsilon}\|_{L^{\infty}(\Omega)}\cdot\sqrt{\bigg(\varepsilon-\dfrac{\varepsilon^{\alpha}}{p}\bigg)^2+\bigg(\dfrac{\varepsilon^{\alpha}}{q}\bigg)^2+\bigg(\dfrac{\varepsilon^{\alpha}}{r}\bigg)^2},
\end{align}
since the term $\dfrac{2(r+q)}{pqr}$ can be bounded with an $\varepsilon$-independent constant. Now, because $p,q,r\geq 1$, we have:
\begin{align*}
\dfrac{\varepsilon^{2\alpha}}{p^2},\dfrac{\varepsilon^{2\alpha}}{q^2},\dfrac{\varepsilon^{2\alpha}}{r^2}\leq \varepsilon^{2\alpha}
\end{align*}
and, because $\varepsilon>0$ and $\alpha\in(1,3/2)$, we also have:
\begin{align*}
0<\varepsilon-\varepsilon^{\alpha}\leq \varepsilon-\dfrac{\varepsilon^{\alpha}}{k}\leq \varepsilon,\;\text{for}\;k\in\{p,q,r\}.
\end{align*}
Therefore, \eqref{eq:control_J_eps_X_J_tilde_eps_X_sqrt} becomes:
\begin{align*}
|J_{\varepsilon}^X[Q_{\varepsilon}]-\tilde{J}^X_{\varepsilon}[Q_{\varepsilon}]|&\lesssim \big(\|Q_{\varepsilon}\|^3_{L^{\infty}(\Omega)}+1\big)\cdot\|\nabla Q_{\varepsilon}\|_{L^{\infty}(\Omega)}\cdot\sqrt{\varepsilon^2+2\varepsilon^{2\alpha}},
\end{align*}
and using the same arguments for $J_{\varepsilon}^Y[Q_{\varepsilon}]$ and $J_{\varepsilon}^Z[Q_{\varepsilon}]$, we obtain:
\begin{align}\label{eq:control_J_eps_J_tilde_no_sqrt}
|J_{\varepsilon}^{\mathcal{T}}[Q_{\varepsilon}]-\tilde{J}_{\varepsilon}[Q_{\varepsilon}]|&\lesssim \big(\|Q_{\varepsilon}\|^3_{L^{\infty}(\Omega)}+1\big)\cdot\|\nabla Q_{\varepsilon}\|_{L^{\infty}(\Omega)}\cdot\sqrt{\varepsilon^2+2\varepsilon^{2\alpha}},
\end{align}
where the constant 3, which comes from adding the three relations obtained, has merged into the ``$\lesssim$" sign.

Using \cref{lemma:Q_eps_beta}, we have:
\begin{align*}
\|Q_{\varepsilon}\|^3_{L^{\infty}(\Omega)}&\lesssim \varepsilon^{-3\beta/2}\big(\|Q\|_{H^1(\Omega)}+\|g\|_{L^{\infty}(\Omega)}\big)^3\\
\|\nabla Q_{\varepsilon}\|_{L^{\infty}(\Omega)}&\lesssim \varepsilon^{-3\beta/2}\big(\|Q\|_{H^1(\Omega)}+\|g\|_{W^{1,\infty}(\Omega)}\big).
\end{align*} 

Now, the constant involved by using the sign ``$\lesssim$" is going to depend also on the $L^{\infty}$-norms of $g$ and $\nabla g$, hence, relation \eqref{eq:control_J_eps_J_tilde_no_sqrt} becomes:
\begin{align*}
|J_{\varepsilon}^{\mathcal{T}}[Q_{\varepsilon}]-\tilde{J}_{\varepsilon}[Q_{\varepsilon}]|&\lesssim \dfrac{\sqrt{\varepsilon^2+\varepsilon^{2\alpha}}}{\varepsilon^{3\beta}}\big(\|Q\|^4_{H^1(\Omega)}+1\big)\lesssim \sqrt{\varepsilon^{2(1-3\beta)}+\varepsilon^{2(\alpha-3\beta)}}\big(\|Q\|^4_{H^1(\Omega)}+1\big).
\end{align*}

Since $\alpha\in(1,3/2)$, we have $1-3\beta<\alpha-3\beta$. Therefore, we can write the last inequality as follows:
\begin{align}\label{eq:first_control_J_eps_J_0_Q_eps}
|J_{\varepsilon}^{\mathcal{T}}[Q_{\varepsilon}]-\tilde{J}_{\varepsilon}[Q_{\varepsilon}]|&\lesssim\varepsilon^{1-3\beta}\big(\|Q\|^4_{H^1(\Omega)}+1\big),
\end{align}
since $\varepsilon\in(0,1)$.

In order to analyse better the second term from \eqref{eq:control_J_eps_Q_eps_J_0_Q}, which contains $\tilde{J}_{\varepsilon}[Q_{\varepsilon}]$ and $J_0[Q_{\varepsilon}]$, we analyse the first terms from \eqref{defn:J_0_new_form} and \eqref{defn:J_tilde_eps_new_form}:
\begin{align*}
\bigg|\dfrac{p\varepsilon-\varepsilon^{\alpha}}{pr(\varepsilon-\varepsilon^{\alpha})}\int_{\Omega}\Psi^{Y}(Q_{\varepsilon})\text{d}\mu_{\varepsilon}^{X}-\dfrac{1}{r}\int_{\Omega}\Psi^{Y}(Q_{\varepsilon})\text{d}x\bigg|&\leq \dfrac{p\varepsilon-\varepsilon^{\alpha}}{pr(\varepsilon-\varepsilon^{\alpha})}\bigg|\int_{\Omega}\Psi^{Y}(Q_{\varepsilon})\text{d}\mu_{\varepsilon}^{X}-\int_{\Omega}\Psi^{Y}(Q_{\varepsilon})\text{d}x\bigg|+\\
&\hspace{2mm}+\bigg|\dfrac{p\varepsilon-\varepsilon^{\alpha}}{pr(\varepsilon-\varepsilon^{\alpha})}-\dfrac{1}{r}\bigg|\cdot\bigg|\int_{\Omega}\Psi^{Y}(Q_{\varepsilon})\text{d}x\bigg|.
\end{align*}

As we have seen before, we have $\dfrac{p\varepsilon-\varepsilon^{\alpha}}{pr(\varepsilon-\varepsilon^{\alpha})}\searrow\dfrac{1}{r}$ and we have chosen $\varepsilon>0$ such that $\dfrac{1}{r}\leq \dfrac{p\varepsilon-\varepsilon^{\alpha}}{pr(\varepsilon-\varepsilon^{\alpha})}<\dfrac{2}{r}$. Moreover, we have $\bigg|\dfrac{p\varepsilon-\varepsilon^{\alpha}}{pr(\varepsilon-\varepsilon^{\alpha})}-\dfrac{1}{r}\bigg|=\dfrac{\varepsilon^{\alpha}(p-1)}{pr(\varepsilon-\varepsilon^{\alpha})}$ and we can impose further conditions regarding the choice of $\varepsilon$, such that $\dfrac{\varepsilon^{\alpha}(p-1)}{pr(\varepsilon-\varepsilon^{\alpha})}<\varepsilon^{\alpha-1}$, which is equivalent to choosing $\varepsilon$ such that $\varepsilon^{\alpha-1}< 1-\dfrac{1}{r}+\dfrac{1}{pr}$. Hence, we have:
\begin{align*}
\bigg|\dfrac{p\varepsilon-\varepsilon^{\alpha}}{pr(\varepsilon-\varepsilon^{\alpha})}\int_{\Omega}\Psi^{Y}(Q_{\varepsilon})\text{d}\mu_{\varepsilon}^{X}-\dfrac{1}{r}\int_{\Omega}\Psi^{Y}(Q_{\varepsilon})\text{d}x\bigg|&\leq \dfrac{2}{r}\bigg|\int_{\Omega}\Psi^{Y}(Q_{\varepsilon})\text{d}\mu_{\varepsilon}^{X}-\int_{\Omega}\Psi^{Y}(Q_{\varepsilon})\text{d}x\bigg|+\varepsilon^{\alpha-1}\bigg|\int_{\Omega}\Psi^{Y}(Q_{\varepsilon})\text{d}x\bigg|.
\end{align*}

Using the definition of $\mathbb{F}_{\varepsilon}$, we have:
\begin{align*}
\bigg|\dfrac{p\varepsilon-\varepsilon^{\alpha}}{pr(\varepsilon-\varepsilon^{\alpha})}\int_{\Omega}\Psi^{Y}(Q_{\varepsilon})\text{d}\mu_{\varepsilon}^{X}-\dfrac{1}{r}\int_{\Omega}\Psi^{Y}(Q_{\varepsilon})\text{d}x\bigg|\leq \dfrac{2}{r}\cdot\mathbb{F}_{\varepsilon}\cdot\|\Psi^{Y}(Q_{\varepsilon})\|_{W^{1,\infty}(\Omega)}+\varepsilon^{\alpha-1}\|\Psi^{Y}(Q_{\varepsilon})\|_{L^{\infty}(\Omega)}.
\end{align*}
Using now the fact that now $Q_{\varepsilon}\in H^1_g(\Omega,\mathcal{S}_0)$, \cref{lemma:control_Psi} and \cref{lemma:Fbb_eps}, we obtain (also by moving the constant $\dfrac{2}{r}$ under the ``$\lesssim$" sign):
\begin{align*}
&\bigg|\dfrac{p\varepsilon-\varepsilon^{\alpha}}{pr(\varepsilon-\varepsilon^{\alpha})}\int_{\Omega}\Psi^{Y}(Q_{\varepsilon})\text{d}\mu_{\varepsilon}^{X}-\dfrac{1}{r}\int_{\Omega}\Psi^{Y}(Q_{\varepsilon})\text{d}x\bigg|\lesssim \varepsilon\big(\|Q_{\varepsilon}\|^4_{L^{\infty}(\Omega)}+1\big)+\\
&\hspace{20mm}+\varepsilon\big(\|Q_{\varepsilon}\|^3_{L^{\infty}(\Omega)}+1\big)\cdot\|\nabla Q_{\varepsilon}\|_{L^{\infty}(\Omega)} + \varepsilon^{\alpha-1}\big(\|Q_{\varepsilon}\|^4_{L^{\infty}(\Omega)}+1\big),
\end{align*}
Applying \cref{lemma:Q_eps_beta}, we get:
\begin{align*}
&\bigg|\dfrac{p\varepsilon-\varepsilon^{\alpha}}{pr(\varepsilon-\varepsilon^{\alpha})}\int_{\Omega}\Psi^{Y}(Q_{\varepsilon})\text{d}\mu_{\varepsilon}^{X}-\dfrac{1}{r}\int_{\Omega}\Psi^{Y}(Q_{\varepsilon})\text{d}x\bigg|\lesssim \varepsilon\bigg(\varepsilon^{-2\beta}\big(\|Q\|_{H^1(\Omega)}+\|g\|_{L^{\infty}(\Omega)}\big)^4+1\bigg)+\\
&\hspace{10mm}+\varepsilon\bigg(\varepsilon^{-3\beta/2}\big(\|Q\|_{H^1(\Omega)}+\|g\|_{L^{\infty}(\Omega)}\big)^3+1\bigg)\cdot\varepsilon^{-3\beta/2}\big(\|Q\|_{H^1(\Omega)}+\|g\|_{W^{1,\infty}(\Omega)}\big) +\\
&\hspace{10mm}+\varepsilon^{\alpha-1}\bigg(\varepsilon^{-2\beta}\big(\|Q\|_{H^1(\Omega)}+\|g\|_{L^{\infty}(\Omega)}\big)^4+1\bigg),
\end{align*}

Moving the terms $\|g\|_{L^{\infty}(\Omega)}$ and $\|g\|_{W^{1,\infty}(\Omega)}$ under the ``$\lesssim$" sign and using the fact that $\beta>0$ and $\varepsilon\in(0,1)$, we have:
\begin{align*}
&\bigg|\dfrac{p\varepsilon-\varepsilon^{\alpha}}{pr(\varepsilon-\varepsilon^{\alpha})}\int_{\Omega}\Psi^{Y}(Q_{\varepsilon})\text{d}\mu_{\varepsilon}^{X}-\dfrac{1}{r}\int_{\Omega}\Psi^{Y}(Q_{\varepsilon})\text{d}x\bigg| \lesssim \big(\varepsilon^{1-2\beta}+\varepsilon^{1-3\beta}+\varepsilon^{\alpha-2\beta-1}\big)\big(\|Q\|^4_{H^1(\Omega)}+1\big).
\end{align*}
Applying the same technique for the other five terms from $J_0$ and $\tilde{J}_{\varepsilon}$, which are in \eqref{defn:J_0_new_form} and \eqref{defn:J_tilde_eps_new_form}, we obtain:
\begin{align*}
|\tilde{J}_{\varepsilon}[Q_{\varepsilon}]-J_0[Q_{\varepsilon}]|\lesssim\big(\varepsilon^{1-2\beta}+\varepsilon^{1-3\beta}+\varepsilon^{\alpha-2\beta-1}\big)\big(\|Q\|^4_{H^1(\Omega)}+1\big)
\end{align*}
and using once again that $\beta>0$ and $\varepsilon\in(0,1)$, we can write:
\begin{align}\label{eq:second_control_J_eps_J_0_Q_eps}
|\tilde{J}_{\varepsilon}[Q_{\varepsilon}]-J_0[Q_{\varepsilon}]|\lesssim\big(\varepsilon^{1-3\beta}+\varepsilon^{\alpha-2\beta-1}\big)\big(\|Q\|^4_{H^1(\Omega)}+1\big).
\end{align}

Moving now to the last term from \eqref{eq:control_J_eps_Q_eps_J_0_Q}, which is $|J_0[Q_{\varepsilon}]-J_0[Q]|$, we once again analyse every difference that can be formed with the six terms from the definition of \eqref{defn:J_0_new_form}. Hence:
\begin{align*}
\bigg|\int_{\Omega}\dfrac{1}{r}\Psi^{Y}(Q_{\varepsilon})\text{d}x-\int_{\Omega}\dfrac{1}{r}\Psi^{Y}(Q)\text{d}x\bigg|\leq \dfrac{1}{r}\int_{\Omega}\big|\Psi^{Y}(Q_{\varepsilon})-\Psi^{Y}(Q)\big|\text{d}x.
\end{align*}
Using \cref{lemma:control_Psi} and moving the constant $\dfrac{1}{r}$ under the ``$\lesssim$" sign, we have:
\begin{align*}
\bigg|\int_{\Omega}\dfrac{1}{r}\Psi^{Y}(Q_{\varepsilon})\text{d}x-\int_{\Omega}\dfrac{1}{r}\Psi^{Y}(Q)\text{d}x\bigg|&\lesssim \int_{\Omega}\big(|Q|^3+|Q_{\varepsilon}|^3+1)|Q-Q_{\varepsilon}|\text{d}x\\
&\lesssim \bigg(\int_{\Omega}\big(|Q|^3+|Q_{\varepsilon}|^3+1\big)^2\text{d}x\bigg)^{1/2}\cdot\bigg(\int_{\Omega}|Q-Q_{\varepsilon}|^2\text{d}x\bigg)^{1/2}\\
&\lesssim \bigg(\int_{\Omega}\big(|Q|^6+|Q_{\varepsilon}|^6+1\big)\text{d}x\bigg)^{1/2}\cdot\|Q-Q_{\varepsilon}\|_{L^2(\Omega)}\\
&\lesssim \big(\|Q\|^3_{L^6(\Omega)}+\|Q_{\varepsilon}\|^3_{L^6(\Omega)}+1\big)\cdot\|Q-Q_{\varepsilon}\|_{L^2(\Omega)}.
\end{align*}
The sequence $(Q_{\varepsilon})_{\varepsilon>0}$ is bounded in $L^6(\Omega)$, due to the continuous Sobolev embedding $H^1(\Omega)\hookrightarrow L^6(\Omega)$ and to \cref{lemma:Q_eps_beta}. Using once again \cref{lemma:Q_eps_beta} to control $\|Q-Q_{\varepsilon}\|_{L^2(\Omega)}$, we obtain:
\begin{align*}
\bigg|\int_{\Omega}\dfrac{1}{r}\Psi^{Y}(Q_{\varepsilon})\text{d}x-\int_{\Omega}\dfrac{1}{r}\Psi^{Y}(Q)\text{d}x\bigg|&\lesssim \varepsilon^{\beta}\big(\|Q\|^4_{H^1(\Omega)}+1\big),
\end{align*}
hence
\begin{align}\label{eq:third_control_J_eps_J_0_Q_eps}
|J_0[Q_{\varepsilon}-J_0[Q]|\lesssim \varepsilon^{\beta}\big(\|Q\|^4_{H^1(\Omega)}+1\big).
\end{align}

Combining now relations \eqref{eq:first_control_J_eps_J_0_Q_eps}, \eqref{eq:second_control_J_eps_J_0_Q_eps} and \eqref{eq:third_control_J_eps_J_0_Q_eps}, we obtain:
\begin{align*}
|J_{\varepsilon}^{\mathcal{T}}[Q_{\varepsilon}]-J_0[Q]|\lesssim \big(\varepsilon^{1-3\beta}+\varepsilon^{\alpha-2\beta-1}+\varepsilon^{\beta}\big)\big(\|Q\|^4_{H^1(\Omega)}+1\big).
\end{align*}

Now we need to find a suitable value for $\beta>0$ such that we can put the minimum positive value between the exponents $1-3\beta$, $\alpha-2\beta-1$ and $\beta$, in order to obtain the best rate of convergence. Since we desire that all exponents are positive, $\alpha-2\beta-1>0\Rightarrow \beta<\dfrac{\alpha-1}{2}$. Since $\dfrac{\alpha-1}{2}<2-\alpha\Leftrightarrow \alpha<\dfrac{5}{3}$, which is true because $\alpha\in(1,3/2)$, then we have: $\alpha-2\beta-1<1-3\beta\Leftrightarrow \beta<2-\alpha$, which is true, because $\beta<\dfrac{\alpha-1}{2}<2-\alpha$. Hence we only consider the exponents $\alpha-2\beta-1$ and $\beta$ and we can see that the optimal choice for $\beta$ is $\dfrac{\alpha-1}{3}$, which is positive because $\alpha\in(1,3/2)$. Hence, we obtain:
\begin{align*}
|J_{\varepsilon}^{\mathcal{T}}[Q_{\varepsilon}]-J_0[Q]|\lesssim \varepsilon^{(\alpha-1)/3}\big(\|Q\|^4_{H^1(\Omega)}+1\big).
\end{align*}
\end{proof}

\begin{appendix}
\section{Appendix}\label{section:appendix}

\subsection{Constructing the cubic microlattice}\label{section:constructing_lattice}

In this subsection, we provide more details regarding the construction of the gray parallelipipeds from Figure \ref{fig:nematiccage1}. 

In each of the points from $\mathcal{Y}_{\varepsilon}$ we construct a parallelipiped that connects the parallelipipeds $\mathcal{C}^i_{\varepsilon}$ and $\mathcal{C}^j_{\varepsilon}$, where $i,j\in\overline{1,N_{\varepsilon}}$ such that $|x_{\varepsilon}^i-x_{\varepsilon}^j|=\varepsilon$.

If $x^i_{\varepsilon}-x^j_{\varepsilon}=\pm(\varepsilon,0,0)^T$, then let:
\begin{align}
& \bullet X_{\varepsilon}=\text{card}\big(\big\{(i,j)\in\overline{1,N_{\varepsilon}}^2\;\big|\;x_{\varepsilon}^i-x_{\varepsilon}^j=(\varepsilon,0,0)^T,\;i<j\big\}\big)\label{defn:Y_eps_x_cardinal}\\
& \bullet \Upsilon_{\varepsilon}^x:\big\{(i,j)\in\overline{1,N_{\varepsilon}}^2\;\big|\;x_{\varepsilon}^i-x_{\varepsilon}^j=(\varepsilon,0,0)^T,\;i<j\big\}\rightarrow\overline{1,X_{\varepsilon}}\text{ a bijection;}\notag\\
& \bullet y_{\varepsilon}^{x,k}=\dfrac{1}{2}(x_{\varepsilon}^i+x_{\varepsilon}^j),\text{ where }k=\Upsilon_{\varepsilon}^x(i,j);\label{defn:y_eps_x_k}\\
& \bullet \mathcal{Y}_{\varepsilon}^x=\bigg\{y_{\varepsilon}^{x,k}\in\mathcal{Y}_{\varepsilon}\;\bigg|\;y_{\varepsilon}^{x,k}=\dfrac{1}{2}(x_{\varepsilon}^i+x_{\varepsilon}^j),\;k=\Upsilon_{\varepsilon}^x(i,j)\bigg\};\label{defn:Y_eps_x}\\
& \bullet \mathcal{P}_{\varepsilon}^{x,k}\;\text{the parallelipiped centered in }y_{\varepsilon}^{x,k}\text{, defined by }\mathcal{P}_{\varepsilon}^{x,k}=y_{\varepsilon}^{x,k}+T_x\mathcal{C}^{\alpha}\text{, where}\label{defn:P_x}\\
&\hspace{14mm} T_x\mathcal{C}^{\alpha}=\bigg[-\dfrac{p\varepsilon-\varepsilon^{\alpha}}{2p},\dfrac{p\varepsilon-\varepsilon^{\alpha}}{2p}\bigg]\times\bigg[-\dfrac{\varepsilon^{\alpha}}{2q},\dfrac{\varepsilon^{\alpha}}{2q}\bigg]\times\bigg[-\dfrac{\varepsilon^{\alpha}}{2r},\dfrac{\varepsilon^{\alpha}}{2r}\bigg];\\
& \bullet \mathcal{T}_x^k\text{ be the union of the four transparent faces of }\mathcal{P}_{\varepsilon}^{x,k}\text{ that have the length equal to }\dfrac{p\varepsilon-\varepsilon^{\alpha}}{p},\notag\\ 
&\text{which are represented in Figure \ref{fig:T_x}.\label{defn:T_x_k}}
\end{align}

If $x^i_{\varepsilon}-x^j_{\varepsilon}=\pm(0,\varepsilon,0)^T$, then let:
\begin{align}
& \bullet Y_{\varepsilon}=\text{card}\big(\big\{(i,j)\in\overline{1,N_{\varepsilon}}^2\;\big|\;x_{\varepsilon}^i-x_{\varepsilon}^j=(0,\varepsilon,0)^T,\;i<j\big\}\big)\label{defn:Y_eps_y_cardinal}\\
& \bullet \Upsilon_{\varepsilon}^y:\big\{(i,j)\in\overline{1,N_{\varepsilon}}^2\;\big|\;x_{\varepsilon}^i-x_{\varepsilon}^j=(0,\varepsilon,0)^T,\;i<j\big\}\rightarrow\overline{1,Y_{\varepsilon}}\text{ a bijection;}\notag\\
& \bullet y_{\varepsilon}^{y,l}=\dfrac{1}{2}(x_{\varepsilon}^i+x_{\varepsilon}^j),\text{ where }l=\Upsilon_{\varepsilon}^y(i,j);\notag\\
& \bullet \mathcal{Y}_{\varepsilon}^y=\bigg\{y_{\varepsilon}^{y,l}\in\mathcal{Y}_{\varepsilon}\;\bigg|\;y_{\varepsilon}^{y,l}=\dfrac{1}{2}(x_{\varepsilon}^i+x_{\varepsilon}^j),\;l=\Upsilon_{\varepsilon}^y(i,j)\bigg\};\label{defn:Y_eps_y}\\
& \bullet \mathcal{P}_{\varepsilon}^{y,l}\;\text{the parallelipiped centered in }y_{\varepsilon}^{y,l}\text{, defined by }\mathcal{P}_{\varepsilon}^{y,l}=y_{\varepsilon}^{y,l}+T_y\mathcal{C}^{\alpha}\text{, where}\label{defn:P_y}\\
&\hspace{14mm} T_y\mathcal{C}^{\alpha}=\bigg[-\dfrac{\varepsilon^{\alpha}}{2p},\dfrac{\varepsilon^{\alpha}}{2p}\bigg]\times\bigg[-\dfrac{q\varepsilon-\varepsilon^{\alpha}}{2q},\dfrac{q\varepsilon-\varepsilon^{\alpha}}{2q}\bigg]\times\bigg[-\dfrac{\varepsilon^{\alpha}}{2r},\dfrac{\varepsilon^{\alpha}}{2r}\bigg];\\
& \bullet \mathcal{T}_y^l\text{ be the union of the four transparent faces of }\mathcal{P}_{\varepsilon}^{y,l}\text{ that have the length equal to }\dfrac{q\varepsilon-\varepsilon^{\alpha}}{q},\notag\\ 
&\text{which are represented in Figure \ref{fig:T_y}.\label{defn:T_y_l}}
\end{align}

If $x^i_{\varepsilon}-x^j_{\varepsilon}=\pm(0,0,\varepsilon)^T$, then let:
\begin{align}
& \bullet Z_{\varepsilon}=\text{card}\big(\big\{(i,j)\in\overline{1,N_{\varepsilon}}^2\;\big|\;x_{\varepsilon}^i-x_{\varepsilon}^j=(0,0,\varepsilon)^T,\;i<j\big\}\big)\label{defn:Y_eps_z_cardinal}\\
& \bullet \Upsilon_{\varepsilon}^y:\big\{(i,j)\in\overline{1,N_{\varepsilon}}^2\;\big|\;x_{\varepsilon}^i-x_{\varepsilon}^j=(0,0,\varepsilon)^T,\;i<j\big\}\rightarrow\overline{1,Z_{\varepsilon}}\text{ a bijection;}\notag\\
& \bullet y_{\varepsilon}^{z,m}=\dfrac{1}{2}(x_{\varepsilon}^i+x_{\varepsilon}^j),\text{ where }m=\Upsilon_{\varepsilon}^z(i,j);\notag\\
& \bullet \mathcal{Y}_{\varepsilon}^z=\bigg\{y_{\varepsilon}^{z,m}\in\mathcal{Y}_{\varepsilon}\;\bigg|\;y_{\varepsilon}^{z,m}=\dfrac{1}{2}(x_{\varepsilon}^i+x_{\varepsilon}^j),\;m=\Upsilon_{\varepsilon}^z(i,j)\bigg\};\label{defn:Y_eps_z}\\
& \bullet \mathcal{P}_{\varepsilon}^{z,m}\;\text{the parallelipiped centered in }y_{\varepsilon}^{z,m}\text{, defined by }\mathcal{P}_{\varepsilon}^{z,m}=y_{\varepsilon}^{z,m}+T_z\mathcal{C}^{\alpha}\text{, where}\label{defn:P_z}\\
&\hspace{14mm} T_z\mathcal{C}^{\alpha}=\bigg[-\dfrac{\varepsilon^{\alpha}}{2p},\dfrac{\varepsilon^{\alpha}}{2p}\bigg]\times\bigg[-\dfrac{\varepsilon^{\alpha}}{2q},\dfrac{\varepsilon^{\alpha}}{2q}\bigg]\times\bigg[-\dfrac{r\varepsilon-\varepsilon^{\alpha}}{2r},\dfrac{r\varepsilon-\varepsilon^{\alpha}}{2r}\bigg];\\
& \bullet \mathcal{T}_z^m\text{ be the union of the four transparent faces of }\mathcal{P}_{\varepsilon}^{z,m}\text{ that have the length equal to }\dfrac{r\varepsilon-\varepsilon^{\alpha}}{r},\notag\\ 
&\text{which are represented in Figure \ref{fig:T_z}.\label{defn:T_z_m}}
\end{align}

\begin{figure}[h!]
     \centering
     \begin{subfigure}[b]{0.3\textwidth}
         \centering
         \includegraphics[width=\textwidth]{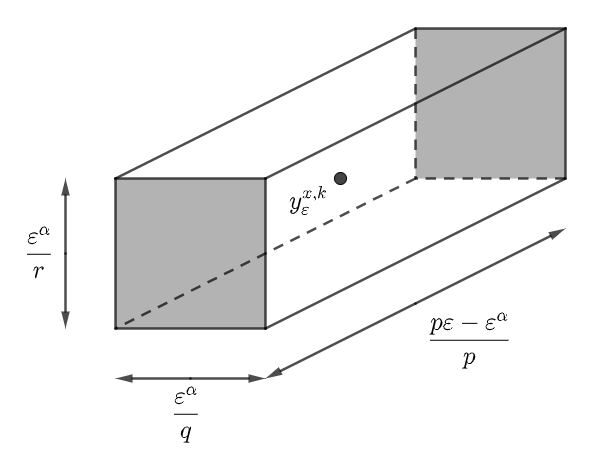}
         \caption{The parallelipiped $\mathcal{P}_{\varepsilon}^{x,k}$, with the center in $y_{\varepsilon}^{x,k}$, with lateral tran\-spa\-rent faces $\mathcal{T}_x^k$.}
         \label{fig:T_x}
     \end{subfigure}
     \hfill
     \begin{subfigure}[b]{0.3\textwidth}
         \centering
         \includegraphics[width=\textwidth]{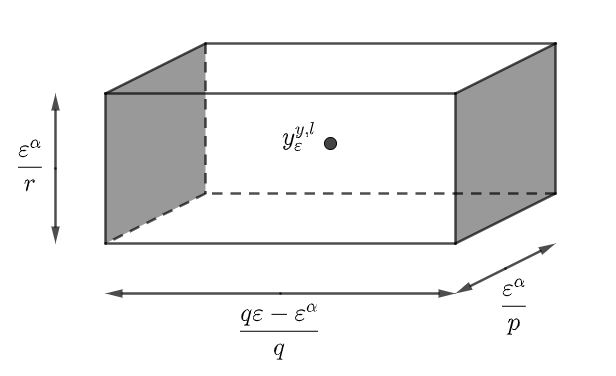}
         \caption{The parallelipiped $\mathcal{P}_{\varepsilon}^{y,l}$, with the center in $y_{\varepsilon}^{y,l}$, with lateral tran\-spa\-rent faces $\mathcal{T}_y^l$.}
         \label{fig:T_y}
     \end{subfigure}
     \hfill
     \begin{subfigure}[b]{0.3\textwidth}
         \centering
         \includegraphics[width=\textwidth]{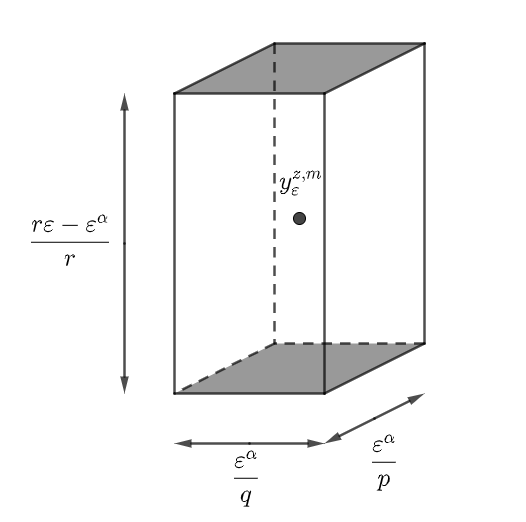}
         \caption{The parallelipiped $\mathcal{P}_{\varepsilon}^{z,m}$, with the center in $y_{\varepsilon}^{z,m}$, with lateral tran\-spa\-rent faces $\mathcal{T}_z^m$.}
         \label{fig:T_z}
     \end{subfigure}
        \caption{The three types parallelipipeds with centers from $\mathcal{Y}_{\varepsilon}$.}
\end{figure}

\begin{remark} We can already see that we need to set $\varepsilon^{\alpha-1}<\text{min}\{p,q,r\}$, otherwise we have $\dfrac{\varepsilon^{\alpha-1}}{p}\geq 1\Rightarrow \dfrac{\varepsilon^{\alpha}}{2p}\geq\dfrac{\varepsilon}{2}$ and, in the same way, $\dfrac{\varepsilon^{\alpha}}{2q}\geq\dfrac{\varepsilon}{2}$ and $\dfrac{\varepsilon^{\alpha}}{2r}\geq\dfrac{\varepsilon}{2}$, hence the inclusions from the family $\mathcal{C}_{\varepsilon}$ are not disjoint anymore and they overlap. More specifically, the gray parallelipipeds from Figure \ref{fig:nematiccage1} cannot be constructed anymore. Since the parameters $p$, $q$ and $r$ are fixed and we are interested what happens when $\varepsilon\rightarrow 0$, then the condition $\varepsilon^{\alpha-1}<\text{min}\{p,q,r\}$ implies that $\alpha\geq1$.  If $\alpha=1$, then it is easy to see that the volume of the scaffold does not tend to zero as $\varepsilon\rightarrow 0$, so we are not in the dilute regime anymore.
\end{remark}

\subsection{Volume and surface area of the scaffold}\label{subsection:appendix1}

\begin{prop}\label{prop:volumegoesto0}
The volume of the scaffold $\mathcal{N}_{\varepsilon}$ tends to $0$ as $\varepsilon\rightarrow 0$.
\end{prop}

\begin{proof}
According to \eqref{defn:L_0_l_0_h_0}, \eqref{defn:points1}, \eqref{defn:points2}, \eqref{defn:Y_eps}, \eqref{defn:Y_eps_x_cardinal}, \eqref{defn:Y_eps_y_cardinal} and \eqref{defn:Y_eps_z_cardinal}, we have:
\vspace{2mm}

\begin{tabular}{llll}
$\bullet\; N_{\varepsilon}<\dfrac{L_0 l_0 h_0}{\varepsilon^3}$; & $\bullet\; X_{\varepsilon}<\bigg(\dfrac{L_0}{\varepsilon}-1\bigg)\cdot\dfrac{l_0h_0}{\varepsilon^2}$; & $\bullet\; Y_{\varepsilon}<\bigg(\dfrac{l_0}{\varepsilon}-1\bigg)\dfrac{L_0h_0}{\varepsilon^2}$; & $\bullet\; Z_{\varepsilon}<\bigg(\dfrac{h_0}{\varepsilon}-1\bigg)\dfrac{L_0l_0}{\varepsilon^2}$.
\end{tabular}

Furthermore, we have:
\begin{align*}
|\mathcal{N}_{\varepsilon}|&=N_{\varepsilon}\cdot\dfrac{\varepsilon^{3\alpha}}{pqr}+X_{\varepsilon}\cdot\dfrac{\varepsilon^{2\alpha}}{pqr}\big(p\varepsilon-\varepsilon^{\alpha}\big)+Y_{\varepsilon}\cdot\dfrac{\varepsilon^{2\alpha}}{pqr}\big(q\varepsilon-\varepsilon^{\alpha}\big)+Z_{\varepsilon}\cdot\dfrac{\varepsilon^{2\alpha}}{pqr}\big(r\varepsilon-\varepsilon^{\alpha}\big),
\end{align*}
where $\dfrac{\varepsilon^{3\alpha}}{pqr}$ represents the volume of a parallelipiped defined in \eqref{defn:ceps} and $\dfrac{\varepsilon^{2\alpha}}{pqr}\big(p\varepsilon-\varepsilon^{\alpha}\big)$, $\dfrac{\varepsilon^{2\alpha}}{pqr}\big(q\varepsilon-\varepsilon^{\alpha}\big)$ and \linebreak $\dfrac{\varepsilon^{2\alpha}}{pqr}\big(r\varepsilon-\varepsilon^{\alpha}\big)$ represent the volume of a parallelipiped $\mathcal{P}_{\varepsilon}^{x,k}$, $\mathcal{P}_{\varepsilon}^{y,l}$ and, respectively, $\mathcal{P}_{\varepsilon}^{z,m}$, which are defined in \eqref{defn:P_x}, \eqref{defn:P_y} and \eqref{defn:P_z}. Hence: 
\begin{align*}
|\mathcal{N}_{\varepsilon}|&<\dfrac{L_0l_0h_0(p+q+r)}{pqr}\varepsilon^{2(\alpha-1)}-2\dfrac{L_0l_0h_0}{pqr}\varepsilon^{3(\alpha-1)}-\bigg(\dfrac{L_0l_0}{pq}+\dfrac{L_0h_0}{pr}+\dfrac{l_0h_0}{qr}\bigg)\varepsilon^{2\alpha-1}+\\
&\;\;\;\;+\dfrac{L_0l_0+L_0h_0+l_0h_0}{pqr}\varepsilon^{3\alpha-2}.
\end{align*}

Because $\alpha>1$, according to ($A_2$), then $2(\alpha-1)>0$, $3(\alpha-1)>0$, $2\alpha-1>0$ and $3\alpha-2>0$, therefore $|\mathcal{N}_{\varepsilon}|\rightarrow 0$ as $\varepsilon\rightarrow 0$. 
\end{proof}

\begin{prop}\label{prop:C_s}
There exists an $\varepsilon$-independent constant $C_s=C_s(p,q,r,\Omega)>0$ such that: $$\lim_{\varepsilon\rightarrow 0}\dfrac{\varepsilon^{3}}{\varepsilon^{\alpha}(\varepsilon-\varepsilon^{\alpha})}|\partial\mathcal{N}_{\varepsilon}|<C_s.$$
\end{prop}

\begin{proof}

Using the same tehnique as in \cref{prop:volumegoesto0}, we have, considering relations \eqref{defn:T_x_k}, \eqref{defn:T_y_l}, \eqref{defn:T_z_m} and \eqref{defn:surface}: \begin{align*}
|\partial\mathcal{N}_{\varepsilon}|&<N_{\varepsilon}\cdot\varepsilon^{2\alpha}\cdot\dfrac{2(p+q+r)}{pqr}+X_{\varepsilon}\cdot\varepsilon^{\alpha}(p\varepsilon-\varepsilon^{\alpha})\cdot\dfrac{2(q+r)}{pqr}+\\
&\hspace{5mm}+Y_{\varepsilon}\cdot\varepsilon^{\alpha}(q\varepsilon-\varepsilon^{\alpha})\cdot\dfrac{2(p+r)}{pqr}+Z_{\varepsilon}\cdot\varepsilon^{\alpha}(r\varepsilon-\varepsilon^{\alpha})\cdot\dfrac{2(p+q)}{pqr}\\
&<C(p,q,r,L_0,l_0,h_0)\cdot \varepsilon^{\alpha-3}\cdot \big((p+q+r)\varepsilon-2\varepsilon^{\alpha}\big).
\end{align*} 
Hence 
$$\lim_{\varepsilon\rightarrow 0}\dfrac{\varepsilon^{3}}{\varepsilon^{\alpha}(\varepsilon-\varepsilon^{\alpha})}|\partial\mathcal{N}_{\varepsilon}|<C(p,q,r,L_0,l_0,h_0)\cdot \lim_{\varepsilon\rightarrow 0}\dfrac{(p+q+r)\varepsilon-2\varepsilon^{\alpha}}{\varepsilon-\varepsilon^{\alpha}}<+\infty.$$ We denote $C_{s}$ the constant obtained in the last inequality.
\end{proof}

\begin{prop}\label{prop:outer_surfaces_go_to_0}
Let $\partial\mathcal{N}_{\varepsilon}^{\mathcal{S}}$ be the set defined in \eqref{defn:surface}. Then, for $\varepsilon\rightarrow 0$, we have $|\partial\mathcal{N}_{\varepsilon}^{\mathcal{S}}|\rightarrow 0$. 
\end{prop}

\begin{proof}
According to \eqref{defn:surface}, we have:
\begin{align*}
\partial\mathcal{N}_{\varepsilon}^{\mathcal{S}}=\bigg(\bigcup_{i=1}^{N_{\varepsilon,2}}\mathcal{S}^i\bigg),
\end{align*}
therefore, we can write:
\begin{align*}
\big|\partial\mathcal{N}_{\varepsilon}^{\mathcal{S}}\big|&\leq \sum_{i=1}^{N_{\varepsilon,2}}\big|\mathcal{S}^i\big|\leq\sum_{i=1}^{N_{\varepsilon,2}}\big|\mathcal{C}_{\varepsilon}^{i}\big|\leq \sum_{i=1}^{N_{\varepsilon,2}}\dfrac{\varepsilon^{2\alpha}(p+q+r)}{pqr}\leq \dfrac{\varepsilon^{2\alpha}(p+q+r)}{pqr}\cdot N_{\varepsilon,2},
\end{align*}
where we have used \eqref{defn:S_i} and \eqref{defn:ceps}. Since $N_{\varepsilon,2}$ counts only the parallelipipeds that are ``close" to the boundary of $\Omega$ (see \eqref{defn:N_eps_2}), then we can write:
\begin{align*}
N_{\varepsilon,2}<\dfrac{L_0\cdot l_0}{\varepsilon^2}+\dfrac{L_0\cdot h_0}{\varepsilon^2}+\dfrac{l_0\cdot h_0}{\varepsilon^2},
\end{align*}
where $L_0$, $l_0$ and $h_0$ are defined in \eqref{defn:L_0_l_0_h_0} and they describe the parallelipiped that contains the entire domain $\Omega$. From here, we obtain:
\begin{align*}
\big|\partial\mathcal{N}_{\varepsilon}^{\mathcal{S}}\big|&\leq \dfrac{\varepsilon^{2\alpha}(p+q+r)}{pqr}\cdot\dfrac{L_0l_0+l_0h_0+h_0L_0}{\varepsilon^2}\lesssim \varepsilon^{2(\alpha-1)}\rightarrow 0,\;\text{ as}\;\varepsilon\rightarrow 0,
\end{align*}
since $\alpha>1$.
\end{proof}

\subsection{Constructing an explicit extension of $Q$ inside the scaffold}\label{subsection:existence_of_extension}

The aim of this subsection is to prove that there exists a function $v\in H^1(\Omega)$ such that $v=Q$ on $\partial\mathcal{N}_{\varepsilon}$, $v=Q$ in $\Omega_{\varepsilon}$ and $\big\|\nabla v\big\|_{L^2(\Omega)}\lesssim \big\|\nabla Q\big\|_{L^2(\Omega_{\varepsilon})}$. 

In order to prove it, we first construct an explicit extension $u:\Omega_{\varepsilon}\cup\mathcal{N}_{\varepsilon}^{\mathcal{T}}\rightarrow\mathcal{S}_0$ such that $u\in H^1(\Omega_{\varepsilon}\cup\mathcal{N}_{\varepsilon}^{\mathcal{T}},\mathcal{S}_0)$ and there exists a constant $C$, independent of $\varepsilon$, for which we have
\begin{align*}
\big\|\nabla u\big\|_{L^2(\mathcal{N}_{\varepsilon}^{\mathcal{T}})}\leq C\big\|\nabla Q\big\|_{L^2(\Omega_{\varepsilon})},
\end{align*}
which implies $\big\|\nabla u\big\|_{L^2(\Omega_{\varepsilon}\cup\mathcal{N}_{\varepsilon}^{\mathcal{T}})}\leq C\big\|\nabla Q\big\|_{L^2(\Omega_{\varepsilon})}$. 

Then we construct $v:\Omega\rightarrow\mathcal{S}_0$ such that $v\in H^1(\Omega)$, $v\equiv u$ on $\Omega_{\varepsilon}\cup\mathcal{N}_{\varepsilon}^{\mathcal{T}}$, $v=u$ on $\partial\mathcal{N}_{\varepsilon}^{\mathcal{T}}$ and there exists a constant $c$ such that:
\begin{align*}
\big\|\nabla v\big\|_{L^2(\Omega)}\leq c\big\|\nabla u\big\|_{L^2(\Omega_{\varepsilon}\cup\mathcal{N}_{\varepsilon}^{\mathcal{T}})},
\end{align*}
which implies that $\big\|\nabla v\big\|_{L^2(\Omega)}\lesssim\big\|\nabla Q\big\|_{L^2(\Omega_{\varepsilon})}$, using the properties mentioned for $u$.

We prove first the following result.

\begin{lemma}\label{lemma:extension_u}
Let $z_0,a,b\in\mathbb{R}$, $a,b,z_0>0$ and let $A_{a,b}=\big\{(\rho\cos\theta,\rho\sin\theta)\;:\; 0\leq\theta< 2\pi,\;a<\rho<b\big\}$ be a  two dimensional annulus, $B_a=\big\{(\rho\cos\theta,\rho\sin\theta)\;:\;0\leq \theta<2\pi,\;0\leq\rho<a\big\}$ be a two dimensional ball with radius $a$, $\mathcal{A}_{a,b}^{z_0}=A_{a,b}\times(-z_0,z_0)\subset\mathbb{R}^3$ and $\mathcal{B}_a^{z_0}=B_a\times(-z_0,z_0)\subset\mathbb{R}^3$ be a three dimensional cylinder.

Let $Q\in H^1\big(\mathcal{A}_{1,2}^{z_0},\mathcal{S}_0\big)$. Then the function $u:\mathcal{B}_1^{z_0}\rightarrow\mathcal{S}_0$ defined for any $z\in(-z_0,z_0)$ as
\begin{equation}\label{eq:ext_def_u}
u(x,y,z)=\begin{cases}\varphi(\sqrt{x^2+y^2})\; Q\bigg(\bigg(\dfrac{2}{\sqrt{x^2+y^2}}-1\bigg)x,\bigg(\dfrac{2}{\sqrt{x^2+y^2}}-1\bigg)y,z\bigg)+\\
\hspace{3mm}+(1-\varphi(\sqrt{x^2+y^2}))\displaystyle{\fint_{A_{1,3/2}}Q(s,t,z)\emph{d}(s,t)},\;\emph{for}\;\dfrac{1}{2}\leq\sqrt{x^2+y^2}<1\\
\displaystyle{\fint_{A_{1,3/2}}Q(s,t,z)\emph{d}(s,t)},\;\emph{for}\;0\leq\sqrt{x^2+y^2}\leq\dfrac{1}{2}\end{cases}
\end{equation} 
is from $H^1(\mathcal{B}_1^{z_0},\mathcal{S}_0)$, where $\varphi\in C_c^{\infty}\bigg(\bigg(\dfrac{1}{2},\dfrac{3}{2}\bigg)\bigg)$ is the following bump function defined as
\begin{equation*}
\varphi(\rho)=\begin{cases}\exp\bigg\{4-\dfrac{4}{(2\rho-1)(3-2\rho)}\bigg\},\;\forall\rho\in\bigg(\dfrac{1}{2},\dfrac{3}{2}\bigg)\\
0,\;\forall\rho\in\mathbb{R}\setminus\bigg(\dfrac{1}{2},\dfrac{3}{2}\bigg)\end{cases},
\end{equation*}
the product $\varphi(\rho)\;Q$ represents product between a scalar and a Q-tensor and $\displaystyle{\fint}$ represents the average integral sign. Moreover, there exists a constant $c>0$, independent of $z_0$, such that: $\big\|u_{t}\big\|_{L^2(\mathcal{B}_1^{z_0})}\leq c\big\|Q_{t}\big\|_{L^2(\mathcal{A}_{1,3/2}^{z_0})}$ for any $t\in\{x,y,z\}$, where $u_t$ represents the partial derivative of $u$ with respect to $t$, and:
\begin{align*}
\|\nabla u\|_{L^2(\mathcal{B}_1^{z_0})}\leq c\|\nabla Q\|_{L^2(\mathcal{A}_{1,3/2}^{z_0})}.
\end{align*} 
\end{lemma}

\begin{proof}

First of all, we can assume without loss of generality that $Q$ and $u$ are scalar functions, instead of $Q$-tensors. Hence, we prove this lemma for each of the component of $Q$ and $u$. 

Let $T:A_{1/2,1}\rightarrow A_{1,3/2}$ be the reflection defined as:
\begin{equation*}
T(x,y)=\bigg(\bigg(\dfrac{2}{\sqrt{x^2+y^2}}-1\bigg)x,\bigg(\dfrac{2}{\sqrt{x^2+y^2}}-1\bigg)y\bigg):=(x',y'),\;\forall(x,y)\in A_{1/2,1}.
\end{equation*}

Then $T$ is invertible and also bi-Lipschitz.

Let $Q\in H^1(\mathcal{A}_{1,3/2}^{z_0})$ and $u$ defined by \eqref{eq:ext_def_u}. By Theorem 3.17 from \cite{Adams}, we can approximate the function $Q\in H^1(\mathcal{A}_{1,3/2}^{z_0})$ with smooth functions from $C^{\infty}\big(\overline{\mathcal{A}_{1,3/2}^{z_0}}\big)$.

Let $(Q_k)_{k\geq 1}\subset C^{\infty}\big(\overline{\mathcal{A}_{1,3/2}^{z_0}}\big)$ such that $Q_k\rightarrow Q$ strongly in $H^1(\mathcal{A}_{1,3/2}^{z_0})$ and, for any $k\geq 1$, let $u_k:\overline{\mathcal{B}_{1}^{z_0}}\rightarrow\mathbb{R}$ defined for all $z\in(-z_0,z_0)$ as:
\begin{equation*}
u_k(x,y,z)=\begin{cases}\varphi(\sqrt{x^2+y^2})\; Q_k(T(x,y),z)+\\
\hspace{3mm}+(1-\varphi(\sqrt{x^2+y^2}))\displaystyle{\fint_{A_{1,3/2}}Q_k(s,t,z)\text{d}(s,t)},\;\text{for}\;\dfrac{1}{2}\leq\sqrt{x^2+y^2}\leq 1\\
\displaystyle{\fint_{A_{1,3/2}}Q_k(s,t,z)\emph{d}(s,t)},\;\text{for}\;0\leq\sqrt{x^2+y^2}\leq\dfrac{1}{2}\end{cases}
\end{equation*}

By the above definition, we have that $u_k\in C^0\big(\overline{\mathcal{B}_1^{z_0}}\big)$ and that:

\begin{align*}
\dfrac{\partial u_k}{\partial x}(x,y,z)&=\begin{cases}\dfrac{x}{\sqrt{x^2+y^2}}\cdot \varphi'(\sqrt{x^2+y^2})\cdot\bigg(Q_k(x',y',z)-\displaystyle{\fint_{A_{1,3/2}}Q_k(s,t,z)\text{d}(s,t)}\bigg)+\\
\hspace{3mm}+\bigg(\dfrac{2y^2}{\big(\sqrt{x^2+y^2)}\big)^3}-1\bigg)\cdot \varphi(\sqrt{x^2+y^2})\cdot\bigg(\dfrac{\partial Q_k}{\partial x'}(x',y',z)\bigg)-\\
\hspace{3mm}-\dfrac{2xy}{\big(\sqrt{x^2+y^2}\big)^3}\cdot\varphi(\sqrt{x^2+y^2})\cdot \bigg(\dfrac{\partial Q_k}{\partial y'}(x',y',z)\bigg),\;\text{for}\;\dfrac{1}{2}\leq\sqrt{x^2+y^2}<1\\
0,\;\text{for}\;0\leq\sqrt{x^2+y^2}\leq\dfrac{1}{2}.
\end{cases}
\end{align*}

Since $\varphi\bigg(\dfrac{1}{2}\bigg)=\varphi'\bigg(\dfrac{1}{2}\bigg)=0$ and $\varphi\in C^{\infty}\bigg(\bigg(\dfrac{1}{2},\dfrac{3}{2}\bigg)\bigg)$, then we have $\dfrac{\partial u_k}{\partial x}\in C^0\big(\overline{\mathcal{B}_1^{z_0}}\big)$. Moreover, we obtain:
\begin{align*}
\dfrac{1}{3}\int_{A_{1/2,1}}\bigg|\dfrac{\partial u_k}{\partial x}\bigg|^2&(x,y,z)\text{d}(x,y) \leq \int_{A_{1/2,1}}\bigg|\dfrac{2xy}{\big(\sqrt{x^2+y^2}\big)^3}\bigg|^2 \big|\varphi(\sqrt{x^2+y^2})\big|^2\bigg|\dfrac{\partial Q_k}{\partial y'}\bigg|^2(x',y',z)\text{d}(x,y)+\\
&\hspace{2mm}+\int_{A_{1/2,1}}\bigg|\dfrac{2y^2}{\big(\sqrt{x^2+y^2}\big)^3}-1\bigg|^2\big|\varphi(\sqrt{x^2+y^2})\big|^2\bigg|\dfrac{\partial Q_k}{\partial x'}\bigg|^2(x',y',z)\text{d}(x,y)+\\
&\hspace{2mm}+\int_{A_{1/2,1}}\bigg|\dfrac{x}{\sqrt{x^2+y^2}}\bigg|^2\big|\varphi'(\sqrt{x^2+y^2})\big|^2\bigg|Q_k(x',y',z)-\displaystyle{\fint_{A_{1,3/2}}Q_k(s,t,z)\text{d}(s,t)}\bigg|^2\text{d}(x,y).
\end{align*}

By the definition of $\varphi$, we have $\|\varphi\|_{L^{\infty}(\mathbb{R})}=1$ and $\|\varphi'\|_{L^{\infty}(\mathbb{R})}=2\sqrt{9+6\sqrt{3}}\cdot e^{1-\sqrt{3}}\approx 4.23<5$ (the maximum is obtained for $\rho=1-\dfrac{1}{6}\sqrt{6\sqrt{3}-9}$).

For any $(x,y)$ such that $\dfrac{1}{2}\leq\sqrt{x^2+y^2}\leq 1$, we have: 
\begin{itemize}
\item[•] $\bigg|\dfrac{x}{\sqrt{x^2+y^2}}\bigg|^2=\dfrac{x^2}{x^2+y^2}\leq 1$;
\item[•] $\bigg|\dfrac{2xy}{x^2+y^2}\bigg|\leq 1\Rightarrow \bigg|\dfrac{2xy}{\big(\sqrt{x^2+y^2}\big)^3}\bigg|^2\leq\dfrac{1}{x^2+y^2}\leq 2$;
\item[•] $0\leq\dfrac{2y^2}{\big(\sqrt{x^2+y^2}\big)^3}\leq\dfrac{2(x^2+y^2)}{(\sqrt{x^2+y^2}\big)^3}\leq 4\Rightarrow -1\leq\dfrac{2y^2}{\big(\sqrt{x^2+y^2}\big)^3}-1\leq 3\Rightarrow \bigg|\dfrac{2y^2}{\big(\sqrt{x^2+y^2}\big)^3}\bigg|^2\leq 9.$
\end{itemize}

Therefore
\begin{align*}
\dfrac{1}{3}\int_{A_{1/2,1}}\bigg|\dfrac{\partial u_k}{\partial x}\bigg|^2(x,y,z)\text{d}(x,y)&\leq 25\int_{A_{1/2,1}}\bigg|Q_k(x',y',z)-\displaystyle{\fint_{A_{1,3/2}}Q_k(s,t,z)\text{d}(s,t)}\bigg|^2\text{d}(x,y)+\\
&\hspace{2mm}+9\int_{A_{1/2,1}}\bigg|\dfrac{\partial Q_k}{\partial x'}\bigg|^2(x',y',z)\text{d}(x,y)+4\int_{A_{1/2,1}}\bigg|\dfrac{\partial Q_k}{\partial y'}\bigg|^2(x',y',z)\text{d}(x,y).
\end{align*}

Using now the change of variables $(x',y')=T(x,y)$, we obtain
\begin{align*}
\text{d}(x,y)=\bigg(\dfrac{2}{\sqrt{(x')^2+(y')2}}-1\bigg)\text{d}(x',y')
\end{align*}
and since $(x',y')\in A_{1,3/2}$, we get $1\geq\dfrac{2}{\sqrt{(x')^2+(y')^2}}-1\geq\dfrac{1}{3}$, which implies
\begin{align*}
\int_{A_{1/2,1}}\bigg|\dfrac{\partial u_k}{\partial x}\bigg|^2(x,y,z)\text{d}(x,y) &\leq 75\int_{A_{1,3/2}}\bigg|Q_k(x',y',z)-\displaystyle{\fint_{A_{1,3/2}}Q_k(s,t,z)\text{d}(s,t)}\bigg|^2\text{d}(x',y')+\\
&\hspace{2mm}+27\int_{A_{1,3/2}}\bigg|\dfrac{\partial Q_k}{\partial x'}\bigg|^2(x',y',z)\text{d}(x',y')+12\int_{A_{1,3/2}}\bigg|\dfrac{\partial Q_k}{\partial y'}\bigg|^2(x',y',z)\text{d}(x',y').
\end{align*}

For the first term from the right hand side from the last inequality we can apply the Poincar\'{e} inequality, since $Q_k(\cdot,\cdot,z)\in H^1(A_{1,3/2})$, for any $z\in(-z_0,z_0)$. Therefore
\begin{align*}
\int_{A_{1/2,1}}\bigg|\dfrac{\partial u_k}{\partial x}\bigg|^2(x,y,z)\text{d}(x,y)&\leq 75\cdot C_P(A_{1,3/2})\int_{A_{1,3/2}}\bigg(\bigg|\dfrac{\partial Q_k}{\partial x'}\bigg|^2(x',y',z)+\bigg|\dfrac{\partial Q_k}{\partial y'}\bigg|^2(x',y',z)\bigg)\text{d}(x',y')+\\
&\hspace{2mm}+27\int_{A_{1,3/2}}\bigg|\dfrac{\partial Q_k}{\partial x'}\bigg|^2(x',y',z)\text{d}(x,y)+12\int_{A_{1,3/2}}\bigg|\dfrac{\partial Q_k}{\partial y'}\bigg|^2(x',y',z)\text{d}(x,y),
\end{align*}
where $C_P(A_{1,3/2})$ is the Poincar\'{e} constant for the two dimmensional domain $A_{1,3/2}$. Hence, there exists $c_1>0$, independent of $z_0$, such that:
\begin{align*}
\int_{A_{1/2,1}}\bigg|\dfrac{\partial u_k}{\partial x}\bigg|^2(x,y,z)\text{d}(x,y,z)&\leq c_1\int_{A_{1,3/2}}\bigg(\bigg|\dfrac{\partial Q_k}{\partial x'}\bigg|^2(x',y',z)+\bigg|\dfrac{\partial Q_k}{\partial y'}\bigg|^2(x',y',z)\bigg)\text{d}(x',y',z).
\end{align*}

Integrating now with respect to $z\in(-z_0,z_0)$, we get:
\begin{align*}
\int_{\mathcal{A}_{1/2,1}^{z_0}}\bigg|\dfrac{\partial u_k}{\partial x}\bigg|^2(x,y,z)\text{d}(x,y,z)&\leq c_1\int_{\mathcal{A}_{1,3/2}^{z_0}}\bigg(\bigg|\dfrac{\partial Q_k}{\partial x'}\bigg|^2(x',y',z)+\bigg|\dfrac{\partial Q_k}{\partial y'}\bigg|^2(x',y',z)\bigg)\text{d}(x',y',z)\Rightarrow\\
\Rightarrow\bigg\|\dfrac{\partial u_k}{\partial x}\bigg\|^2_{L^2(\mathcal{A}_{1/2,1}^{z_0})}&\leq c_1\big\|\nabla Q_k\|^2_{L^2(\mathcal{A}_{1,3/2}^{z_0})}.
\end{align*}

Using now the fact that $\dfrac{\partial u_k}{\partial x}\in C^0\big(\overline{\mathcal{B}_1^{z_0}}\big)$ and that $\dfrac{\partial u_k}{\partial x}(x,y,z)=0$ if $0\leq\sqrt{x^2+y^2}\leq\dfrac{1}{2}$, then we can write:
\begin{align*}
\bigg\|\dfrac{\partial u_k}{\partial x}\bigg\|^2_{L^2(\mathcal{B}_1^{z_0})}&\leq c_1\big\|\nabla Q_k\big\|^2_{L^2(\mathcal{A}_{1,3/2}^{z_0})}.
\end{align*}

In a similar fashion, $\dfrac{\partial u_k}{\partial y},\dfrac{\partial u_k}{\partial z}\in C^0\big(\overline{\mathcal{B}_1^{z_0}}\big)$, where
\begin{align*}
\dfrac{\partial u_k}{\partial y}(x,y,z)&=\begin{cases}\dfrac{y}{\sqrt{x^2+y^2}}\cdot \varphi'(\sqrt{x^2+y^2})\cdot\bigg(Q_k(x',y',z)-\displaystyle{\fint_{A_{1,3/2}}Q_k(s,t,z)\text{d}(s,t)}\bigg)-\\
\hspace{3mm}-\dfrac{2xy}{\big(\sqrt{x^2+y^2}\big)^3}\cdot\varphi(\sqrt{x^2+y^2})\cdot \bigg(\dfrac{\partial Q_k}{\partial x'}(x',y',z)\bigg)+\\
\hspace{3mm}+\bigg(\dfrac{2x^2}{\big(\sqrt{x^2+y^2)}\big)^3}-1\bigg)\cdot \varphi(\sqrt{x^2+y^2})\cdot\bigg(\dfrac{\partial Q_k}{\partial y'}(x',y',z)\bigg),\;\text{for}\;\dfrac{1}{2}\leq\sqrt{x^2+y^2}\leq 1\\
0,\;\text{for}\;0\leq\sqrt{x^2+y^2}\leq\dfrac{1}{2}
\end{cases}
\end{align*}
with
\begin{align*}
\bigg\|\dfrac{\partial u_k}{\partial y}\bigg\|^2_{L^2(\mathcal{B}_1^{z_0})}&\leq c_1\big\|\nabla Q_k\big\|^2_{L^2(\mathcal{A}_{1,3/2}^{z_0})}.
\end{align*}
and
\begin{align*}
\dfrac{\partial u_k}{\partial z}(x,y,z)=\begin{cases}\varphi(\sqrt{x^2+y^2})\; \dfrac{\partial Q_k}{\partial z}(x',y',z)+\\
\hspace{3mm}+(1-\varphi(\sqrt{x^2+y^2}))\displaystyle{\fint_{A_{1,3/2}}\dfrac{\partial Q_k}{\partial z}(s,t,z)\text{d}(s,t)},\;\text{for}\;\dfrac{1}{2}\leq\sqrt{x^2+y^2}\leq 1\\
\displaystyle{\fint_{A_{1,3/2}}\dfrac{\partial Q_k}{\partial z}(s,t,z)\emph{d}(s,t)},\;\text{for}\;0\leq\sqrt{x^2+y^2}\leq\dfrac{1}{2}\end{cases}
\end{align*}
since $Q\in C^{\infty}\big(\overline{\mathcal{A}_{1,3/2}^{z_0}}\big)$ and $A_{1,3/2}$ is independent of $z$, so we can move the derivative under the integral.

Then for any $(x,y,z)\in\mathcal{A}_{1/2,1}^{z_0}$, we have:
\begin{align*}
\dfrac{1}{2}\int_{A_{1/2,1}}\bigg|\dfrac{\partial u_k}{\partial z}\bigg|^2(x,y,z)\text{d}(x,y)&\leq \big\|\varphi\big\|^2_{L^{\infty}(\mathbb{R})}\int_{A_{1/2,1}}\bigg|\dfrac{\partial Q_k}{\partial z}\bigg|^2(x',y',z)\text{d}(x,y)+\\
&\hspace{2mm}+\big\|1-\varphi\big\|^2_{L^{\infty}(\mathbb{R})}\int_{A_{1/2,1}}\bigg|\displaystyle{\fint_{A_{1,3/2}}\dfrac{\partial Q_k}{\partial z}(s,t,z)\text{d}(s,t)}\bigg|^2\text{d}(x,y)\\
&\leq\int_{A_{1,3/2}}\bigg|\dfrac{\partial Q_k}{\partial z}\bigg|^2(x',y',z)\text{d}(x',y')+\dfrac{3\pi}{4}\cdot\dfrac{16}{25\pi^2}\int_{A_{1,3/2}}\bigg|\dfrac{\partial Q_k}{\partial z}\bigg|^2(x',y',z)\text{d}(x',y')
\end{align*}
and integrating with respect to $z\in(-z_0,z_0)$, we obtain
\begin{align*}
\int_{\mathcal{A}_{1/2,1}^{z_0}}\bigg|\dfrac{\partial u_k}{\partial z}\bigg|^2(x,y,z)\text{d}(x,y,z)&\leq\bigg(2+\dfrac{24}{25\pi}\bigg)\int_{\mathcal{A}_{1,3/2}^{z_0}}\bigg|\dfrac{\partial Q_k}{\partial z}\bigg|^2(x',y',z)\text{d}(x',y',z).
\end{align*}

For any $(x,y,z)\in\mathcal{B}_1^{z_0}$ with $0\leq\sqrt{x^2+y^2}\leq\dfrac{1}{2}$ we have:
\begin{align*}
\bigg|\dfrac{\partial u_k}{\partial z}\bigg|^2(x,y,z)&\leq\dfrac{16}{25\pi^2}\int_{A_{1,3/2}}\bigg|\dfrac{\partial Q_k}{\partial z}\bigg|^2(x',y',z)\text{d}(x',y')
\end{align*}
which implies
\begin{align*}
\int_{\overline{B_{1/2}}}\bigg|\dfrac{\partial u_k}{\partial z}\bigg|^2(x,y,z)\text{d}(x,y)&\leq\dfrac{4}{25\pi}\int_{A_{1,3/2}}\bigg|\dfrac{\partial Q_k}{\partial z}\bigg|^2(x',y',z)\text{d}(x',y',z)
\end{align*}
and from here we obtain that
\begin{align*}
\bigg\|\dfrac{\partial u_k}{\partial z}\bigg\|^2_{L^2(\mathcal{B}_1^{z_0})}&\leq\bigg(3+\dfrac{3}{25\pi}\bigg)\big\|\nabla Q_k\big\|^2_{L^2(\mathcal{A}_{1,3/2}^{z_0})}. 
\end{align*}

Now we prove that we can control $\|u_k\|_{L^2(\mathcal{B}_1^{z_0})}$ with $\|Q_k\|_{L^2(\mathcal{A}_{1,3/2}^{z_0})}$. For any $(x,y)\in A_{1/2,1}$, we have:
\begin{align*}
|u_k|^2(x,y,z)&\leq |\varphi(\sqrt{x^2+y^2}|^2\big|Q_k(x',y',z)|^2+(1-\varphi(\sqrt{x^2+y^2}))^2\bigg|\displaystyle{\fint_{A_{1,3/2}}Q_k(s,t,z)\text{d}(s,t)}\bigg|^2
\end{align*}
which implies
\begin{align*}
\dfrac{1}{2}\int_{A_{1/2,1}}|u_k|^2(x,y,z)\text{d}(x,y)&\leq\big\|1-\varphi\|^2_{L^{\infty}(\mathbb{R})}\cdot\dfrac{3}{4\pi}\cdot\dfrac{16}{25\pi^2}\int_{A_{1,3/2}}|Q_k|^2(x',y',z)\text{d}(x,y)+\\
&\hspace{2mm}+\big\|\varphi\|^2_{L^{\infty}(\mathbb{R})}\int_{A_{1/2,1}}|Q_k|^2(x',y',z)\text{d}(x,y).
\end{align*}

Using the same change of variables, the same bounds for $\varphi$ and for $1-\varphi$ and integrating with respect to $z\in(-z_0,z_0)$, we get:
\begin{align*}
\big\|u_k\big\|^2_{L^2(\mathcal{A}_{1/2,1}^{z_0})}\leq\bigg(2+\dfrac{24}{25\pi}\bigg)\big\|Q_k\big\|^2_{L^2(\mathcal{A}_{1,3/2}^{z_0})}.
\end{align*}

For any $(x,y)\in\overline{B_{1/2}}$ we have:
\begin{align*}
|u_k|^2(x,y,z)=\bigg|\displaystyle{\fint_{A_{1,3/2}}Q_k(s,t,z)\text{d}(s,t)}\bigg|^2
\end{align*}
which implies
\begin{align*}
\big\|u_k\big\|^2_{L^2(\overline{B_{1/2}}\times(-z_0,z_0))}\leq \dfrac{4}{25\pi}\big\|Q_k\big\|^2_{L^2(\mathcal{A}_{1,3/2}^{z_0})},
\end{align*}
hence
\begin{align*}
\big\|u_k\big\|^2_{L^2(\mathcal{B}_1^{z_0})}\leq\bigg(3+\dfrac{3}{25\pi}\bigg)\big\|Q_k\big\|^2_{L^2(\mathcal{A}_{1,3/2}^{z_0})}.
\end{align*}

Combining all the relations that we have obtained, we see that for any $k\geq 1$ we have:
\begin{itemize}
\item[•] $u_k\in H^1(\mathcal{B}_1^{z_0})$;
\item[•] $\big\|u_k\big\|_{L^2(\mathcal{B}_1^{z_0})}\leq c_2\big\|Q_k\big\|_{L^2(\mathcal{A}_{1,3/2}^{z_0})}$, where $c_2=\sqrt{3+\dfrac{3}{25\pi}}$;
\item[•] $\big\|\nabla u_k\big\|_{L^2(\mathcal{B}_1^{z_0})}\leq c_3\big\|\nabla Q_k\big\|_{L^2(\mathcal{A}_{1,3/2}^{z_0})}$, where $c_3=\max\big\{\sqrt{c_1},c_2\big\}$.
\end{itemize}

Now, if we repeat the same argument (as the one used in order to achieve the $L^2$ control between $u_k$ and $Q_k$) for the functions $(u_k-u)$, for any $k\geq 1$, we get:
\begin{align*}
\big\|u_k-u\big\|_{L^2(\mathcal{B}_1^{z_0})}\leq c_2\big\|Q_k-Q\big\|_{L^2(\mathcal{A}_{1,3/2}^{z_0})}
\end{align*}
and since $Q_k\rightarrow Q$ strongly in $H^1(\mathcal{A}_{1,3/2}^{z_0})$, hence in $L^2(\mathcal{A}_{1,3/2}^{z_0})$, we obtain that $u_k\rightarrow u$ strongly in $L^2(\mathcal{B}_1^{z_0})$.

Because $Q_k\rightarrow Q$ strongly in $H^1(\mathcal{A}_{1,3/2}^{z_0})$, then $(Q_k)_{k\geq 1}$ is a bounded sequence in $H^1(\mathcal{A}_{1,3/2}^{z_0})$ and using the inequalities proved before, we get that $(u_k)_{k\geq 1}$ is a bounded sequence in $H^1(\mathcal{B}_1^{z_0})$, therefore there exists a subsequence $(u_{k_j})_{j\geq 1}$ which has the property that $u_{k_j}\rightharpoonup u_0$, with $u_0\in H^1(\mathcal{B}_1^{z_0})$. From here, we have the following convergences in $L^2(\mathcal{B}_1^{z_0})$: $u_{k_j}\rightharpoonup u_0$ and $u_{k_j}\rightarrow u$, so $u=u_0$ a.e. in $\mathcal{B}_1^{z_0}$. However, since $u_0\in H^1(\mathcal{B}_1^{z_0})$, we obtain that $u\in H^1(\mathcal{B}_1^{z_0})$ with $\nabla u=\nabla u_0$ a.e. in $\mathcal{B}_1^{z_0}$.

Let $\tilde{u}_x:\mathcal{B}_1^{z_0}\rightarrow\mathbb{R}$ be the function defined as:

\begin{align*}
\tilde{u}_x(x,y,z)=\begin{cases}\dfrac{x}{\sqrt{x^2+y^2}}\cdot \varphi'(\sqrt{x^2+y^2})\cdot\bigg(Q(x',y',z)-\displaystyle{\fint_{A_{1,3/2}}Q(s,t,z)\text{d}(s,t)}\bigg)+\\
\hspace{3mm}+\bigg(\dfrac{2y^2}{\big(\sqrt{x^2+y^2)}\big)^3}-1\bigg)\cdot \varphi(\sqrt{x^2+y^2})\cdot\bigg(\dfrac{\partial Q}{\partial x'}(x',y',z)\bigg)-\\
\hspace{3mm}-\dfrac{2xy}{\big(\sqrt{x^2+y^2}\big)^3}\cdot\varphi(\sqrt{x^2+y^2})\cdot \bigg(\dfrac{\partial Q}{\partial y'}(x',y',z)\bigg),\;\text{for}\;\dfrac{1}{2}\leq\sqrt{x^2+y^2}<1\\
0,\;\text{for}\;0\leq\sqrt{x^2+y^2}\leq\dfrac{1}{2}.
\end{cases}
\end{align*}
for every $z\in(-z_0,z_0)$.

Using the same argument as before (we only control the $L^2$ norm), we can see that:
\begin{align*}
\bigg\|\dfrac{\partial u_k}{\partial x}-\tilde{u}_x\bigg\|_{L^2(\mathcal{B}_1^{z_0})}\leq c_3\big\|\nabla Q_k-\nabla Q\big\|_{L^2(\mathcal{A}_{1,3/2}^{z_0})}
\end{align*}
and since $\nabla Q_k\rightarrow \nabla Q$ strongly in $L^2(\mathcal{A}_{1,3/2}^{z_0})$, we obtain that $\dfrac{\partial u_k}{\partial x}\rightarrow\tilde{u}_x$ strongly in $L^2(\mathcal{B}_1^{z_0})$. But at the same time, we have $\dfrac{\partial u_k}{\partial x}\rightharpoonup\dfrac{\partial u}{\partial x}$ weakly in $L^2(\mathcal{B}_1^{z_0})$, hence $\dfrac{\partial u}{\partial x}=\tilde{u}_x$ a.e. in $L^2(\mathcal{B}_1^{z_0})$ and $\dfrac{\partial u_k}{\partial x}\rightarrow\dfrac{\partial u}{\partial x}$ strongly in $L^2(\mathcal{B}_1^{z_0})$. Applying the same argument, we finally prove that $\nabla u_k\rightarrow\nabla u$ strongly in $L^2(\mathcal{B}_1^{z_0})$.

In the end, we see that:
\begin{align*}
\big\|\nabla u\big\|_{L^2(\mathcal{B}_1^{z_0})}&\leq\big\|\nabla u-\nabla u_k\big\|_{L^2(\mathcal{B}_1^{z_0})}+\big\|\nabla u_k\big\|_{L^2(\mathcal{B}_1^{z_0})}\\
&\leq \big\|\nabla u-\nabla u_k\big\|_{L^2(\mathcal{B}_1^{z_0})}+c_3\big\|\nabla Q_k\big\|_{L^2(\mathcal{A}_{1,3/2}^{z_0})}\\
&\leq \big\|\nabla u-\nabla u_k\big\|_{L^2(\mathcal{B}_1^{z_0})}+c_3\big\|\nabla Q_k-\nabla Q\big\|_{L^2(\mathcal{A}_{1,3/2}^{z_0})}+c_3\big\|\nabla Q\big\|_{L^2(\mathcal{A}_{1,3/2}^{z_0})}.
\end{align*} 
Because $\nabla u_k\rightarrow \nabla u$ strongly in $L^2(\mathcal{B}_1^{z_0})$ and because $\nabla Q_k\rightarrow \nabla Q$ strongly in $L^2(\mathcal{A}_{1,3/2}^{z_0})$, we conclude that
\begin{align*}
\big\|\nabla u\big\|_{L^2(\mathcal{B}_1^{z_0})}\leq c_3\big\|\nabla Q\big\|_{L^2(\mathcal{A}_{1,3/2}^{z_0})}.
\end{align*}
\end{proof}

Now we transform in several steps the sets $\mathcal{B}_1^{z_0}$ and $\mathcal{A}_{1,3/2}^{z_0}$ from the previous lemma into the corresponding regions related to $\Omega_{\varepsilon}$ and $\mathcal{N}_{\varepsilon}^{\mathcal{T}}$, that is, $\mathcal{B}_1^{z_0}$ into $\mathcal{P}_{\varepsilon}^{z,m}$, which is included in $\mathcal{N}_{\varepsilon}^{\mathcal{T}}$, and $\mathcal{A}_{1,3/2}^{z_0}$ into a parallelipiped with an interior hole, surrounding $\mathcal{P}_{\varepsilon}^{z,m}$, which is included in $\Omega_{\varepsilon}$ (the hole is exactly the parallelipiped $\mathcal{P}_{\varepsilon}^{z,m}$).

Let $T_2:\mathbb{R}^3\rightarrow\mathbb{R}^3$ be the transformation defined as:
\begin{align*}
T_2(x,y,z)=\begin{cases}
(0,0,z) &\text{if}\;x=y=0,\\
\bigg(\sqrt{x^2+y^2},\dfrac{4}{\pi}\sqrt{x^2+y^2}\arctan{\dfrac{y}{x}},z\bigg)&\text{if}\;|y|\leq x,\;x>0,\\
\bigg(-\sqrt{x^2+y^2},-\dfrac{4}{\pi}\sqrt{x^2+y^2}\arctan{\dfrac{y}{x}},z\bigg)&\text{if}\;|y|\leq -x,\;x<0,\\
\bigg(\dfrac{4}{\pi}\sqrt{x^2+y^2}\arctan{\dfrac{x}{y}},\sqrt{x^2+y^2},z\bigg)&\text{if}\;|x|\leq y,\;y>0,\\
\bigg(-\dfrac{4}{\pi}\sqrt{x^2+y^2}\arctan{\dfrac{x}{y}},-\sqrt{x^2+y^2},z\bigg)&\text{if}\;|x|\leq -y,\;y<0,
\end{cases}
\end{align*}
with the inverse
\begin{align*}
T_2^{-1}(\xi,\eta,z)=\begin{cases}
(0,0,z)&\text{if}\;\eta=\xi=0,\\
\bigg(\xi\cos{\dfrac{\pi}{4}\dfrac{\eta}{\xi}},\xi\sin{\dfrac{\pi}{4}\dfrac{\eta}{\xi}},z\bigg)&\text{if}\;|\eta|\leq|\xi|,\;\xi\neq 0,\\
\bigg(\eta\sin{\dfrac{\pi}{4}\dfrac{\xi}{\eta}},\eta\cos{\dfrac{\pi}{4}\dfrac{\xi}{\eta}},z\bigg)&\text{if}\;|\xi|\leq|\eta|,\;\eta\neq 0.
\end{cases}
\end{align*}
More specifically, $T_2(x,y,z)=(\Lambda_2(x,y),z)$, where $\Lambda_2$ is, according to \cite{Rehberg}, a bi-Lipschitz continuous map that maps, in $\mathbb{R}^2$, the unit ball into the unit cube and the Jacobian of $\Lambda_2$ is constant almost everywhere in $\mathbb{R}^2$. Hence, the transformation $T_2$ is bi-Lipschitz and the Jacobian of $T_2$ is constant almost everywhere in $\mathbb{R}^3$. 

In our case, we have: $T_2(\mathcal{B}_1^{z_0})=(-1,1)^2\times(-z_0,z_0)$ and $T_2(\mathcal{A}_{1,3/2}^{z_0})=\big((-3/2,3/2)^2\setminus(-1,1)^2\big)\times(-z_0,z_0)$. 

Let $u\in H^1(\mathcal{B}_1^{z_0})$, $Q\in H^1(\mathcal{A}_{1,3/2}^{z_0})$ and the constant $c>0$ such that $\big\|\nabla u\big\|_{L^2(\mathcal{B}_1^{z_0})}\leq c\big\|\nabla Q\big\|_{L^2(\mathcal{A}_{1,3/2}^{z_0})}$ be given the previous lemma, constant which is independent of $z_0$. Then we obtain that the functions $\tilde{u}:=u\circ T_2^{-1}\in H^1((-1,1)^2\times(-z_0,z_0))$ and $\tilde{Q}:=Q\circ T_2^{-1}\in H^1\big(\big((-3/2,3/2)^2\setminus(-1,1)^2\big)\times(-z_0,z_0)\big)$ and that there exists constants $c_{j}$ and $c_{J}$, which are also independent of $z_0$, but dependent on the constants given by the Jacobians of $T_2$ and $T_2^{-1}$, such that:
\begin{align*}
c_j\big\|\nabla \tilde{u}\big\|^2_{L^2(T_2(\mathcal{B}_1^{z_0}))}\leq \big\|\nabla u\big\|^2_{L^2(\mathcal{B}_1^{z_0})}\leq c_J\big\|\nabla\tilde{u}\big\|^2_{L^2(T_2(\mathcal{B}_1^{z_0}))}
\end{align*}
and
\begin{align*}
c_j\big\|\nabla \tilde{Q}\big\|^2_{L^2(T_2(\mathcal{A}_{1,3/2}^{z_0}))}\leq \big\|\nabla Q\big\|^2_{L^2(\mathcal{A}_{1,3/2}^{z_0})}\leq c_J\big\|\nabla\tilde{Q}\big\|^2_{L^2(T_2(\mathcal{A}_{1,3/2}^{z_0}))}.
\end{align*}

Hence, the inequality $\big\|\nabla u\big\|_{L^2(\mathcal{B}_1^{z_0})}\leq c\big\|\nabla Q\big\|_{L^2(\mathcal{A}_{1,3/2}^{z_0})}$ implies that there exists a constant $c_0$, also independent of $z_0$, such that:
\begin{align*}
\big\|\nabla\tilde{u}\big\|_{L^2(T_2(\mathcal{B}_1^{z_0}))}\leq c_0\big\|\nabla\tilde{Q}\big\|_{L^2(T_2(\mathcal{A}_{1,3/2}^{z_0}))}.
\end{align*}

Now if we use the transformation $T_3(x,y,z)=\varepsilon^{\alpha}(x,y,z)$ and denote $\overline{u}:=\tilde{u}\circ T_3^{-1}$ and $\overline{Q}:=\tilde{Q}\circ T_3^{-1}$, we get:
\begin{align*}
\varepsilon^{-\alpha}\big\|\nabla\overline{u}\big\|^2_{L^2((T_3\circ T_2)(\mathcal{B}_1^{z_0}))}=\big\|\nabla\tilde{u}\big\|^2_{L^2(T_2(\mathcal{B}_1^{z_0}))}\leq c_0^2\big\|\nabla\tilde{Q}\big\|^2_{L^2(T_2(\mathcal{A}_{1,3/2}^{z_0}))}=c_0^2\varepsilon^{-\alpha}\big\|\nabla\overline{Q}\big\|^2_{L^2((T_3\circ T_2)(\mathcal{B}_1^{z_0}))}
\end{align*}
which implies that
\begin{align*}
\big\|\nabla\overline{u}\big\|_{L^2((T_3\circ T_2)(\mathcal{B}_1^{z_0}))}\leq c_0\big\|\nabla\overline{Q}\big\|_{L^2((T_3\circ T_2)(\mathcal{A}_{1,3/2}^{z_0}))}.
\end{align*}

Since the constant $c_0$ is independent of the choice of $z_0$, we can have $z_0=\dfrac{r\varepsilon-\varepsilon^{\alpha}}{\varepsilon^{\alpha}}$.

The final change of variables is based on the mapping $T_4:\mathbb{R}^3\rightarrow \mathbb{R}^3$ defined as: $T_4(x,y,z)=\bigg(\dfrac{x}{2p},\dfrac{y}{2q},\dfrac{z}{2r}\bigg)$, where $p$, $q$ and $r$ are from relation \eqref{defn:initial_cube_alpha}. In this way, if we translate the origin into the center of the parallelipiped $\mathcal{P}_{\varepsilon}^{z,m}$, we obtain that $(T_4\circ T_3\circ T_2)(\mathcal{B}_1^{z_0})=\mathcal{P}_{\varepsilon}^{z,m}$ and we denote by $\mathcal{R}_{\varepsilon}^{z,m}$ the set $(T_4\circ T_3\circ T_2)(\mathcal{A}_{1,3/2}^{z_0})$, which is the box contained in $\Omega_{\varepsilon}$ (for $\varepsilon$ small enough) that ``surrounds" $\mathcal{P}_{\varepsilon}^{z,m}$.

The transformation $T_4$ is bi-Lipschitz and applying the same arguments as before, we obtain that there exists a function $u\in H^1(\mathcal{P}_{\varepsilon}^{z,m})$ ($u$ can be seen as $\overline{u}\circ(T_4^{-1})$) such that $u=Q$ on the ``contact" faces $\mathcal{T}_z^m$ of $\mathcal{P}_{\varepsilon}^{z,m}$ and an $\varepsilon$-independent constant $c>0$ such that 
\begin{align*}
\big\|\nabla u\big\|_{L^2(\mathcal{P}_{\varepsilon}^{z,m})}\leq c\big\|\nabla Q\big\|_{L^2(\mathcal{R}_{\varepsilon}^{z,m})}.
\end{align*}

Since the objects $\mathcal{R}_{\varepsilon}^{z,m}$ are pairwise disjoint (if we look only at the boxes surrounding the parallelipipeds with centers in $\mathcal{Y}_{\varepsilon}^{z}$), repeating the same argument for every parallelipiped of this type (with centers in $\mathcal{Y}_{\varepsilon}^z$) and then repeating the same argument for any parallelipiped from $\mathcal{N}_{\varepsilon}^{\mathcal{T}}$ (that is, with centers in $\mathcal{Y}_{\varepsilon}$), we obtain $u\in H^1(\Omega_{\varepsilon}\cup\mathcal{N}_{\varepsilon}^{\mathcal{T}},\mathcal{S}_0)$ an extension of $Q\in H^1(\Omega_{\varepsilon},\mathcal{S}_0)$ such that:
\begin{align*}
\begin{cases}
u=Q\;\text{in}\;\Omega_{\varepsilon}\\
u=Q\;\text{on}\;\partial\mathcal{N}_{\varepsilon}^{\mathcal{T}}\\
\big\|\nabla u\big\|_{L^2(\Omega_{\varepsilon}\cup\mathcal{N}_{\varepsilon}^{\mathcal{T}})}\leq c\big\|\nabla Q\big\|_{L^2(\Omega_{\varepsilon})}.
\end{cases}
\end{align*}

Let $\Omega'_{\varepsilon}=\Omega_{\varepsilon}\cup\mathcal{N}_{\varepsilon}^{\mathcal{T}}$. We want now to construct a function $v:\mathcal{N}_{\varepsilon}^{\mathcal{S}}\rightarrow\mathcal{S}_0$ such that $v=u$ on $\partial\Omega'_{\varepsilon}$ and that there exists a constant $c>0$, independent of $\varepsilon$ such that $\big\|\nabla v\big\|_{L^2(\mathcal{N}_{\varepsilon}^{\mathcal{S}})}\leq c\big\|\nabla u\big\|_{L^2(\Omega'_{\varepsilon})}$.

But in the case of the family $\mathcal{N}_{\varepsilon}^{\mathcal{S}}$, the parallelipipeds are pairwise disjoint for $\varepsilon$ small enough, therefore we can construct $v$ in each $\mathcal{C}_{\varepsilon}^i$, for every $i\in\overline{1,N_{\varepsilon}}$ and control, independent of $\varepsilon$, $\big\|\nabla v\big\|_{L^2(\mathcal{C}_{\varepsilon}^i)}$ with $\big\|\nabla u\big\|_{L^2(\mathcal{P}_{\varepsilon}^i)}$, where $\mathcal{R}_{\varepsilon}^i$ is the ``surrounding" box for $\mathcal{C}_{\varepsilon}^i$, constructed in the same way as $\mathcal{R}_{\varepsilon}^{z,m}$.

\begin{lemma}\label{lemma:extension_v}
Let $a,b\in\mathbb{R}^*_+$ with $a<b$, let $\mathcal{B}_{a}=\big\{x\in\mathbb{R}^3\;\big|\;|x|<a\big\}$ and let $\mathcal{A}_{a,b}=\mathcal{B}_{b}\setminus\overline{\mathcal{B}_a}$.

Let $u\in H^1\big(\mathcal{A}_{1,2},\mathcal{S}_0\big)$. Then the function $v:\mathcal{B}_1\rightarrow\mathcal{S}_0$ defined as
\begin{equation*}
v(x,y,z)=\begin{cases}\varphi(\sqrt{x^2+y^2+z^2})\; u\bigg(\bigg(\dfrac{2}{\sqrt{x^2+y^2+z^2}}-1\bigg)(x,y,z)\bigg)+\\
\hspace{3mm}+(1-\varphi(\sqrt{x^2+y^2+z^2}))\displaystyle{\fint_{\mathcal{A}_{1,3/2}}u(\xi,\eta,\tau)\emph{d}(\xi,\eta,\tau)},\;\emph{for}\;\dfrac{1}{2}\leq\sqrt{x^2+y^2+z^2}<1\\
\displaystyle{\fint_{\mathcal{A}_{1,3/2}}u(\xi,\eta,\tau)\emph{d}(\xi,\eta,\tau)},\;\emph{for}\;0\leq\sqrt{x^2+y^2+z^2}\leq\dfrac{1}{2}\end{cases}
\end{equation*} 
is from $H^1(\mathcal{B}_1,\mathcal{S}_0)$, where $\varphi\in C_c^{\infty}\bigg(\bigg(\dfrac{1}{2},\dfrac{3}{2}\bigg)\bigg)$ is the following bump function defined as
\begin{equation*}
\varphi(\rho)=\begin{cases}\exp\bigg\{4-\dfrac{4}{(2\rho-1)(3-2\rho)}\bigg\},\;\forall\rho\in\bigg(\dfrac{1}{2},\dfrac{3}{2}\bigg)\\
0,\;\forall\rho\in\mathbb{R}\setminus\bigg(\dfrac{1}{2},\dfrac{3}{2}\bigg)\end{cases},
\end{equation*}
the product $\varphi(\rho)\;u$ represents product between a scalar and a Q-tensor and $\displaystyle{\fint}$ represents the average integral sign. Moreover, there exists a constant $c>0$ such that: $\big\|v_{t}\big\|_{L^2(\mathcal{B}_1)}\leq c\big\|u_{t}\big\|_{L^2(\mathcal{A}_{1,3/2})}$ for any $t\in\{x,y,z\}$, where $v_t$ represents the partial derivative of $v$ with respect to $t$, and:
\begin{align*}
\|\nabla v\|_{L^2(\mathcal{B}_1)}\leq c\|\nabla u\|_{L^2(\mathcal{A}_{1,3/2})}.
\end{align*} 
\end{lemma}

\begin{remark}
\cref{lemma:extension_v} is just a different version of \cref{lemma:extension_u}. The proof follows the same steps as in \cref{lemma:extension_u}.
\end{remark}

Now if we use instead of $T_2$ the transformation $\Lambda_3$, from \cite{Rehberg}, which is a bi-Lipschitz mapping that transforms the unit ball into the unit cube, and then the transformations $T_3$ and $T_4$ as before, we end up with the function $v$ being an extension of $u$ that satisfies:
\begin{align*}
\begin{cases}
v\in H^1(\mathcal{C}^i_{\varepsilon})\\
v=u\;\text{on}\;\partial\mathcal{C}^i_{\varepsilon}\\
\big\|\nabla v\big\|_{L^2(\mathcal{C}^i_{\varepsilon})}\leq c\big\|\nabla u\big\|_{L^2(\mathcal{R}_{\varepsilon}^i)}.
\end{cases}
\end{align*}

Because the objects $\mathcal{R}_{\varepsilon}^i$ are pairwise disjoint for $\varepsilon$ small enough, we construct therefore an extension $v\in H^1(\Omega,\mathcal{S}_0)$ of $u\in H^1(\Omega'_{\varepsilon},\mathcal{S}_0)$ such that:
\begin{align*}
\begin{cases}
v=u\;\text{in}\;\Omega'_{\varepsilon}\Rightarrow v=Q\;\text{in}\;\Omega_{\varepsilon}\;\text{and}\;v=Q\;\text{on}\;\partial\mathcal{N}_{\varepsilon}^{\mathcal{T}}\\
v=u\;\text{on}\;\partial\mathcal{N}_{\varepsilon}^{\mathcal{S}}\Rightarrow v=Q\;\text{on}\;\partial\mathcal{N}_{\varepsilon}^{\mathcal{S}}\\
\big\|\nabla v\big\|_{L^2(\Omega)}\leq c\big\|\nabla u\big\|_{L^2(\Omega'_{\varepsilon})}\leq \tilde{c}\big\|\nabla Q\big\|_{L^2(\Omega_{\varepsilon})}.
\end{cases}
\end{align*}

So we have $v\in H^1(\Omega)$, $v=Q$ in $\Omega_{\varepsilon}$, $v=Q$ on $\partial\mathcal{N}_{\varepsilon}$ and there exists an $\varepsilon$-independent constant such that:
\begin{align*}
\big\|\nabla v\big\|_{L^2(\Omega)}\leq c\big\|\nabla Q\big\|_{L^2(\Omega_{\varepsilon})}.
\end{align*}

\subsection{Integrated energy densities}

In this subsection, we present two propositions that are used in order to prove that using relation \eqref{defn:f_hom_sym}, that is:
\begin{align*}
f_{hom}(Q)=\dfrac{2}{p}\int_{\partial\mathcal{C}}f_s(Q,\nu)\text{d}\sigma,
\end{align*}
then by using, for example, the choice of the \textit{surface energy density} defined in \eqref{defn:f_s_LDG}, which is:
\begin{align*}
f_s^{LDG}(Q,\nu)=\dfrac{p}{4}\bigg((a'-a)(\nu\cdot Q^2\nu)-(b'-b)(\nu\cdot Q^3\nu)+2(c'-c)(\nu\cdot Q^4\nu)\bigg),
\end{align*}
we can obtain the corresponding homogenised functional defined in \eqref{defn:f_hom_LDG}, that is:
\begin{align*}
f_{hom}^{LDG}(Q)=(a'-a)\,\text{tr}(Q^2)-(b'-b)\,\text{tr}(Q^3)+(c'-c)\,\big(\text{tr}(Q^2)\big)^2.
\end{align*}

More specifically, \cref{prop:int_en_dens_LDG} treats the case of the classical quartic polynomial in the scalar invariants of $Q$ for the \textit{bulk energy}, defined in \eqref{defn:f_b_LDG}, where the choice of the \textit{surface energy density} is in \eqref{defn:f_s_LDG}, and the more general version of it, defined in \eqref{defn:f_b_gen}, with the \textit{surface energy density} defined in \eqref{defn:f_s_gen}. Both cases have all of the terms from the picked \textit{surface energy densities} of the form $\nu\cdot Q^k\nu$, with $k\geq 2$. \cref{prop:int_en_dens_RP} treats only the Rapini-Papoular case, where the \textit{surface energy density} is defined in \eqref{defn:f_s_RP}. 

\begin{prop}\label{prop:int_en_dens_LDG}
For any $k\in\mathbb{N}$, $k\geq 2$ and for a fixed matrix $Q\in\mathcal{S}_0$, we have:
\begin{align*}
\emph{tr}(Q^k)=\dfrac{1}{2}\int_{\partial\mathcal{C}}\big(\nu\cdot Q^k\nu\big)\emph{d}\sigma,
\end{align*}
where $\partial\mathcal{C}$ is defined in \eqref{defn:C_x_C_y_C_z} and $\nu$ is the exterior unit normal to $\partial\mathcal{C}$.
\end{prop}
\begin{proof}
Let $Q^k=\begin{pmatrix}
q_{11,k} & q_{12,k} & q_{13,k}\\
q_{21,k} & q_{22,k} & q_{23,k}\\
q_{13,k} & q_{32,k} & q_{33,k}
\end{pmatrix}$. According to \eqref{defn:C_x_C_y_C_z}, we have $\partial\mathcal{C}=\mathcal{C}^x\cup\mathcal{C}^y\cup\mathcal{C}^z$. We compute first the intergral for $\mathcal{C}^x$, on which $\nu=(\pm 1,0,0)^T$:
\begin{align*}
\int_{\mathcal{C}^x}\big(\nu\cdot Q^k\nu\big)\text{d}\sigma=\int_{\mathcal{C}^x}\big((\pm 1,0,0)^T\cdot(\pm q_{11,k}, \pm q_{21,k}, \pm q_{31,k})^T\big)\text{d}\sigma=\int_{\mathcal{C}^x} q_{11,k}\text{d}\sigma =2q_{11,k},
\end{align*}
since $\mathcal{C}$ has length 1.

In the same way, we obtain:
\begin{align*}
\int_{\mathcal{C}^y}\big(\nu\cdot Q^k\nu\big)\text{d}\sigma=2q_{22,k}\hspace{3mm}\text{and}\hspace{3mm}\int_{\mathcal{C}^z}\big(\nu\cdot Q^k(x_0)\nu\big)\text{d}\sigma=2q_{33,k},
\end{align*}
from which we obtain
\begin{align*}
\int_{\partial\mathcal{C}}\big(\nu\cdot Q^k\nu\big)\text{d}\sigma=2\text{tr}(Q^k).
\end{align*}
\end{proof}

For the Rapini-Papoular case, we prove that:

\begin{prop}\label{prop:int_en_dens_RP}
For a fixed matrix $Q\in\mathcal{S}_0$, we have:
\begin{align*}
6\emph{tr}(Q^2)+4=\int_{\partial\mathcal{C}}\emph{tr}(Q-Q_{\nu})^2\emph{d}\sigma,
\end{align*}
where $\partial\mathcal{C}$ is defined in \eqref{defn:C_x_C_y_C_z}, $Q_{\nu}=\nu\otimes\nu-\mathbb{I}_3/3$, $\nu$ represents the exterior unit normal to $\partial\mathcal{C}$ and $\mathbb{I}_3$ is the $3\times 3$ identity matrix.
\end{prop}

\begin{proof}
First of all, we can see that $\text{tr}(Q-Q_{\nu})^2=\text{tr}(Q^2)-\text{tr}(QQ_{\nu})-\text{tr}(Q_{\nu}Q)+\text{tr}(Q^2_{\nu})$. 

According to \eqref{defn:C_x_C_y_C_z}, we have $\partial\mathcal{C}=\mathcal{C}_x\cup\mathcal{C}_y\cup\mathcal{C}_z$. On $\mathcal{C}^x$, we have $\nu=\begin{pmatrix}
\pm 1\\
0\\
0
\end{pmatrix}$ and $Q_{\nu}=\begin{pmatrix}
2/3 & 0 & 0\\
0 & -1/3 & 0\\
0 & 0 & -1/3
\end{pmatrix}$. Then $\text{tr}(Q_{\nu}^2)=\bigg(\dfrac{2}{3}\bigg)^2+\bigg(-\dfrac{1}{3}\bigg)^2+\bigg(-\dfrac{1}{3}\bigg)^2=\dfrac{2}{3}$. We also obtain $\text{tr}(Q_{\nu}^2)=\dfrac{2}{3}$ on $\mathcal{C}^y$ and $\mathcal{C}^z$. Therefore we obtain:
\begin{align*}
\int_{\partial\mathcal{C}}\text{tr}(Q^2_{\nu})\text{d}\sigma=6\cdot\dfrac{2}{3}=4,
\end{align*}
where the constant 6 comes from the total surface of the cube $\mathcal{C}$.

Let $Q=\begin{pmatrix}
q_{11} & q_{12} & q_{13}\\
q_{12} & q_{22} & q_{23}\\
q_{13} & q_{23} & -q_{11}-q_{22}
\end{pmatrix}$. Using the computations done earlier for $Q_{\nu}$ on $\mathcal{C}^x$, $\mathcal{C}^y$ and $\mathcal{C}^z$, we get:
\begin{align*}
\int_{\mathcal{C}^x}\big(\text{tr}(Q_{\nu}Q)+\text{tr}(QQ_{\nu})\big)\text{d}\sigma &=2\bigg(\dfrac{2q_{11}}{3}-\dfrac{q_{22}}{3}+\dfrac{q_{11}+q_{22}}{3}\bigg)=2q_{11}\\
\int_{\mathcal{C}^y}\big(\text{tr}(Q_{\nu}Q)+\text{tr}(QQ_{\nu})\big)\text{d}\sigma &=2\bigg(-\dfrac{q_{11}}{3}+\dfrac{2q_{22}}{3}+\dfrac{q_{11}+q_{22}}{3}\bigg)=2q_{22}\\
\int_{\mathcal{C}^z}\big(\text{tr}(Q_{\nu}Q)+\text{tr}(QQ_{\nu})\big)\text{d}\sigma &=2\bigg(-\dfrac{q_{11}}{3}-\dfrac{q_{22}}{3}-\dfrac{2q_{11}+2q_{22}}{3}\bigg)=-2q_{11}-2q_{22}.
\end{align*}

Combining the last three relations, we get that
\begin{align*}
\int_{\partial\mathcal{C}}\big(\text{tr}(Q_{\nu}Q)+\text{tr}(QQ_{\nu})\big)\text{d}\sigma &=0,
\end{align*}
from which the conclusion follows, with the observation that the constant 6 in front of $\text{tr}(Q^2)$ appears from the total surface of the cube $\mathcal{C}$, which has the length equal to 1.
\end{proof}

\begin{remark}
The constant 4 from \cref{prop:int_en_dens_RP} is neglected when we are studying the asymptotic behaviour of the minimisers of the functional \eqref{defn:F_eps_RP}, since adding constants do not influence the form and the existence of the possible minimisers.
\end{remark}
\end{appendix}

\section*{Acknowledgements}

The work of Razvan-Dumitru Ceuca is supported by the Spanish Ministry of Economy and Competitiveness MINECO through BCAM Severo Ochoa excellence accreditation SEV-2013-0323-17-1 (BES-2017-080630) and through project MTM2017-82184-R funded by (AEI/FEDER, UE) and acronym “DESFLU”. The author would like to thank Jamie Taylor and Giacomo Canevari for insightful discussions that have benefited this work and for their support granted during the making of it. The author would also like to thank his Ph.D. supervisor, Arghir-Dani Zarnescu, for the constant mathematical and moral support offered during the proccess of generating this work.


\begin{thebibliography}{99}
\bibitem{Adams}
Adams, R., \& Fournier, J. (2003). Sobolev Spaces. London: Academic Press. 

\bibitem{Alama1}
Alama, S., Bronsard, L. \& Lamy, X. (2016). Minimizers of the Landau-de Gennes energy around a spherical colloid particle. Archive for Rational Mechanics and Analysis. 222.1(2016): 427-450.

\bibitem{Alama2}
Alama, S., Bronsard, L., \& Lamy, X. (2018). Spherical particle in nematic liquid crystal under an external field: the Saturn ring regime. Journal of Nonlinear Science (2018): 1-23.

\bibitem{2PP}
Baldacchini, T. (2015). Three-Dimensional Microfabrication Using Two-Photon Polymerization. 1st Edition. Elsevier.

\bibitem{Bennett}
Bennett, T. P., D'Alessandro, G., \& Daly, K. R. (2014). Multiscale models of colloidal dispersion of particles in nematic liquid crystals. Physical Review E 90.6 (2014): 062505.

\bibitem{Berlyland}
Berlyland, L., Cioranescu, D., \& Golovaty, D. (2005). Homogenization of Ginzburg-Landau model for a nematic liquid crystal with inclusions. Journal de mathematiques pures et appliqu\' ees, 84(1), 97-136.

\bibitem{BPSM1}
Buscaglia, M., Bellini, T. , Chiccoli, C., Mantegazza, F., Pasini, P., Rotunno, M., \& Zannoni, C. (2006). Phys. Rev. E 2006, 74, 011706.

\bibitem{Calderer}
Calderer, M. C., DeSimone, A., Golovaty, D., \& Panchenko, A. (2014). An effective model for nematic liquid crystal composites with ferromagnetic inclusions. SIAM Journal of Applied Mathematics, 74(2), 237-262.


\bibitem{Canevari1}
Canevari, G., Ramaswamy, M., \& Majumdar, A. (2016). Radial symmetry on three-dimensional shells in the Landau-de Gennes theory. Physica D 314, 18-34.

\bibitem{Canevari2}
Canevari, G., Segatti, A., \& M. Veneroni, M. (2015). Morse's index formula in VMO on compact manifold with boundary. J. Funct. Anal. 269(10), 3043-3082.

\bibitem{Canevari3}
Canevari, G. \& Segatti, A. (2018). Defects in Nematic Shells: a $\Gamma$-convergence discrete-to-continuum approach. Arch. Ration. Mech. An. 229 (1), 125-186.

\bibitem{Canevari4}
Canevari, G., \& Segatti, A. (2018). Variational analysis of nematic shells. Trends in Applications of \linebreak Mathematics to Mechanics, 81-102, Springer-INdAM series 27.

\bibitem{CanevariZarnescu1}
Canevari, G., \& Zarnescu, A. D. (2019). Design of Effective Bulk Potentials for Nematic \linebreak Liquid Crystals Via Colloidal Homogenisation. Mathematical Models and Methods in Applied Sciences. 10.1142/S0218202520500086. 

\bibitem{CanevariZarnescu2}
Canevari, G., \& Zarnescu, A. D. (2020). Polydispersity and surface energy strength in nematic colloids. Mathematics in Engineering. 2. 290-312. 10.3934/mine.2020015. 

\bibitem{CioranescuDonato}
Cioranescu, D., \& Donato, P. (1999). An introduction to homogenization. Oxford Lecture Series in \linebreak Mathematics and Its Applications. 

\bibitem{deGennes}
De Gennes, P. G., \& Prost., J. (1993).
The Physics of Liquid Crystals.
International series of monographs on physics. Clarendon Press.

\bibitem{Rehberg}
Griepentrog, J., Höppner, W., Kaiser, H.-C., \& Rehberg, J. (2008). A bi-Lipschitz continuous, volume \linebreak preserving map from the unit ball onto a cube. Note di Matematica. 28. 177-193. 10.1285/i15900932v28n1p177. 

\bibitem{Ravnik1}
Jayasri, D., Ravnik, M., \& Žumer, Š. (2012). Shape tuning the colloidal assemblies in nematic liquid crystals. Soft Matter. 8. 1657.

\bibitem{Longa}
Longa, L., Montelesan, D., \& Trebin, H. R. (1987). An extension of the Landau-Ginzburg-de Gennes theory for liquid crystals. Liquid Crystal 2(6).

\bibitem{Mottram}
Mottram, N. J., \& Newton, C. JP. (2014). Introduction to Q-tensor theory. arXiv:1409.3542.

\bibitem{Ravnik2}
Muševič, I., Škarabot, M., Tkalec, U., Ravnik, M., \& Žumer, Š. (2006). Two-Dimensional Nematic Colloidal Crystals Self-Assembled by Topological Defects. Science
2006, 313, 954.

\bibitem{Ravnik3}
Ravnik, M., Škarabot, M., Žumer, Š., Tkalec, U., Poberaj, I., Babič, D.,
Osterman, N., \& Muševič, I. (2007). Entangled Nematic Colloidal Dimers and Wires. Phys. Rev. Lett. 2007, 99, 247801.

\bibitem{nematiccage}
Serra, F., Eaton, S., Cerbino, R., Buscaglia, M., Cerullo, G., Osellame, R., \& Bellini, T. (2013). Liquid Crystals: Nematic Liquid Crystals Embedded in Cubic Microlattices: Memory Effects and Bistable Pixels (Adv. Funct. Mater. 32/2013). Advanced Functional Materials. 23. 3990. 10.1002/adfm.201203792. 

\bibitem{BPSM2}
Serra, F., Vishnubhatla, K. C., Buscaglia, M., Cerbino, R., Osellame, R.,
Cerullo, G., \& Bellini, T. (2011) Soft Matter 2011, 7, 10945.

\bibitem{Wang}
Wang, Y., Canevari, G., \& Majumdar, A. (2018). Order reconstruction for nematics on squares with isotropic inclusions: A Landau-de Gennes study. Preprint arXiv 1803.02597.

\bibitem{Ziemer}
Ziemer, W. P. (1989). Weakly differentiable functions. Springer-Verlag New York. 


\end{thebibliography}
\end{document}